\documentclass{article}

\usepackage{geometry}

\usepackage{graphicx}
\usepackage{color}

\usepackage{amsmath}
\usepackage{nicefrac}
\usepackage{amsfonts}
\usepackage{amsthm}
\newtheorem{definition}{Definition}
\newtheorem{example}{Example}
\newtheorem{remark}{Remark}

\newtheorem{proposition}{Proposition}
\newtheorem{corollary}{Corollary}
\newtheorem{theorem}{Theorem}
\newtheorem{lemma}{Lemma}

\newcommand{\keywords}{{\bf Keywords.}\ }
\newcommand{\subclass}{{\bf MSC2010.}\ }
\allowdisplaybreaks

\newcommand{\CWENO}{\ensuremath{\mathsf{CWENO}}}
\newcommand{\CWENOZ}{\ensuremath{\mathsf{CWENOZ}}}
\newcommand{\CWENOZDB}{\ensuremath{\mathsf{CWENOZDB}}}
\newcommand{\WENO}{\ensuremath{\mathsf{WENO}}}
\newcommand{\WENOZ}{\ensuremath{\mathsf{WENOZ}}}
\newcommand{\ENO}{\ensuremath{\mathsf{ENO}}}

\renewcommand{\vec}[1]{\mathbf{#1}}
\newcommand{\R}{\mathbb{R}}
\newcommand{\RR}{\mathbb{R}}
\newcommand{\Ogrande}{\mathcal{O}}
\newcommand{\Opiccolo}{o}

\newcommand{\dx}{\mathrm{d}\vec{x}}
\newcommand{\ca}[1]{\overline{#1}}

\newcommand{\Poly}[1]{\mathbb{P}^{#1}}
\newcommand{\Prec}{P_{\text{\sf rec}}}

\newcommand{\Popt}{P_{\text{\sf opt}}}
\newcommand{\Sopt}{\mathcal{S}_{\text{\sf opt}}}
\newcommand{\Sk}{\mathcal{S}_{k}}

\newcommand{\elle}{\textit l}
\newcommand{\mi}[1]{\textcolor{black}{\boldsymbol{#1}}}
\newcommand{\aalpha}{\mi{\alpha}}
\newcommand{\abeta} {\mi{\beta}}
\newcommand{\agamma}{\mi{\gamma}}

\newcommand{\auno}  {\mi{1}}
\newcommand{\azero}{\mi{0}}

\newcommand{\up}{{u^{\prime}(0)}}
\newcommand{\us}{{u^{\prime\prime}(0)}}
\newcommand{\ut}{{u^{\prime\prime\prime}(0)}}
\newcommand{\uq}{{u^{IV}(0)}}
\newcommand{\uc}{u^{V}(0)}

\begin{document}

\title{Optimal definition of the nonlinear weights in multidimensional Central WENOZ reconstructions}

\author{I. Cravero 
\thanks{
  Dipartimento di Matematica -
  Universit\`a di Torino --
  Via C. Alberto, 8 - Torino (Italy) --
  {\sl isabella.cravero@unito.it}
}
\and 
M. Semplice 
\thanks{
  Dipartimento di Matematica -
  Universit\`a di Torino --
  Via C. Alberto, 8 - Torino (Italy) --
  {\sl matteo.semplice@unito.it}
}
\and
G. Visconti
\thanks{
  RWTH Aachen University -
  Templergraben 55, 52062 Aachen, Germany --
  {\sl visconti@igpm.rwth-aachen.de}
}
}

\date{}

\maketitle

\begin{abstract}
  Central \WENO\ reconstruction procedures have shown very good performances in finite volume and finite difference schemes for hyperbolic conservation and balance laws in one and more space dimensions, on different types of meshes. Their most recent formulations include \WENOZ-type nonlinear weights, but in this context a thorough analysis of the definition of the global smoothness indicator $\tau$ is still lacking.
  In this work we first prove results on the asymptotic expansion of multi-dimensional Jiang-Shu smoothness indicators that are useful for the rigorous design of \CWENOZ\ schemes also beyond those considered in this paper. Next, we introduce the optimal definition of $\tau$ for the one-dimensional \CWENOZ\ schemes and for one example of two-dimensional \CWENOZ\ reconstruction. 
  Numerical experiments of one and two dimensional test problems show the good performance of the new schemes.
  
 
\subclass{65M08 \and 65M20}
\end{abstract}

\section{Introduction}\label{sec:intro}

In this paper we focus on point-value reconstruction from cell averages employed in multidimensional finite volume schemes for the numerical approximation of the solutions of systems of hyperbolic balance laws of the form
\begin{equation} \label{eq:balanceLaw}
\partial_t \vec{u} + \nabla_{\vec{x}} \cdot \vec{f}(\vec{u}) = \vec{s}(\vec{u}).
\end{equation}
Here $\vec{u} = \vec{u}(t,\vec{x}):\R^+\times\R^n\to\R^J$ is the unknown describing the physical states, $n$ is the number of space dimensions and $J$ is the number of equations. The vector valued function $\vec{f}(\vec{u})$ is the  flux and is a smooth known function of $\vec{u}$, whose Jacobian is assumed diagonalizable with real eigenvalues along all possible directions in $\R^n$ in order to guarantee that \eqref{eq:balanceLaw} is a hyperbolic system. Finally, the vector valued function $\vec{s}(\vec{u})$ is the source term. We assume that equation \eqref{eq:balanceLaw} is set in a bounded domain $ {\cal{D}} \subset\R^n$ and is complemented by appropriate initial conditions and boundary conditions on $\partial\mathcal{D}$.

Finite volume methods partition the domain $\mathcal{D}$ into cells 
$\Omega_j \subset \R^n$ such that 
$ {\cal{D}} = \cup_j \Omega_j$, where $\Omega_i\cap\Omega_j\subset\partial\Omega_i$ for any $i\neq j$. 
The time evolution of the solution is computed by solving the system of ordinary differential equations (ODEs)
\begin{equation} \label{eq:TheSemidis}
\frac{\mathrm{d}\ca{\vec{u}}_j}{\mathrm{d}t} = -\frac{1}{|\Omega_j|}\int_{\partial\Omega_j} \vec{f}(\vec{u}(t,\gamma))\cdot \vec{n}(\gamma) \mathrm{d}\gamma +  \frac{1}{|\Omega_j|}\int_{\Omega_j} \vec{s}(\vec{u}(t,\vec{x})) \dx
\end{equation}
for the cell averages
\begin{equation} \label{eq:cellaverage}
\ca{\vec{u}}_j (t) = \frac 1 {\vert \Omega_j \vert} \int_{\Omega_j} \vec{u}(t,\vec{x}) \dx,
\end{equation} 
where $\partial\Omega_j$ is the cell boundary, $\vert \Omega_j \vert$ is the volume of $\Omega_j$, and $\vec{n}$ is the outward unit normal to $\partial\Omega_j$. 
In order to compute the right-hand side of \eqref{eq:TheSemidis}, one needs point values of $\vec{u}$ at the cell boundary (and inside the cell, if $\vec{s}\neq\vec{0}$). Since only the cell averages are stored, such point values must be approximated  by a so-called reconstruction procedure, which computes  a function $\vec{R}_j(\vec{x})$ that ``reconstructs'' the unknown function  $\vec{u}(\vec{x})$ in the cell $\Omega_j$ with a chosen accuracy  $G +1 $.
In particular, such reconstructions are needed at the nodes of the quadrature rules used to compute the integrals in \eqref{eq:TheSemidis}. For example, a third order scheme in two space dimensions would require at least 8 nodes for the boundary term and 4 extra inner nodes for the source.

It is well known that defining $\vec{R}_j(\vec{x})$ as a polynomial interpolant on a fixed stencil, yields oscillatory results in any high order scheme.
The most popular reconstruction procedure for accuracy greater than two is surely the Weighted Essentially Non Oscillatory reconstruction (\WENO) \cite{JiangShu:96,Shu:2009:WENOreview}, which reconstructs point values as a convex combination of point values of low-degree interpolants. The coefficients of the combination, the so-called \emph{nonlinear weights}, depend nonlinearly on the data and are designed to reproduce, at a specific point, the value of a high-degree central interpolant in smooth areas and to provide a lower accuracy but non-oscillatory reconstruction elsewhere.
The whole construction depends on the existence and positivity of a set of, so-called, {\em linear} or {\em optimal weights}.

The  \WENO\ linear weights are thus fixed by accuracy requirements and depend on the location of the reconstruction point and on the size and relative location of the neighbouring cells. As a consequence, the computation of nonlinear weigths must be repeated for each reconstruction point. Moreover, even in one space dimension, the existence (and positivity) of the \WENO\ linear weights is not guaranteed for a generic point location \cite{QiuShu:02}.

Attempts to extend the definition of \WENO\ to non-Cartesian meshes had to tackle the formidable task of computing the linear weights for very general cell arrangements and most authors choose either to stick to the original philosophy and use only low order polynomials but very complex computations of the optimal linear weights for each reconstruction point 
\cite{HuShu99:WENO3d,ShiHuShu:2002,ZhangShu09:WENO3d}
or to employ a central high order polynomial with low order directionally-biased ones
\cite{BRDM:divfreeWENO,DK07:nonlinear,TTD11:WENOmixed3D}.

We point out that the older \ENO\ approach \cite{HartenEtAl:ENO}, which simply selects a reconstruction polynomial among a set of candidates, does not depend on point- nor accuracy-dependent quantities and additionally provides a polynomial that is defined and uniformly accurate in the whole cell. However the stencil of an \ENO\ reconstruction is much wider that the stencil of a \WENO\ one of the same order.

The Central \WENO\ (\CWENO) reconstruction, first introduced by Levy, Puppo and Russo in the one-dimensional context  \cite{LPR:99}, enjoys the benefits of both \ENO\ and \WENO, but none of their drawbacks. In fact, \CWENO\ makes use of linear weights that can be fixed independently of the reconstruction point, since they need not satisfy accuracy requirements. As a consequence, \CWENO, unlike \WENO, does not suffer from the existence and positivity issues of the weights and moreover, like \ENO, it yields a reconstruction polynomial that is valid in the entire cell. The nonlinear weights are then computed by a non-linear procedure that is very similar to the \WENO\ one, but needs to be applied only once per cell and not once per reconstruction point. This is particularly advantageous for very high order schemes, given the high number of flux quadrature points, and even more for balance laws, due to the additional source term quadrature (see e.g. \cite{BRS:stiffsource}).

After the first paper, the one-dimensional \CWENO\ technique was extended to fifth order \cite{Capdeville:08}, the properties of the third order versions were studied in detail on uniform meshes  \cite{Kolb:14} and non-uniform ones  \cite{CS:epsweno}. Finally arbitrary high order variants were introduced \cite{CPSV:cweno}, where the class of \CWENO\ reconstructions was defined. The Adaptive Order \WENO\ of \cite{BGS:wao} also belong to this class.

To extend 
the \CWENO\ procedure  to more than one space dimension, \cite{LPR:00:SIAMJSciComp} combined a bivariate central parabola and four linear polynomials. In a similar fashion, \cite{SCR:CWENOquadtree} obtained a third-order accurate reconstruction on two-dimensional quad-tree meshes.
Other approaches define the reconstruction as combination of central polynomials defined in each neighbour \cite{LPR:00:ANM,LP:12} or combination of polynomials, each of which can be of degree one or two\cite{gallardo2011}.
More recently,  
in the finite volume schemes as a seed for an {\sf ADER} predictor, \cite{DBSR:ADER_CWENO} combined high and low degree polynomials  to obtain an arbitrary high order \CWENO\ construction for triangular and tetrahedral meshes. Later, this idea was also exploited  in the subcell limiter for a discontinuous Galerkin scheme, see \cite{DBSR:DG_CWENO}.

An alternative definition of the nonlinear weights has been given in \cite{BCCD:2008:wenoz5} and further developed in \cite{CCD:11,DB:2013}, obtaining the \WENOZ\ reconstruction procedure. The nonlinear weights of \WENOZ\ are based on an additional global smoothness indicator $\tau$ which is a linear combination of the \WENO\ indicators. A similar definition of the nonlinear weights was employed in the \CWENO\ context as well, obtaining the so-called \CWENOZ\ schemes. Although sometimes under different names, \cite{CPSV:coolweno,Zahran09,ZhuQiu:FVCWENO,ZhuQiu18:triFV} present schemes of this class.
\cite{CPSV:coolweno} compares the \WENOZ\ and \CWENOZ\ reconstructions, with the \WENO\ and \CWENO\ ones, finding that they have better spectral properties.
Note that a further technique to compute the nonlinear weights was introduced in \cite{ABC16:improvedWENOZ}, which still relies on the \WENOZ\ idea and introduces an additional indicator.

In this paper we focus on the generalization of the \CWENOZ\ reconstruction of \cite{CPSV:coolweno} to higher space dimensions, without relying on dimensional splitting, in the same spirit as \cite{SCR:CWENOquadtree}.
A very important contribution of this paper is a set of theoretical results that give sufficient conditions for a \CWENOZ\ reconstruction to be of optimal order of accuracy. This is discussed and enriched by many examples in \S\ref{sec:indicators}, where some useful results on the Taylor expansion of Jiang-Shu indicators in one or more space dimensions are also proven.

In \S\ref{sec:glob_ind}, we build on the previous results to derive %
new and optimal definition of the $\tau$ parameter of the one- and two-dimensional \CWENOZ\ reconstructions tested in this paper. \S\ref{sec:numerics} contains the numerical tests for the accuracy of the reconstructions and on their application to semidiscrete finite-volume schemes for  systems of conservation laws in one and two space dimensions. Finally, \S\ref{sec:concl} summarizes the findings of this paper and Appendix \ref{sec:appendix} contains the explicit expression of the quantities involved in the results of \S\ref{sec:indicators} in some well-known cases.

\section{Analysis of $\boldsymbol{n}$-dimensional $\CWENO$ and $\CWENOZ$ reconstructions}
\label{sec:indicators}


In this section we prove the theoretical results to analyse the \CWENO\ and \CWENOZ\ reconstruction procedures in one and more space dimensions. To this end, we restrict here to the scalar case, since usually, in the case of systems of conservation laws, the reconstruction procedures are applied component-wise, directly to the conserved variables or after the local characteristic projection.

We consider a Cartesian grid in $n$ space dimensions, composed by a union of rectangular cells $\Omega_k \subset \R^n$ of dimension $\boldsymbol{\Delta} \vec{x} $ and diameter $\rho:=\|\boldsymbol{\Delta} \vec{x}\|_2$. 

\subsection{Definition and examples} \label{ssec:cwdef}
We recall here the definition of the $\CWENO$ and $\CWENOZ$ operators, as given in \cite{CPSV:cweno} and \cite{CPSV:coolweno}, respectively. Note that the difference of the two methods is in the computation of the nonlinear coefficients: the $\CWENOZ $ method uses the idea of Borges, Carmona, Costa and Don in \cite{BCCD:2008:wenoz5}, where they introduced an extra smoothness indicator $\tau$ and a new definition of the nonlinear coefficients, that drives them closer to their optimal values in the smooth case.

Let $\Poly{k}_{n}$ be the space of polynomials in $n$ variables with degree at most $k\in\mathbb{N}$.

In order to describe the reconstruction, we consider as given data the  cell averages $\ca{u}_k$ of a function ${u}$ over the cells of a grid.
To simplify the notation, we describe the reconstruction of ${u}(\vec{x})$ in the cell $\Omega_0$ centred in the point  $\vec{x}_0= \vec{0}. $

\begin{definition}
Let $\mathcal{S}$ be set of $\eta$ cell indices including $0$ (stencil). We associate to $\cal{S}$ the polynomial
 \begin{equation} \label{eq:ps}
 {P}^{(d)}_{\mathcal{S}}(\vec{x})= \arg \min \left \{ \sum_{i \in \mathcal{S}} \vert \langle P^{(d)}_{\mathcal{S}} \rangle_{\Omega_i} - \ca{u}_i \vert^2, \, 
 \text{such that }  
  P^{(d)}_{\mathcal{S}}\in\Poly{d}_n,\,
 \langle  P^{(d)}_{\mathcal{S}} \rangle_{\Omega_0}= \ca{u}_0 \right \},
 \end{equation}
where the operator $\langle\cdot\rangle_{\Omega_i}$ denotes
 the cell average of its argument over the cell $\Omega_i$. Of course $\eta$ should be larger than the number of coefficients in $P^{(d)}_{\mathcal{S}}$.
\end{definition}

\begin{remark}
If the number of coefficients in $ P^{(d)}_{\mathcal{S}}(\vec{x}) $ is equal to $\eta$, then the constrained least square polynomial $ P^{(d)}_{\mathcal{S}}(\vec{x}) $ is the ordinary polynomial interpolating the given cell averages of the stencil $\cal{S}$ exactly.
We observe, as in \cite{SCR:CWENOquadtree}, that if
\eqref{eq:ps} is overdetermined, then it suffices to choose a basis such that
$\langle {\varphi}_k \rangle_{\Omega_0} =0$
and express
$$P^{(d)}_{\mathcal{S}}(\vec{x}) = \ca{u}_0 + \sum_k c_k {\varphi}_k(\vec{x}) $$
in order to turn \eqref{eq:ps} into an unconstrained least square problem for the coefficients $c_k$ with right hand sides $\ca{u}_k-\ca{u}_0$.
\end{remark}

The nonlinear selection or blending of polynomials taking place in any Essentially Non-Oscillatory reconstruction relies on so called {\em oscillation indicators}.
These are in general scalar quantities $I[P]$ associated to a polynomial $P$ that are designed in such a way that $I[P]\to0$ under grid refinement if $P$ is associated to smooth data. Moreover, $I[P]$ is in all cases a bounded quantity, even if a discontinuity is present in the stencil of $P$; in this later case it is desirale that $I[P]\asymp1$. The Jiang-Shu indicators defined in \cite{JiangShu:96} are the most widely used (see Definition~\ref{def:ind}).

We are now in position to define the \CWENO\ and \CWENOZ\ reconstructions.

\begin{definition} \label{def:CWENO}
Given a stencil $\Sopt$ that includes the cell $\Omega_0$, let $\Popt\in\Poly{G}_n$ (\emph{optimal polynomial}) be the polynomial of degree $G$ associated to $\Sopt$. Further, let $ {P}_1, {P}_2, \ldots, {P}_m$ be a set of $m\geq 1$ polynomials  of degree $g$ with $g<G$ associated to substencils such that $0\in\Sk\subset\Sopt$.
Let also $\{d_k\}_{k=0}^m$ be a set of strictly positive real coefficients such that $\sum_{k=0}^m d_k=1$.

The $\CWENO$ and the $\CWENOZ$ operators compute a reconstruction polynomial
\[ 
\begin{aligned} 
\Prec^{\text{CW}}  & 
= \CWENO ({\Popt},{P}_1,\ldots,{P}_m) \in \Poly{G}_{n} \\
\Prec^{\text{CWZ}} & 
= \CWENOZ ({\Popt},{P}_1,\ldots,{P}_m) \in \Poly{G}_{n} 
\end{aligned} 
\]
as follows:
\begin{enumerate}
\item first, introduce the polynomial ${P}_0$ defined as
\begin{equation} \label{eq:p0}
{P}_0(\vec{x}) = \frac{1}{d_0}\left({\Popt}(\vec{x})-\sum_{k=1}^{m}d_k {P}_k(\vec{x}) \right) \in \Poly{G}_{n};
\end{equation}
\item compute
\[
I_0=I[\Popt], 
\qquad
I_k=I[P_k], k\geq1
\]
where $I[P]$ is a suitable regularity indicator
\item compute  the nonlinear coefficients $\{\omega_k\}_{k=0}^m$ or $\{\omega^Z_k\}_{k=0}^m$ as
\begin{subequations}\label{eq:Omega:CWandCWZ}
\begin{itemize}
\item[2.i)] $\CWENO$ operator:   for $k=0,\dots,m$ 
\begin{equation} \label{eq:Omega}
\alpha_k = \frac{d_k}{(I_k+\epsilon)^{\elle}},
\qquad
\omega^{\text{\sf CW}}_k = \frac{\alpha_k}{\sum_{i=0}^{m}\alpha_i},
\end{equation}
\item[2.ii)] $\CWENOZ$ operator: for $k=0,\dots,m$
\begin{equation} \label{eq:OmegaZ}
\alpha^Z_k = {d_k} \left( 1 + \left( \frac {\tau} {I_k+\epsilon} \right)^{\elle} \right),
\qquad
\omega^{\text{\sf CWZ}}_k = \frac{\alpha^Z_k}{\sum_{i=0}^{m}\alpha^Z_i},
\end{equation}
\end{itemize}
\end{subequations}
where $\epsilon$ is a small positive quantity, 
$\elle \ge 1$
and, in the case of \CWENOZ, $\tau$ is a global smoothness indicator
\item and finally define the reconstruction polynomial as
\begin{subequations} \label{eq:prec}
\begin{align}
{\Prec^{\text{CW}}}(\vec{x}) 
&= \sum\limits_{k=0}^{m} \omega^{\text{\sf CW}}_k 
{P}_k(\vec{x}) \in\Poly{G}_{n}, \label{eq:precCW}
\\
\Prec^{\text{CWZ}}(\vec{x}) 
&= \sum\limits_{k=0}^{m} \omega^{\text{\sf CWZ}}_k {P}_k(\vec{x}) \in\Poly{G}_{n}. \label{eq:precCWZ}
\end{align}
\end{subequations}
\end{enumerate}
\end{definition}

This extends the definitions given in \cite{CPSV:cweno,CPSV:coolweno}. We point out that the use of equation \eqref{eq:p0} is what characterizes a Central WENO reconstruction and that both WENO-ZQ \cite{ZhuQiu:FVCWENO} and WAO \cite{BGS:wao} in fact belong to this class.

Note that the reconstruction polynomial defined in \eqref{eq:prec} can be evaluated at any reconstruction point in the computational cell at a very low computational cost, since the coefficients of $\Prec$ can be computed with \eqref{eq:prec} in any convenient basis for $\Poly{G}_n$. It is important to note that the linear coefficients $\{d_k\}_{k=0}^m$ do not depend on the reconstruction point and thus the nonlinear coefficients \eqref{eq:Omega} and \eqref{eq:OmegaZ} can be computed once per cell and not once per reconstruction point, as in the standard \WENO. This makes the \CWENO\ idea less computationally expensive for balance laws, multidimensional computations and unstructured meshes than the \WENO\ procedure.

\begin{remark}
  Thanks to the constraint included in \eqref{eq:ps}, if all the interpolating polynomials involved in Definition \ref{def:CWENO} are defined by \eqref{eq:ps}, they all satisfy the conservation property on the reconstruction cell $\Omega_0$. Then also the cell average on $\Omega_0$ of $P_0$, $\Prec^{\text{CW}}$ and $\Prec^{\text{CWZ}}$ is exactly $\ca{u}_0$ and the reconstruction is conservative.
\end{remark}

The accuracy and non-oscillatory properties of \CWENO\ and \CWENOZ\ schemes are guaranteed by the dependence of their nonlinear weights \eqref{eq:Omega:CWandCWZ} on suitable regularity indicators $I_k$.
On smooth data, \eqref{eq:Omega:CWandCWZ} ought to drive the nonlinear weights sufficiently close to the optimal ones, so that ${\Prec} \approx {\Popt}$ and the reconstruction should reach the optimal order of accuracy $G+1$.

\begin{remark} \label{rem:weights}
Assume that the stencils are chosen such that the approximation orders of $\Popt \in \Poly{G}_{n} $ and $P_k \in \Poly{g}_{n}$, for $k=1,\ldots,m$, are
 $$ \begin{aligned}
    \vert \Popt (\vec{x}) - u(\vec{x}) \vert 
    = \Ogrande (\rho^{G+1}) 
    \quad \text{and} \quad
	\vert P_k (\vec{x}) - u(\vec{x}) \vert 
    = \Ogrande (\rho^{g+1}) 
	\end{aligned} 
 $$ 
 at any point $\vec{x}$ in the computational cell, if the function $u(\vec{x})$ is sufficiently regular.
 Then, using \eqref{eq:prec} and since $\Popt=\sum_{k=0}^{m}d_kP_k$,
 the reconstruction error at $\vec{x}$ is
	\begin{equation*}
	u(\vec{x})- \Prec(\vec{x}) 
    =\underbrace{(u(\vec{x})-\Popt(\vec{x}))}_{\Ogrande(\rho^{G+1})}
	+ \sum_{k=0}^{m} 
	(d_k-\omega_k)
	\underbrace{(P_k(\vec{x})-u(\vec{x}))}_{\Ogrande(\rho^{g+1})} 
    ,
	\end{equation*}
    where $\omega_k$ and $\Prec$ are defined either by \eqref{eq:Omega} and \eqref{eq:precCW} or by \eqref{eq:OmegaZ} and \eqref{eq:precCWZ}.
    Thus the condition $(d_k-\omega_k)=\Ogrande(\rho^{G-g})$ is sufficient to ensure that the accuracy of the \CWENO\ and \CWENOZ\ reconstruction equals the accuracy of its first argument $\Popt$ in the case of smooth data.
\end{remark}

On the other hand, if there is an oscillating ${P}_{\hat{k}}$ for some $\hat{k}\in\{1,\dots,m\}$, then $I_{\hat{k}}\asymp1$ and $w_{\hat{k}} \approx 0$; moreover also $w_0 \approx 0$ (see \cite{CPSV:cweno}) and
$\Prec$ is a nonlinear combination of polynomials of degree $g$: the accuracy of the reconstruction reduces to $g+1$, but spurious oscillations in the PDE solution can be  controlled.
\begin{remark} \label{rem:I0Iopt}
Note that in Definition \ref{def:CWENO}, the regularity indicator $I_0$ is computed by using the optimal polynomial $\Popt$. In previous works, see, e.g., \cite{CPSV:coolweno,CPSV:cweno,CS:epsweno,SCR:CWENOquadtree}, instead $I_0$ was defined as $I_0=I[P_0]$. 
\end{remark}

The positive parameter $\epsilon$ prevents the division by zero in the computation of the nonlinear weights \eqref{eq:Omega} and \eqref{eq:OmegaZ}. One would like to choose it as small as possible, in order to control spurious oscillations. However, in \cite{CS:epsweno,Kolb:14}, the authors proved that the choice of $\epsilon$ can influence the convergence of the method on smooth parts of the solution. In \cite{DB:2013} it was proven that for \WENOZ\ schemes the condition on $\epsilon$ to achieve the optimal order of accuracy is weaker than in standard \WENO. In this paper, we also study the lower bound for $\epsilon$ in \CWENOZ (see \S\ref{sec:glob_ind} and Table~\ref{tab:epsilon}).

In view of Remark ~\ref{rem:weights}, we point out that the weights of the \CWENOZ\ reconstruction defined in \eqref{eq:OmegaZ}, when compared to equation \eqref{eq:Omega} for \CWENO, are designed to drive the nonlinear coefficients close to the optimal ones in the case of smooth data: this goal is achieved, as in standard \WENOZ\ schemes \cite{BCCD:2008:wenoz5,CCD:11,DB:2013},  by including the additional regularity indicator $\tau$ in the computation of the nonlinear weights.

This paper addresses the optimal definition of the extra regularity indicator $\tau$ for the \CWENOZ\ scheme, thus generalizing the results of \cite{DB:2013} to the case of Central \WENO\ reconstructions and improving the \CWENOZ\ reconstructions of \cite{CPSV:coolweno}.

In Definition \ref{def:CWENO} the number $m$ and the degree $g$ of the lower-degree polynomials is not specified nor linked to the degree $G$ of the optimal polynomial. Here we give some hints on traditional choices that have been put forward in the literature.

\begin{example}[1D of accuracy $2r-1$]
\label{ex:CW1d}
In one space dimension it is customary to choose $g=r-1$, $m=g+1$ and $\Popt$ is of degree $G=2r-2$. This latter is determined by the exact interpolation of the data in a symmetric stencil $\cal{S}_{\text{opt}}$ centered on $\Omega_0$ containing the cells $\Omega_{-g},\dots,\Omega_{g}$. 
Furthermore, for $k=1, \cdots, m$, the lower-degree polynomials ${P}_k$ are defined as the exact interpolants on the substencils
${\cal S}_k=\{k-r,\ldots,k-1 \} \subset {\cal S}_{\text{opt}}$. This is the same choice considered in \cite{CPSV:cweno,CPSV:coolweno}.
\end{example}

\begin{example}[2D of accuracy $3$ on quad-tree grids]
\label{ex:CW2d}
In more space dimensions,
\cite{SCR:CWENOquadtree} considers a \CWENO\ reconstruction of order $3$ in which $\mathcal{S}_{\text{opt}}$
is composed by all cells that intersect $\Omega_0$ in a face/edge or vertex in a quad-tree mesh, while $m=4$ and ${\cal S}_k $ are substencils in North-East, North-West, South-East and South-West directions. The optimal polynomial is of degree $G=2$ and is chosen as $\Popt=P^{(2)}_{\cal{S}_{\text{opt}}}$
and the four degree $g=1$ polynomials are chosen as $P_k= P^{(1)}_{{\cal S}_k}$. In \cite{SCR:CWENOquadtree} is proven that the aforementioned definitions of the stencils always lead to overdetermined systems in the least squares problem \eqref{eq:ps} on any possible configuration in a quad-tree mesh. In the same paper, a generalization to octrees in 3 space dimensions is also suggested.
\end{example}

\begin{example}[2D and 3D of arbitrary accuracy on simplicial meshes]
\label{ex:cwader}
\CWENO\ reconstructions of arbitrary order on simplicial meshes in two and three space dimensions were introduced in \cite{DBSR:ADER_CWENO}. There, a polynomial of degree $G \geq 2$ is combined with $m=3$ (in 2D) or $m=4$ (in 3D) polynomials of degree $g=1$ in order to enhance the non-oscillatory properties of the schemes.
At order 3 and 4, also \cite{ZhuQiu18:triFV} presents a similar reconstruction.
\end{example}

\begin{example}[2D of order 4 on cartesian meshes]
\label{ex:cwandaluz}
\CWENO\ reconstructions of order 4 on uniform cartesian meshes in two space dimensions were introduced in \cite{CWENOandaluz}. There, a polynomial of degree $G=3$ defined by a diamond-shaped central stencil is combined with $m==4$ polynomials of degree $g=1$ or $g=2$.
\end{example}

\subsection{Properties of the Jiang-Shu smoothness indicators}
\label{ssec:JiangShuAnalysis}

We now turn to the study of the accuracy of the  \CWENOZ\ reconstructions, but we first need to prove some general properties of the smoothness indicators.
The results of this section generalize the ones of \cite{DB:2013} for $n>1$ and of course specialize to them in the case $n=1$.

In this section we use the multi-index notation for partial derivatives. In $n$ space dimensions, 
for a smooth enough function $q$, we denote, for $\aalpha = (\alpha_1,\ldots, \alpha_n)\in\mathbb{N}^n$,
$$ \partial_{\aalpha} q :=  \frac {\partial^{\vert \aalpha \vert} q}
{\partial x_1^{\alpha_1} \ldots \partial x_n^{\alpha_n}}.
$$
Obviously, for $n=1$, $ \partial_{\aalpha} q$ denotes the ordinary derivative.

In this paper we employ the multi-dimensional generalization of the classical smoothness indicators defined in \cite{JiangShu:96}, which was already employed since \cite{HuShu:WENOtri}.

\begin{definition}\label{def:ind}
The smoothness indicator of a polynomial $q\in\Poly{M}_n$ is
\begin{equation} \label{eq:ind}
I[q] := 
\sum_{\vert \abeta \vert =1}^M  \boldsymbol{\Delta} \vec{x}^{2 \abeta -\auno}  \int_{\Omega_0} (\partial_{\abeta} q (\vec{x}))^2 \dx . \end{equation}
\end{definition}

Here $\vec{x}\in\R^n$ denotes
$\vec{x}^{\aalpha} := x_1^{\alpha_1} \cdots x_n^{\alpha_n}$.
In the sequel we will also use the following notations. 
For each $k \in \mathbb{N}$ we denote $\mi{k} := (k,k,\dots,k)$. For any multi-index $\aalpha\in\mathbb{N}^n$, we define $\aalpha ! := \alpha_1 ! \cdots \alpha_n !, \;
\vert \aalpha \vert := \sum_{i=1}^{n}\alpha_i$ and $\Pi \aalpha := \prod_{i=1}^{n}\alpha_i.$ 
We say that 
$\aalpha$ is even if $\alpha_i$ is even for each $i=1,\ldots,n$. Finally, we define a partial ordering among multi-indices by
$\aalpha \le \abeta  $ if $ \alpha_i \le \beta_i $ for all $i=1,\ldots,n$.

In the sequel, we will use the notation $\theta(g(\rho))=r$ to mean that  $g(\rho)= a_r \rho^r+o(\rho^r)$ for $ \rho \to 0$, with $ a_r \neq 0$.

Let $q(\vec{x}) = \sum_{\vert \aalpha \vert \le M} a_{\aalpha}
{\vec{x}}^{{\aalpha}}$  be a polynomial of degree $M$ and let $\vec{a} = \{ a_{\aalpha} \}_{\vert \aalpha \vert \le M}$ be the vector of coefficients indexed by $ \aalpha.$  Let us focus again on the cell $\Omega_0$ centered in the point $\vec{0}$, whose sides are $\boldsymbol{\Delta} \vec{x},$  with diameter $\rho$.

\begin{proposition} \label{prop:exA} 
Let $q$ be a polynomial in $\Poly{M}_n$ and let
$\vec{w}(q)$ be vector of size 
$\mathsf{dim}(\Poly{M}_n)=\binom{M+n}{n}$
whose components are
\begin{equation} \label{eq:w}
(\vec{w}(q))_{{\aalpha}} = \boldsymbol{\Delta} \vec{x}^{{\aalpha}} \int_{\Omega_0} \partial_{\aalpha} q(\vec{x}) \dx, \qquad  \forall \aalpha \text{  s.t. } \vert \aalpha \vert =0, 1, \dots, M.
\end{equation}
Then, there exists a square symmetric matrix  $\vec{A}$ with  constant entries such that
$$ \boldsymbol{\Delta} \vec{x}^{\auno}  \int_{\Omega_0} (q(\vec{x}))^2 \dx = \langle {\vec{w}(q), \vec{A} \vec{w}(q)} \rangle,
$$
\end{proposition}
\begin{proof}
Direct computation shows that 
$$ (\vec{w}(q))_{\aalpha} 
= \boldsymbol{\Delta} \vec{x}^{{\aalpha}} \int_{\Omega_0} \partial_{\aalpha} q(\vec{x}) \dx 
=  \boldsymbol{\Delta} \vec{x}^{{\aalpha}} 
\sum_{\substack{ \vert \abeta \vert \le M, \\  \abeta \ge \aalpha }} a_{{\abeta }}  \int_{\Omega_0} \partial_{\aalpha} {\vec{x}}^{{\abeta}}  \dx.
$$
The integral terms are nonzero only if $ \abeta-\aalpha $ is even and they are equal to
$$   \int_{\Omega_0} \partial_{\aalpha} {\vec{x}}^{{\abeta}}  \dx = \frac { \abeta !} {(\abeta - \aalpha )!}  \int_{\Omega_0}  {\vec{x}}^{{\abeta} - \aalpha}  \dx  =   \frac {   \abeta ! \, \boldsymbol{\Delta} \vec{x}^{{\abeta - \aalpha + \auno}} }  {2^{\vert \abeta - \aalpha \vert } (\abeta - \aalpha + \auno)!}.$$
Thus
$$ 
(\vec{w}(q))_{{\aalpha}} 
=  \sum_{\substack{ \vert \abeta \vert \le M, \abeta \ge \aalpha, \\ \abeta - \aalpha \; \text{even}}} a_{\abeta } \; \frac { \abeta !}  {(\abeta - \aalpha + \auno)! 
\; 2^{\vert \abeta - \aalpha \vert }} \; \boldsymbol{\Delta} \vec{x}^{{\abeta}+ \auno} $$
or, in a matrix form,
$$ \vec{w}(q) = \vec{U} \, \vec{D} \, \vec{a} $$
where $ \vec{D}$ is a diagonal matrix $\vec{D}_{\aalpha, \aalpha} = \boldsymbol{\Delta} \vec{x}^{{\aalpha} + \auno}$
and $\vec{U}$  is an upper triangular constant matrix  whose elements $ \vec{U}_{\aalpha, \abeta} $ are given by 
\begin{equation} \vec{U}_{\aalpha, \abeta} = \begin{cases} \frac { \displaystyle{\abeta !}}  {\displaystyle{(\abeta - \aalpha + \auno)! 
\; 2^{\vert \abeta - \aalpha \vert } }} & \text{if} \; \abeta - \aalpha \; \text{is even,} \; 
\abeta \ge \aalpha
 \\ 0
& \text{otherwise.} \end{cases}
\label{eq:matrixU} 
\end{equation}
We observe that $\vec{U}_{\aalpha,\aalpha} = \aalpha ! \neq 0,$ so $\vec{U}$ is invertible and
$  \vec{D} \, \vec{a}  = \vec{U}^{-1} \, \vec{w}.$

We obtain
$$ \begin{aligned}  
& \boldsymbol{\Delta} \vec{x}^{\auno} \int_{\Omega_0} (q(\vec{x}))^2 \dx  =  \boldsymbol{\Delta} \vec{x}^{\auno}  \sum_{\vert \aalpha \vert \le M} \sum_{\substack{\vert \abeta \vert \le M \\}} a_{\aalpha} a_{\abeta}  \int_{\Omega_0} \vec{x}^{\aalpha + \abeta } \dx  \\
& = \sum_{\vert \aalpha \vert \le M} \sum_{\substack{ \vert \abeta \vert \le M, \\ \aalpha + \abeta \; \text{even}}} \frac  1 { 2^{\vert \aalpha + \abeta \vert } \Pi (\aalpha + \abeta + \auno)} a_{\aalpha} a_{\abeta} 
\boldsymbol{\Delta} \vec{x}^{{\aalpha + \abeta + \underline{2}}}   \\
& = \sum_{\vert \aalpha \vert \le M} a_{\aalpha} \boldsymbol{\Delta} \vec{x}^{{\aalpha + \auno}} \sum_{\substack{\vert \abeta \vert \le M, \\  \; \aalpha + \abeta \; \text{even}}} \frac 1 { 2^{\vert \aalpha + \abeta \vert } \Pi (\aalpha + \abeta + \auno)}  a_{\abeta} \boldsymbol{\Delta} \vec{x}^{{\abeta + \auno}}  \\
& = \langle \vec{D} \, \vec{a}, \vec{B} \, \vec{D} \,
\vec{a} \rangle = \langle  \vec{w}(q), \vec{A} \,
\vec{w}(q) \rangle 
\end{aligned}  $$
where $ \vec{B}$ is a symmetric constant matrix whose entries are
\begin{equation} \vec{B}_{\aalpha, \abeta} = \begin{cases} \frac { \displaystyle{1}}  {\displaystyle{\Pi (\aalpha + \abeta + \auno) 
\; 2^{\vert \aalpha + \abeta \vert } }} & \text{if} \; \abeta + \aalpha \; \text{is even} \\ 0
& \text{otherwise} \end{cases} \label{eq:matrixB} \end{equation}
and $ \vec{A}= (\vec{U}^{-1})^{T} \, \vec{B} \, \vec{U}^{-1}.$
\end{proof}

Note that $\vec{A}$ can be viewed as a $ (M+1) \times (M+1)$ block matrix by grouping its entries according to
$ \vert \aalpha \vert$. This generalizes the 1D case, where all blocks are $ 1 \times 1.$

\begin{proposition} The smoothness indicator $ I[q]$ of Definition \ref{def:ind} is a bilinear form
$$   
I[q]   =    \langle    \vec{v},  \vec{C} \, \vec{v} \rangle
 $$
 where $\vec{C}$ is a constant, symmetric and semi-positive matrix and $\vec{v}(q) = \vec{w}(q) / {\boldsymbol{\Delta} \vec{x}^{\auno}},$ with $\vec{w}(q)$ defined by \eqref{eq:w}.
 \label{prop:ind}
\end{proposition}
\begin{proof}
From Proposition \ref{prop:exA}, we have that
\begin{equation}  \boldsymbol{\Delta} \vec{x}^{\auno} \int_{\Omega_0} (\partial_{\abeta} q (\vec{x}))^2 \dx = \langle \vec{w}(\partial_{\abeta} q), \vec{A} \, \vec{w}(\partial_{\abeta} q) \rangle.  \label{eq:intdersquare} \end{equation}
For any $ \abeta$, we consider the shift operator $\vec{Q}^{\abeta}$ such that $ \vec{Q}^{\abeta} \vec{w}(q)= \boldsymbol{\Delta} \vec{x}^{\abeta} \vec{w}(\partial_{\abeta} q).$
We note that the $\aalpha$ component of $ \vec{Q}^{\abeta} \vec{w}(q) $ is given by
$$ 
\begin{aligned}
\boldsymbol{\Delta} \vec{x}^{\abeta} (\vec{w}(\partial_{\abeta} q))_{\aalpha} =
&\boldsymbol{\Delta} \vec{x}^{\aalpha + \abeta} \int_{\Omega_0} \partial_{\aalpha}(\partial_{\abeta} \; q(\vec{x})) \; \dx 
\\
=
&\boldsymbol{\Delta} \vec{x}^{\aalpha + \abeta} \int_{\Omega_0} \partial_{\aalpha+ \abeta}\; q(\vec{x}) \; \dx =
(\vec{w}( q))_{\aalpha + \abeta}, 
\end{aligned}
$$
 and then the entries of  $\vec{Q}^{\abeta}$ are
\begin{equation} \vec{Q}^{\abeta}_{\aalpha, \agamma} = \left \{ \begin{aligned}  & 1 \quad \text{if} \; \agamma= \aalpha + \abeta
 \\ & 0
\quad \text{otherwise}. \end{aligned} \right. \label{eq:matrixQ} \end{equation}
From \eqref{eq:intdersquare} we have
\begin{equation*}
\boldsymbol{\Delta} \vec{x}^{ 2 \abeta + \auno} \int_{\Omega_0} (\partial_{\abeta} q (\vec{x}))^2 \dx = 
\langle \vec{Q}^{\abeta}  \vec{w}( q), \vec{A} \, \vec{Q}^{\abeta} \, \vec{w}( q) \rangle.   
\end{equation*}
Upon introducing the vector $ \vec{v} = \frac {\vec{w}} {\boldsymbol{\Delta} \vec{x}^{\auno}}$, we have
$$ \begin{aligned}  
I[q]  
& = \sum_{\vert \abeta \vert = 1}^M \boldsymbol{\Delta} \vec{x}^{2 \abeta -\auno}  \int_{\Omega_0} (\partial_{\abeta} q (\vec{x}))^2 \dx   
= \sum_{\vert \abeta \vert = 1}^M \boldsymbol{\Delta} \vec{x}^{ -\mi{2}}  \langle   \vec{Q}^{\abeta} \, \vec{w}, \vec{A} \, \vec{Q}^{\abeta} \, \vec{x}\rangle 
\\ & =  \sum^M_{\vert \abeta \vert = 1}   \langle   \vec{Q}^{\abeta} \, \vec{v}, \vec{A} \, \vec{Q}^{\abeta} \, \vec{v} \rangle =    \langle    \vec{v},  \vec{C} \, \vec{v} \rangle,
\end{aligned} $$
where $ \vec{C}=   \sum_{\vert \abeta \vert =1}^G (\vec{Q}^{\abeta})^T \, \vec{A} \, \vec{Q}^{\abeta} $ is the smoothness measuring matrix, which is  constant and symmetric. Moreover $\vec{C}$ is positive semi-definite since, by equation \eqref{eq:ind}, $I[q] \ge 0.$ 
\end{proof}

For some concrete examples of the matrices $\vec{A}, \vec{Q}^{\abeta} $ and $ \vec{C}$, see the Appendix. For an alternative and more sparse representation of matrix $C$ in one space dimension see \cite{BGS:wao}, where a suitable orthogonal basis for the polynomials is introduced.

\begin{proposition} \label{prop:behav_v}
Let $\mathcal{S}$ be a stencil including $\Omega_0$ and let $q(\vec{x})$ be a polynomial  with  $\deg q(\vec{x}) \ge M $ for which the components of the vector $ \vec{v}(q)$  defined in Proposition \ref{prop:ind} satisfy the condition
 \begin{equation}  
  (\vec{v}(q))_{\aalpha}  =  
  \boldsymbol{\Delta} \vec{x}^{{\aalpha- \auno}} \int_{\Omega_j} \partial_{\aalpha} q(\vec{x}) \dx =
  \sum_{\substack{\vert \abeta \vert =0, \\  \; \abeta  \; \text{even}}}^{M- \vert \aalpha \vert} \frac  { \partial_{\aalpha + \abeta} u(\azero)} { 2^{\vert \abeta \vert}(\abeta + \auno)!} \boldsymbol{\Delta} \vec{x}^{{\abeta + \aalpha}} +  \Ogrande( \rho^{M+1} ), 
 \label{eq:cond_v}
 \end{equation}
 for a regular function $u(\vec{x})$ and $ 1 \le \vert \aalpha \vert \le M$, 
  then
 $$ I[q]  = B_M+ R[q]$$
 where $B_M$ depends on $M$ but not on $q(\vec{x})$.
 \end{proposition}
 \begin{proof}
 We define
 $$ (\vec{v}^{B_M})_{\aalpha} 
 := 
 \left\{ 
 \begin{aligned} 
 & \sum_{\substack{\vert \abeta \vert =0, \\  \; \abeta  \; \text{even}}}^{M- \vert \aalpha \vert} \frac  { \partial_{\aalpha + \abeta} u(\azero)} { 2^{\vert \abeta \vert}(\abeta + \auno)!} \boldsymbol{\Delta} \vec{x}^{{\abeta + \aalpha}}  \quad  & 1 \le \vert \aalpha \vert \le M, \\ 
 & 0 &  \vert \aalpha \vert = 0 \; \text{and} \;\vert \aalpha \vert > M, \end{aligned} \right.
$$
and
$$ 
 (\vec{v}^{R[q]})_{\aalpha} = \Ogrande ( \rho^{M+1} )  $$
 and, as in Corollary 3 in \cite{DB:2013}, using Proposition \ref{prop:ind}, the thesis holds with
 \[ 
 B_M:= \langle \vec{v}^{B_M}, \vec{C} \, \vec{v}^{B_M} \rangle, \qquad R[q]:= 2 \langle \vec{v}^{B_M}, \vec{C} \, \vec{v}^{R[q]} \rangle + \langle
 \vec{v}^{R[q]}, \vec{C} \, \vec{v}^{R[q]} \rangle.
 \]
 \end{proof}
 
The actual size of the term $B_M$ and of the remainder $R[q]$ depends on the presence of a critical point in the data. 
\begin{definition}\label{def:ncp}
A smooth enough function is said to have a 
{\em critical point of order $n_{cp}$} in $\hat{\bf x}$, if 
$ \partial_{\aalpha} u(\hat{\bf x})= 0,$ for every $\vert \aalpha \vert \le k $ and $ \partial_{\aalpha} u(\azero) \neq 0$ for at least one multi-index with $ \vert \aalpha \vert = k +1$. If $n_{cp}=0$, the function will be called {\em regular}.
\end{definition}

\begin{corollary} \label{coro:sizeBR}
Assume that the hypothesis of Proposition \ref{prop:behav_v} hold.
If there is a critical point with 
$n_{cp} \ge M$ we have 
$B_M=0$ and 
$R[q] = \Ogrande(\rho^{2M+2})$.
In the case $n_{cp} < M$, we have 
$B_M= \Ogrande(\rho^{2(n_{cp} +1)}) $ and 
$R[q]= \Ogrande(\rho^{M+2+n_{cp}})$, so that 
$ R[q] = \Opiccolo(B_M)$.
\end{corollary}

\begin{proof}
Only the derivatives of orders $\vert \aalpha \vert$ to $M$ appear in $(\vec{v}^{B_M})_{\aalpha}.$ 
If all these derivatives are equal to zero since $n_{cp} \ge M$ we have $\vec{v}^{B_M}=\azero$ and $B_M=0$.
Then $ R[q]$ depends only on $ \vec{v}^{R[q]}$, so $R[q] = \Ogrande(\rho^{2M+2})$.
Otherwise, if $n_{cp} < M,$ 
we have $\min_{1 \le \vert \aalpha \vert \le M}\theta((\vec{v}^{B_M})_{\aalpha})= n_{cp} + 1.$
Then $\theta(B_M) =2 n_{cp} +2$  and $\theta(R[q])= M + n_{cp}+2$
so $\theta(R[q]) > \theta(B_M),$ i.e.
 $ R[q] = \Opiccolo(B_M)$.
\end{proof} 

 \begin{remark} \label{rem:P0BM}
 In view of Remark \ref{rem:I0Iopt}, we point out that if $P_1(\vec{x}), P_2(\vec{x}), \cdots, P_m(\vec{x})$ and $\Popt(\vec{x})$ satisfy Proposition \ref{prop:behav_v} with the same $u(\vec{x})$,  then $P_0(\vec{x})$ satisfies 
 the same Proposition since $\vec{w}(q)$, and then $\vec{v}(q)$, are linear in $q$.
 \end{remark}

\begin{example} \label{ex:ind1_1d}
In the one dimensional \CWENO\ and \CWENOZ\ reconstructions of order $3$ we have that the polynomials $P_1(x), P_2(x)\in\Poly{1}_1 $ and
$\Popt(x), P_0(x) \in\Poly{2}_1 $ satisfy the hypothesis of Proposition \ref{prop:behav_v} with $M=1$. We have $B_1= \up^2 \Delta x^2$
and
$$ \begin{aligned}
 I[P_1] &= B_1 - \up \us \Delta x^3 +  \Ogrande(\Delta x^4), \\
 I[P_2] &=  B_1 + \up \us \Delta x^3 +  \Ogrande(\Delta x^4), \\
 I[\Popt] &= B_1 + \Ogrande(\Delta x^4), \\
 I[P_0] &= B_1 + \frac {d_1-d_2} {d_0} \up \us \Delta x^3 +  \Ogrande(\Delta x^4),
 \end{aligned} 
$$ 
and thus $R[P_1], R[P_2],R[P_0] $ are $\Ogrande(\Delta x^3)$ while $R[\Popt]$ is $\Ogrande(\Delta x^4).$
(See also \cite{CS:epsweno})
 \end{example}

\begin{example} \label{ex:ind2_1d}
 In the one dimensional \CWENO\ and \CWENOZ\  reconstructions of order $5$ we have that $P_1(x), P_2(x), P_3(x)\in\Poly{2}_1$ and $\Popt(x)\in\Poly{4}_1 $ satisfy the hypothesis of Proposition \ref{prop:behav_v} with $M=2$. We have 
 $B_2= \up^2 \Delta x^2+ \frac {13} {12} \us^2 \Delta x^4$
 and $R[\Popt]$, $R[P_1]$, $R[P_2]$, $R[P_3]$ are all $\Ogrande(\Delta x^4).$
\end{example}
 
\begin{example} \label{ex:ind_1d}
 In the one dimensional $\mathsf{CWENOZ}$ reconstructions of order $7$ we have that $P_1(x), P_2(x), P_3(x), P_4(x)\in\Poly{3}_1$ and $\Popt(x)\in\Poly{6}_1 $ satisfy the hypothesis of Proposition \ref{prop:behav_v} with $M=3$,
 $B_3= \up^2 \Delta x^2
   +  \left( \frac {13} {12} \us^2 + \frac 1 {12} \up \ut \right) \Delta x^4 
   + \frac {1043} {960} \ut^2 \Delta x^6$
    and $R[P_1], \ldots, R[P_4]$ are  $\Ogrande(\Delta x^5)$ while $R[\Popt]$ is $\Ogrande(\Delta x^6)$ .
\end{example}

In the following, we will consider $q(\vec{x}) \in \{ P_1(\vec{x}), \ldots, P_m(\vec{x}), \Popt(\vec{x}) \}$, so, to simplify the notations,
we write $R_k=R[P_k]$ for $k=1,\ldots,m$ and $R_0=R[\Popt].$

In order to obtain a computationally cheap \CWENOZ\ reconstruction procedure, we follow the idea of \cite{BCCD:2008:wenoz5} and of the later \WENOZ\ constructions and define the global smoothness indicator $\tau$ as a linear combination of the other smoothness indicators $I_0,\ldots,I_m$.

\begin{definition} \label{def:tau}
The global regularity indicator of a \CWENOZ\ scheme is defined as a linear combination
\begin{equation}   \label{eq:tau}
\tau := \left| \sum_{k=0}^m \lambda_k I_k \right|,
\end{equation}
for some choice of coefficients $\lambda_0,\ldots,\lambda_m$ such that
$\sum_{k=0}^{m}\lambda_k=0$.
\end{definition}

We point out that the idea to consider also the smoothness indicator of $\Popt$ has been also exploited, in the context of \WENO\ schemes, by \cite{MR3341001} to define an improved $\WENOZ$ scheme.

In the next section we will study the optimal choice for the coefficients $\lambda_k$ in \eqref{eq:tau}, but first we prove some general results using only the assumption of Definition \ref{def:tau}.
\begin{remark}\label{rem:tauRk}
From the proposition \ref{prop:behav_v}, we have  $\tau= \left| B_M \sum_{k=0}^m \lambda_k +  \sum_{k=0}^m \lambda_k R_k \right|, $  so,
thanks to the hypotesis $\sum_k\lambda_k=0$,
 $\tau=\Ogrande(\sum_{k=0}^m R_k)$.
\end{remark}

\subsection{Order of accuracy of a \CWENOZ\ scheme} \label{ssec:ord_acc}

A \CWENOZ\ reconstruction is a combination of polynomials, each of which is accurate of order $g+1$ and we are interested in designing the nonlinear weights of the combination such that, for smooth data, the accuracy of the reconstruction polynomial $\Prec^{\text{CWZ}}$ is boosted to the accuracy of $\Popt$, i.e. to $G+1$. In this section we exploit the sufficient condition of Remark \ref{rem:weights} in order to choose the optimal values for the parameters $\elle$ and $\epsilon$ appearing in \eqref{eq:OmegaZ}. 

Let us first point out that, although $\epsilon$ cannot be taken exactly $0$, the smaller it is, the less influence it will have on the accuracy of the reconstruction and on its ability to reach the optimal order of convergence already on coarse meshes. As in \cite{Arandiga:11,DB:2013,Kolb:14,CS:epsweno}, we allow a dependence of $\epsilon$ on the cell size, namely $\epsilon=\rho^{\hat{m}}$, with an exponent $\hat{m}$ that we will choose as large as possible. Regarding the parameter $\elle$, it should be taken as small as possible: high values for $\elle$ enhance the ratios between the indicators of the polynomials and make the reconstruction more dissipative on discontinuous solution, as pointed out in \cite{DB:2013}.
The main result of this section is that a proper choice for the global smoothness indicator $\tau$ will help in having optimal reconstruction order with small $\elle$ and $\epsilon$.

We will make use of the following result, proven in \cite{DB:2013}.
\begin{lemma}[Lemma 6 of \cite{DB:2013}] \label{lemma:thetaEps}
If $\alpha^{Z}_k=d_k\left(1+A\rho^{\gamma}+\Ogrande\left(\rho^{\gamma+1}\right)\right)$ for $k=0,\ldots,m$ with $\gamma >0$ and $A$ independent on $k$, then
$\omega^{\text{CWZ}}_k-d_k=\Ogrande(\rho^{\gamma+1})$ for $k=0,\ldots,m$.
\end{lemma}

The accuracy of a $\CWENOZ$ scheme is expressed by the following result.

\begin{theorem}\label{teo:accuracy}
  Assume that  the polynomials $P_1(\vec{x}), \ldots, P_m(\vec{x})$ and $\Popt(\vec{x}) $ in the $ \CWENOZ $ scheme  
  satisfy Proposition \ref{prop:behav_v} for some $M$ and the same $u(x)$
  and 
  that the parameters $\epsilon = C_{\epsilon} \rho^{\hat{m}}$, for some   $C_{\epsilon}\neq 0,$  and $\elle$ in \eqref{eq:OmegaZ} satisfy
  \begin{subequations} \label{eq:teorema}
  \begin{align}
  & \hat{m} \leq 2M+1 \label{eq:teorema:a} \\
  & \elle(2M+2-\hat{m})\geq G-g-1 \label{eq:teorema:b}\\
  &  \elle \left[\left.\theta(\tau)\right|_{n_{cp}=0}
    -\min(\hat{m},2M)
   \right]
  \geq G-g-1. \label{eq:teorema:c}
  \end{align}
  \end{subequations}
  Then, on smooth data, the \CWENOZ\ scheme  achieves the optimal order $G+1$ as 
  $\rho\to0$.
\end{theorem} 

To prove that the optimal order  $G+1$ as $ \rho \to 0$ is achieved, it  is sufficient to verify that $\theta(\omega^Z_k -d_k ) \ge G-g$, as seen in the Remark \ref{rem:weights}.
We consider separately the cases $n_{cp}\geq M$ and $n_{cp}<M$.
\begin{proof}[Proof for $n_{cp}\geq M$]
Corollary \ref{coro:sizeBR} ensures that $B_M=0$ and $\theta(R_k)\geq 2M+2$ and, together with Remark \ref{rem:tauRk}, one has that
\[
I_k = C_k\rho^{2M+2}+\Ogrande(\rho^{2M+3})
\quad\text{and}\quad
\tau = C_{\tau}\rho^{2M+2}+\Ogrande(\rho^{2M+3}),
\]
for some, possibly null, constants $C_0,\ldots,C_m,C_{\tau}$ that do not depend on $\rho$.
Then $\alpha^Z_k=d_k\left(1+\gamma_k^\elle\right)$ for
\[
\gamma_k 
= \frac{\tau}{I_k+\epsilon}
=\frac
{C_{\tau}\rho^{2M+2}+\Ogrande(\rho^{2M+3})}
{C_k\rho^{2M+2}+\Ogrande(\rho^{2M+3}) + C_{\epsilon}\rho^{\hat{m}}
}
\,.
\]
For the convergence of the scheme it is, of course, necessary that $\gamma_k\to0$ and thus that $\hat{m}\leq2M+1$. Under this hypothesis,
\[
\gamma_k
\sim\frac{C_{\tau}}{C_{\epsilon}} \rho^{2M+2-\hat{m}}
=C \rho^{2M+2-\hat{m}}
\]
for some, possibly null, constant $C$ that does not depend on $\rho$. Thus
\[
\alpha^Z_k \sim d_k\left(1+C^\elle\rho^{\elle(2M+2-\hat{m})}\right),
\]
and $\elle(2M+2-\hat{m})\geq \elle\geq 1$. This allows to apply Lemma \ref{lemma:thetaEps}, concluding that
\begin{equation} \label{eq:omegaMenD:Mgrande}
\theta(\omega^Z_k-d_k)
\geq \elle(2M+2-\hat{m})+1 \geq G-g
\end{equation}
for $n_{cp}\geq M$ and concludes the proof for this case.
\end{proof}

We now turn to the other case.
\begin{proof}[Proof for $n_{cp}< M$]
In this case
\[
\gamma_k 
= \frac{\tau}{I_k+\epsilon}
= \frac{\tau}{B_M+\epsilon}
  \frac{1}{1+\frac{R_k}{B_M+\epsilon}}
=\frac{\tau}{B_M+\epsilon} (1+o(1))
\,,
\]
provided that $b_k=\frac{R_k}{B_M+\epsilon} \to 0.$
However, this is true for any $\hat{m}\geq0$ since Corollary \ref{coro:sizeBR} states that $B_M\neq0$ and $R_k/B_M\to0$. Then $\gamma_k=C\rho^t$ for a constant $C\neq0$ and 
\[
  t
  =\theta(\tau)-\theta(B_M+\epsilon)
  =\theta(\tau)-\min\left(\theta(B_M),\hat{m}\right)
.\]
Corollary \ref{coro:sizeBR} implies that $\theta(B_M)= 2n_{cp}+2$ and, together with Remark \ref{rem:tauRk}, that $\theta(\tau)\geq \theta(R_k)=M+n_{cp}+2$, so that
$t\geq M-n_{cp}\geq 1$.

Finally, we have
$\alpha^Z_k
 =d_k\left(1+\gamma_k^\elle\right)
 =d_k\left(1+C^\elle\rho^{\elle t}\right)
$
with $\elle t\geq1$ and thus Lemma \ref{lemma:thetaEps}, together with $\theta(B_M)=2n_{cp}+2$, implies that
\begin{equation}\label{eq:omegaMenD:Mpiccolo}
\theta(\omega^Z_k-d_k)
\geq 1+\elle t = 1+ \elle \big[\theta(\tau)-\min\left(2n_{cp}+2,\hat{m}\right)\big].
\end{equation}
Since the minimum  value of  $\theta(\tau)$ is attained when $n_{cp}=0$,  we have that
$$
\theta(\omega^Z_k-d_k)
\geq 
1+\elle \big[\left.\theta(\tau)\right|_{n_{cp}=0}
    -\min(\hat{m},2M)
\big] \ge G-g,
$$
which concludes the proof for this case.
\end{proof}

Note that bounds for $\theta(\omega_k-d_k)$ for each specific choice of $\hat{m}, \elle$ and in the presence of a critical point of order $n_{cp}$ is given by equation \eqref{eq:omegaMenD:Mgrande} if $n_{cp}\geq M$ and by equation \eqref{eq:omegaMenD:Mpiccolo} otherwise. 

\begin{remark} \label{rem:conIP0}
Since $P_0$  satisfies 
the Proposition \ref{prop:behav_v} (see Remark~\ref{rem:P0BM}), Theorem~\ref{teo:accuracy} is also true for $I_0=I[P_0]$ in Definition~\ref{def:CWENO}.
\end{remark}

\begin{remark}
We point out that condition \eqref{eq:teorema:b} is always satisfied if $\elle$ is taken large enough.
If $\theta(\tau) \ge 2M+1 $ the same is true for \eqref{eq:teorema:c} and thus $ \forall \hat{m} \le 2M+1, \exists \, \elle \ge 1 $ such that the \CWENOZ\ is convergent with optimal order. 
On the other hand if $\theta(\tau) \le 2M $, \eqref{eq:teorema:c}  can be satisfied only when $\hat{m} < \theta(\tau)$. In this case $ \exists  \, \elle \ge 1 $ such that the \CWENOZ\ is convergent with optimal order
only for $\hat{m} < \theta(\tau) < 2M+1$.
Moreover, we point out that \eqref{eq:teorema} implies that a larger value of $\theta(\tau)$ allows to use a smaller $\epsilon$ and a smaller power parameter $\elle$.
This allows us to design \CWENOZ\ schemes that outperform their \CWENO\ counterparts.
\end{remark}

\section{Optimal \CWENOZ\ reconstructions} \label{sec:glob_ind}

\subsection{\CWENOZ\ in one spatial dimension} 
\label{ssec:tau1d}
In the $\CWENOZ $ reconstruction of order $2r-1$, we employ the stencils and polynomials defined in Example \ref{ex:CW1d}, as in \cite{CPSV:coolweno}.
In that paper, however, we had considered the same global  smoothness indicator $ \tau_{2r-1}$ that are optimal for \WENOZ\ (see \cite{DB:2013}), which are based on the  $r$ polynomials of degree $r-1$, i.e. $\lambda_0=0$ in \eqref{eq:tau}. Now we allow $\lambda_0\neq0$ and we will denote $\hat{\tau}_{2r-1}$ the optimal definition of the global smoothness indicator for \CWENOZ\ of order $2r-1$.
The following Lemma is a generalization of the Examples~\ref{ex:ind1_1d}, \ref{ex:ind2_1d} and \ref{ex:ind_1d} for an arbitrary polynomial.

\begin{lemma} \label{lemma:1D}  
Let $q(x)$ be a polynomial of degree at least $\gamma$, interpolating a set of consecutive cell averages
of $u(x)$  in the sense of \eqref{eq:ps}, with stencils as in Example \ref{ex:CW1d}.
Then $q(x)$ satisfies the hypothesis of Proposition \eqref{prop:behav_v} for any $M\leq\gamma$.
\end{lemma}

\begin{proof}
We note that $q(x)$  is the derivative of the polynomial $Q(x)$ that interpolates the regular function $U(x)= \int_{-\infty}^x u(t) \, dt $ in the $\gamma+2$ points delimiting $\gamma+1$ consecutive cells. 
Then $\vert U(x) - Q(x) \vert = \Ogrande(\Delta x^{\gamma+2})$ and we have
\begin{align*} 
(\vec{v}(q))_{\alpha}   
&= 
{\Delta x}^{{\alpha-1}} 
\int_{ -\frac {\Delta x} {2}}^{ \frac {\Delta x} {2}} 
\partial_{\alpha} \,q(x) dx 
=  {\Delta x}^{{\alpha-1}} \int_{ -\frac {\Delta x} {2}}^{ \frac {\Delta x} {2}}  \partial_{\alpha+1} \, Q(x) \, dx
\\
&=  
{\Delta x}^{{\alpha-1}}
  \int_{ -\frac {\Delta x} {2}}^{ \frac {\Delta x} {2}} \left( \partial_{\alpha + 1} U(x) + \Ogrande(\Delta x^{\gamma +2 - \alpha -1 }) \right) \, dx  
  \\
&=  
{\Delta x}^{{\alpha-1}} \int_{ -\frac {\Delta x} {2}}^{ \frac {\Delta x} {2}} \partial_{\alpha} u(x) \, dx  + \Ogrande(\Delta x^{\gamma+1}) 
\\  
&= 
\sum_{\substack{\beta =0, \\  \; \beta  \; \text{even}}}^{\gamma-  \alpha } \frac  { \partial_{\alpha + \beta} u(0)} { 2^{ \beta}(\beta + 1)!} {\Delta x}^{{\beta + \alpha}}
  +  \Ogrande(\Delta x^{\gamma+1}).
\end{align*}
where $0 \le \alpha \le \gamma.$
\end{proof} 


From the Remark \ref{rem:tauRk}, we know that $ \tau = \left | \sum_{k=0}^r \lambda_k I_k \right |= \Ogrande(\sum_{k=0}^m R_k)$ whenever 
$ \sum \lambda_k=0$ and any choice of coefficients $\lambda_k$  sum up to $0$, would define a $\tau$ for \CWENOZ\ that generalizes the non-optimal definition of $\tau$ in \cite{CCD:11}. However, by examining closely each case, it is possible to obtain an even smaller $\tau$, as is the case for the optimal definition for \WENOZ\ given in \cite{DB:2013}.

\paragraph{CWENOZ3}
In {\bf{$ \mathsf{CWENOZ3} $}} (see \cite{CS:epsweno}) we have
\begin{align*}
I_1 
&= I[P_1] 
= B_1 
- \up \us \Delta x^3 + \left(\frac 5 {12} \up \ut + \frac 1 4 \us^2\right)\Delta x^4 + \Ogrande(\Delta x^5), 
\\
I_2 
&= I[P_2] 
= B_1 
+ \up \us \Delta x^3 
+ \left(\frac 5 {12} \up \ut + \frac 1 4 \us^2\right)\Delta x^4 + 
\Ogrande(\Delta x^5), 
\\
I_0 
&= 
I[\Popt] 
= B_1 
+ \left(\frac 5 {12} \up \ut + \frac {13} {12} \us^2\right)\Delta x^4 + \Ogrande(\Delta x^5).
\end{align*} 
Thus any set of coefficients such that $\lambda_1=\lambda_2$ and $\lambda_0= - 2 \lambda_1$ cancels all terms up to $\Ogrande(\Delta x^3)$, but it is never possible to cancel all the $\Ogrande(\Delta x^4)$ terms.
So we define
\begin{equation} \label{eq:tau3}
\hat{\tau}_{3}= \vert  t I_{1} +  t I_{2} -2t I_0 \vert,
\end{equation}
and, for any $t \in \mathbb{R}$, $\theta(\hat{\tau}_3)=4$, which allows to employ a smaller $\epsilon$ than the $\tau_3$ used in \cite{CPSV:coolweno}, since
$\theta(\tau_3)=3$.

\begin{remark}
We have explored also the possibility of defining the indicator $I_0=I[P_0]$, as in previous works on \CWENO\ (Remark~\ref{rem:I0Iopt}). We recall that convergence is established by Remark \ref{rem:conIP0}. In this case, assuming symmetry of the linear coefficients, we have $d_0=(1-d_1-d_2)=1-2d_1$ and the indicator $I_0$ depends on $d_0$ as
$$
I_0 = I[P_0] = B_1 + \left(\frac 5 {12} \up \ut + \frac {13} {12 d_0^2} \us^2\right)\Delta x^4 + \Ogrande(\Delta x^5).
 $$
However, it is easy to verify that no values for $\lambda_k$ or $d_0$ permit to improve and obtain $\theta(\hat{\tau}_3)=5$.
\end{remark}

In this case $\theta(\hat{\tau}) \ge 3= 2M+1$ then we obtain the optimal order of convergence $ \forall \hat{m} \le 3$ and in particular \eqref{eq:teorema} are satisfied $ \forall \elle \ge 1$.

\begin{table}
\begin{center}
  \caption{$\tau$, $\max_{k=0,1,2}\{|\omega^Z_k-d_k|\}$ and reconstruction errors for the \CWENOZ3\ reconstruction on a critical point with $n_{cp}=1$.}
  \label{tab:CWZ3:weights}
  \tiny
  \setlength{\tabcolsep}{3pt}
  \begin{tabular}{|r|r@{\hspace{2pt}}r|r@{\hspace{2pt}}r|r@{\hspace{2pt}}r|r@{\hspace{2pt}}r|r@{\hspace{2pt}}r|r@{\hspace{2pt}}r|r@{\hspace{2pt}}r|} 
    \hline
  	& \multicolumn{2}{|c|}{} 
    & \multicolumn{4}{c|}{$\elle=2, \hat{m}=2$} 
  	& \multicolumn{4}{c|}{$\elle=2, \hat{m}=3$} 
  	& \multicolumn{4}{c|}{$\elle=2, \hat{m}=4$} 
  \\
  	\multicolumn{1}{|c|}{$\Delta x$} 
    & \multicolumn{1}{c}{$\tau$}     & \multicolumn{1}{c|}{rate} 
    & \multicolumn{1}{c}{$\omega^Z-d$} & \multicolumn{1}{c|}{rate} 
    & \multicolumn{1}{c}{error}      & \multicolumn{1}{c|}{rate} 
    & \multicolumn{1}{c}{$\omega^Z-d$} & \multicolumn{1}{c|}{rate} 
    & \multicolumn{1}{c}{error}      & \multicolumn{1}{c|}{rate} 
    & \multicolumn{1}{c}{$\omega^Z-d$} & \multicolumn{1}{c|}{rate} 
    & \multicolumn{1}{c}{error}      & \multicolumn{1}{c|}{rate} 
  \\ \hline
5.0e-02 & 1.4e-03 &  --- & 1.7e-02 &  --- & 2.5e-04 &  --- & 4.4e-01 &  --- & 1.8e-03 &  --- & 5.6e-01 &  --- & 2.8e-03 &  --- \\ 
2.5e-02 & 9.1e-05 & 3.99 & 4.8e-04 & 5.11 & 3.9e-05 & 2.66 & 3.4e-01 & 0.38 & 2.9e-04 & 2.63 & 5.6e-01 & 0.01 & 5.8e-04 & 2.26 \\ 
1.3e-02 & 5.7e-06 & 4.00 & 8.7e-06 & 5.79 & 5.0e-06 & 2.97 & 2.1e-01 & 0.70 & 4.1e-05 & 2.83 & 5.6e-01 & 0.00 & 1.3e-04 & 2.15 \\ 
6.3e-03 & 3.6e-07 & 4.00 & 1.4e-07 & 5.95 & 6.3e-07 & 2.99 & 9.2e-02 & 1.18 & 4.2e-06 & 3.27 & 5.6e-01 & 0.00 & 3.1e-05 & 2.08 \\ 
3.1e-03 & 2.2e-08 & 4.00 & 2.2e-09 & 5.99 & 7.9e-08 & 3.00 & 2.7e-02 & 1.76 & 2.8e-07 & 3.94 & 5.6e-01 & 0.00 & 7.4e-06 & 2.04 \\ 
1.6e-03 & 1.4e-09 & 4.00 & 3.5e-11 & 6.00 & 9.8e-09 & 3.00 & 5.6e-03 & 2.28 & 8.2e-09 & 5.06 & 5.6e-01 & 0.00 & 1.8e-06 & 2.02 \\ 
7.8e-04 & 8.7e-11 & 4.00 & 5.4e-13 & 6.00 & 1.2e-09 & 3.00 & 9.0e-04 & 2.64 & 5.1e-10 & 4.03 & 5.6e-01 & 0.00 & 4.5e-07 & 2.01 \\ 
3.9e-04 & 5.4e-12 & 4.00 & 8.5e-15 & 6.00 & 1.5e-10 & 3.00 & 1.3e-04 & 2.83 & 1.3e-10 & 1.98 & 5.6e-01 & 0.00 & 1.1e-07 & 2.01 \\ 
1.9e-04 & 3.4e-13 & 4.00 & 1.3e-16 & 6.00 & 1.9e-11 & 3.00 & 1.7e-05 & 2.92 & 1.8e-11 & 2.80 & 5.6e-01 & 0.00 & 2.8e-08 & 2.00 \\ 
9.8e-05 & 2.1e-14 & 4.00 & 2.1e-18 & 6.00 & 2.4e-12 & 3.00 & 2.2e-06 & 2.96 & 2.4e-12 & 2.95 & 5.6e-01 & 0.00 & 7.1e-09 & 2.00 \\ 
4.9e-05 & 1.3e-15 & 4.00 & 3.2e-20 & 6.00 & 3.0e-13 & 3.00 & 2.7e-07 & 2.98 & 3.0e-13 & 2.99 & 5.6e-01 & 0.00 & 1.8e-09 & 2.00 \\ 
2.4e-05 & 8.3e-17 & 4.00 & 5.1e-22 & 6.00 & 3.7e-14 & 3.00 & 3.4e-08 & 2.99 & 3.7e-14 & 3.00 & 5.6e-01 & -0.00 & 4.4e-10 & 2.00 \\ 
1.2e-05 & 5.2e-18 & 4.00 & 7.9e-24 & 6.00 & 4.7e-15 & 3.00 & 4.3e-09 & 3.00 & 4.7e-15 & 3.00 & 5.6e-01 & -0.00 & 1.1e-10 & 2.00 \\ 
6.1e-06 & 3.2e-19 & 4.00 & 1.2e-25 & 6.00 & 5.9e-16 & 3.00 & 5.4e-10 & 3.00 & 5.9e-16 & 3.00 & 5.6e-01 & -0.00 & 2.7e-11 & 2.00 \\ 
  \hline
  \end{tabular}
\end{center}
\end{table}

We have conducted a thorough study to check the conditions given by Theorem~\ref{teo:accuracy}. We report here in Table~\ref{tab:CWZ3:weights}, as an example, only the most difficult case, i.e. the case of a critical point with $n_{cp}=1$.

A grid of size $\Delta x$ was set up with a cell centre in the critical point of the function $\sin(\pi x-\sin(\pi x)/\pi)$ located at about $x_{\text{crit}}\simeq0.596683186911209$, see \cite{HAP:2005:mappedWENO}. The cell averages of this cell and of its immediate neighbours were initialized with the 2-point gaussian quadrature rule. Next, the \CWENOZ3\ reconstruction was computed for the cell containing the critical point. Due to the very fine grids employed, all computations were performed in quadruple precision.

In the table we show the case of $\elle=2$ and compare three values for $\hat{m}$. 
First, $\tau$ decays proportionally to $\Delta x^4$, as indicated above. 
When $\hat{m}=2$, the rate of decay of the distance between the nonlinear weights and the optimal ones is $6$ and the reconstruction error is of order $3$ already on very coarse grids.
When $\hat{m}=3$, the rate of decay of the nonlinear weights approaches $3$ very slowly and the reconstruction error is of order $3$ only on very fine grids.
When $\hat{m}=4$, the nonlinear weights do not tend to the linear ones and the reconstruction error is of order $2$.
All the finding are in agreement with the results of Theorem~\ref{teo:accuracy}.

\paragraph{CWENOZ5}
For $ \mathsf{CWENOZ5}$  we have (see \cite{CPSV:cweno})
\begin{align*}
I_1 
=  
I[P_1] 
&= 
B_2 - \frac 7 {12} \up \ut \Delta x^4
+\left( \frac {1} {2} \up \uq - \frac {13} 6 \us \ut \right) \Delta x^5
\\ 
& + \left(  - \frac {749} {2880} \up \uc + \frac {65} {48} \us \uq + \frac {673} {576} \ut^2 \right) \Delta x^6 +\Ogrande(\Delta x^7), 
\\
I_2 = I[P_2] &= 
B_2 + \frac 5 {12} \up \ut \Delta x^4 
\\
&+\left(   \frac {91} {2880} \up \uc + \frac {13} {48} \us \uq + \frac {25} {576} \ut^2 \right) \Delta x^6 +\Ogrande(\Delta x^7), 
\\
I_3 =I[P_3] 
&= B_2 - \frac 7 {12} \up \ut \Delta x^4 
 + \left(- \frac {1} {2} \up \uq + \frac {13} 6 \us \ut \right) \Delta x^5   
\\ 
& + \left(  - \frac {749} {2880} \up \uc + \frac {65} {48} \us \uq + \frac {673} {576} \ut^2 \right) \Delta x^6 +\Ogrande(\Delta x^7), 
\\
I_0 = I[\Popt] 
&= B_2 + \frac 1 {12} \up \ut \Delta x^4 
\\ & + 
 \left(  - \frac {21} {320} \up \uc + \frac {7} {80} \us \uq + \frac {1043} {960} \ut^2 \right) \Delta x^6 +\Ogrande(\Delta x^7).
\end{align*}
We can thus obtain $\theta(\hat{\tau}_5)=6$ with any choice
\begin{equation} \label{eq:tau5}
\hat{\tau}_{5}= \vert  t I_1 +  4t I_2 + t I_3 -6t I_0 \vert, \qquad t \in \mathbb{R}
\end{equation}
and no value of $t$ can lead to an improvement. In any case, this is better than the situation for \WENOZ\ and for the \CWENOZ\ of \cite{CPSV:coolweno}, where one has only $\theta(\tau_5)=5$.

Also in this case we have explored the choice $I_0=I[P_0]$, (see Remark \ref{rem:conIP0}), leaving free the values of the linear coefficients, but this has not lead to any improvement.

In this case $\theta(\hat{\tau}) \ge 5= 2M+1$ then we obtain the optimal order of convergence $ \forall \hat{m} \le 5$ and in particular \eqref{eq:teorema} are satisfied $ \forall \elle \ge 1$.

\paragraph{CWENOZ of orders higher than 5}
Proceding in a similar manner, we obtain for $ \mathsf{CWENOZ7}  $, $ \mathsf{CWENOZ9}$, $ \mathsf{CWENOZ11}$ respectively
%
\begin{center}
\begin{tabular}{ll}
$\hat{\tau_{7}} = \vert  -t I_1  -3  t I_2+ 3tI_3 + t I_4  \vert
$
&
$\theta(\hat{\tau}_7)=7$
\\
$\hat{\tau_{9}} = \vert  t I_1 +(2t+u) I_2+ (3u-6t)I_3 + (2t+u) I_4+ t I_5  -5u I_0 \vert
$
&
$  \theta(\hat{\tau}_9)=8$
\\
$\begin{aligned}
\hat{\tau}_{11}  = \vert  &(u-s) I_1 +(37u-t) I_2 + (10s+118u-2t) I_{3}+ \\
 &+ (2t-10s+54u ) I_4+ t I_5  +s I_6-210 u I_0 \vert 
\end{aligned}
$
&
$\theta(\hat{\tau}_{11})=9$ ,
\end{tabular}
\end{center}
where the parameters $t,s,u$ can take any real values.
In these last cases we have obtained a $\hat\tau$ of the same order as the $\tau$ in \cite{CPSV:coolweno} and, the global smoothness indicators of \cite{CPSV:coolweno} correspond to the choice $t=1$ in 
$\hat{\tau}_{7}, \, 
 u=0,\, t=1 $ in $\hat{\tau}_{9} $ and $ u=0,\,  t=1,\,  s=t$ in $\hat{\tau}_{11}. $

In the case of $ \mathsf{CWENOZ7}  $ again, the optimal order of convergence is possible for $ \hat{m} \le 7$, but \eqref{eq:teorema} requires  $  \elle \ge 2$ if $\hat{m} =7$.    Condition \eqref{eq:teorema:c} can be satisfied only for $ \hat{m} \le 7$ for $ \mathsf{CWENOZ9} $ and for $\hat{m} \le 8$ in the $ \mathsf{CWENOZ11} $ case.

\begin{table}
\begin{center}
\caption{Upper bound for the sensitivity order $\hat{m}$  
in 1D \CWENOZ\ reconstructions of Example \ref{ex:CW1d}.}
\label{tab:epsilon}
\footnotesize
\begin{tabular}{|rrr|cccccc|}
\hline
$r$ & 2$r$-1 & $\theta(\hat{\tau}_{2r-1})$ & $\elle=1$ & $\elle=2$ & $\elle=3$ & $\elle=4$ & $\elle=5$ & $\elle=6$  \\
 \hline
2 &   3 &   4    & 3 &   3    &  3   &   3     &  3 &  3   \\
3 &   5 &   6    & 5 &   5    &  5   &   5     &  5 &  5   \\
4 &   7 &   7   & 6 &   7    &  7   &   7     &  7 &  7   \\
5 &   9 &   8    & 5 &   6    &  7   &   7     &  7 &  7   \\
6 &   11 &   9    & 5 &   7    &  7   &   8     &  8 &  8   \\
\hline
\end{tabular}
\end{center}
\end{table}

We summarize the result of this section in Table \ref{tab:epsilon}, where we report  the possible choices for the order $\hat{m}$ of the sensitivity parameter $ \epsilon $ for various choices of $\elle$ in \eqref{eq:OmegaZ}.
 
\subsection{CWENOZ3 in two spatial dimension} 
\label{ssec:tau2d}
We now consider the third order reconstruction of \cite{SCR:CWENOquadtree}, restricted to the Cartesian uniform grid case.
The stencil for the reconstruction is the patch of $3\times3$ cells centered in the cell in which the reconstruction is sought.
The optimal polynomial $\Popt^{(2)} (\vec{x})$ is the polynomial of degree 2, defined on $\eta=9$ cells as in \eqref{eq:ps}.
Additionally we consider $m=4$  polynomials of degree 1 that fit 4 cell averages in a sector of the main stencil. In particular, we denote as $P_{NE}(\vec{x})\equiv P_1^{(1)}(\vec{x})$ the polynomial of degree 1 that interpolates exactly the central cell average and, in a least-square sense, the 3 cells located on the north, east and north-east directions; $P_{SE}(\vec{x})\equiv P_2^{(1)}(\vec{x}) , P_{SW}(\vec{x}) \equiv P_3^{(1)}(\vec{x}), P_{NW}(\vec{x}) \equiv P_4^{(1)}(\vec{x})$ the analogous polynomials in the south-east, south-west or north-west sub-stencils respectively.
The expression for all the polynomials involved is given in 
\cite{CWENOandaluz} and is reported here for convenience.

As in \cite{SCR:CWENOquadtree}, we employ a polynomial basis such that all non-constant basis elements have zero mean in the central cell and considering $\boldsymbol{\Delta} \vec{x} = (h,k), $
we have:
$$ 
\begin{array}  {lll}
 \varphi_{(0,0)}(\vec{x}) = 1   & \varphi_{(1,0)}(\vec{x})  = x &
 \varphi_{(0,1)}(\vec{x})  = y   \\ \varphi_{(2,0)}(\vec{x})  = x^2 - \frac {h^2} {12} \qquad &
 \varphi_{(0,2)}(\vec{x})  = y^2 - \frac {k^2} {12} \qquad & 
 \varphi_{(1,1)}(\vec{x})  = xy.
\end{array}
$$
In this basis, we have
\begin{equation*}
\Popt^{(2)} (\vec{x}) = \overline{u}_0 + \sum_{\vert \abeta \vert =1}^2  c_{\abeta}^{(2)} \varphi_{\abeta}(\vec{x}) 
\;\text{ and }\;
P_{j}^{(1)} (\vec{x})  = \overline{u}_0 +  \sum_{\vert \abeta \vert =1} c_{\abeta}^{j} \varphi_{\abeta}(\vec{x}), 
\;  j=1, \cdots, 4
\end{equation*}  
with coefficients computed in \cite{CWENOandaluz}.
This example satisfies the hypothesis of Proposition \ref{prop:behav_v} with $1 \le \vert \aalpha \vert \le M=1$.
In fact, using the vertical bar to denote the evaluation point, the Taylor expansion of the cell averages of $u(\vec{x})$ is

\newcommand{\atzerosmall}{\raisebox{-.5em}{\rule{.5pt}{1.25em}}_{\azero}}
\newcommand{\atzero}{\raisebox{-.75em}{\rule{.5pt}{1.75em}}_{\azero}}

\begin{equation} \label{eq:cellavg2d} 
\begin{aligned}
 {\overline{u}}_{NW} 
 & = u\atzerosmall - h  u_x\atzerosmall +k  u_y\atzerosmall 
   + \frac {13} {24}
  h^2  u_{xx}\atzerosmall -hk  u_{xy}\atzerosmall 
  +  \frac {13} {24}   k^2 u_{yy}\atzerosmall + \Ogrande(\rho^3)\\
 {\overline{u}}_{N} & = u\atzerosmall  +k  u_y\atzerosmall + \frac {1} {24}
  h^2  u_{xx}\atzerosmall +  \frac {13} {24}
   k^2  u_{yy}\atzerosmall + \Ogrande(\rho^3) \\
 {\overline{u}}_{NE} & = u\atzerosmall + h  u_x\atzerosmall +k  u_y\atzerosmall + \frac {13} {24}
  h^2  u_{xx}\atzerosmall +hk  u_{xy}\atzerosmall 
  +  \frac {13} {24} k^2  u_{yy}\atzerosmall + \Ogrande(\rho^3)\\
   {\overline{u}}_{W} & = u\atzerosmall - h  u_x\atzerosmall  + \frac {13} {24}
    h^2  u_{xx}\atzerosmall  +  \frac {1} {24}
     k^2  u_{yy}\atzerosmall + \Ogrande(\rho^3)\\
 {\overline{u}}_0 & = u\atzerosmall  + \frac {1} {24}
  h^2  u_{xx}\atzerosmall  +  \frac {1} {24}
   k^2  u_{yy}\atzerosmall + \Ogrande(\rho^3)\\
 {\overline{u}}_{E} & =  u\atzerosmall + h  u_x\atzerosmall + \frac {13} {24}
  h^2  u_{xx}\atzerosmall  +  \frac {1} {24}
   k^2  u_{yy}\atzerosmall +\Ogrande(\rho^3)\\ 
 {\overline{u}}_{SW} 
 & = u\atzerosmall - h  u_x\atzerosmall -k  u_y\atzerosmall + \frac {13} {24}
  h^2  u_{xx}\atzerosmall +hk  u_{xy}\atzerosmall 
  +  \frac {13} {24}
   k^2  u_{yy}\atzerosmall + \Ogrande(\rho^3) \\
 {\overline{u}}_{S} 
 & = u\atzerosmall -k  u_y\atzerosmall + \frac {1} {24}
  h^2  u_{xx}\atzerosmall  +  \frac {13} {24}
   k^2  u_{yy}\atzerosmall + \Ogrande(\rho^3)\\
 {\overline{u}}_{SE} 
 & = u\atzerosmall + h  u_x\atzerosmall -k  u_y\atzerosmall + \frac {13} {24}
  h^2  u_{xx}\atzerosmall -hk  u_{xy}\atzerosmall 
  +  \frac {13} {24}
   k^2 u_{yy}\atzerosmall + \Ogrande(\rho^3)
   \end{aligned}
\end{equation}
Substituting the above expressions in the coefficients given in \cite{CWENOandaluz},  for $  \vert \aalpha \vert =1$,  we obtain
$$ 
 c_{\aalpha}^{(2)} 
 = \frac 1 {\aalpha !} \partial_{\aalpha} u\atzero  + \Ogrande(\rho^2),    
 \text{ and }
 c_{\aalpha}^r 
 =  \frac 1 {\aalpha !} \partial_{\aalpha}  u\atzero  + \Ogrande(\rho), \qquad r=1, \cdots,4.
$$
We consider $q(\vec{x}) \in \{ P_{1}^{(1)} (\vec{x}), \cdots ,P_{4}^{(1)} (\vec{x}), \Popt^{(2)}(\vec{x}) \}.$
The basis functions $ \varphi_{\abeta}(\vec{x})$ differ from the monomial basis only for constant terms and we have, for $  \vert \aalpha \vert =1$,
\begin{align*} 
 (\vec{v}(q))_{\aalpha} 
 & = \boldsymbol{\Delta} \vec{x}^{{\aalpha -\auno }} \int_{\Omega_0} \partial_{\aalpha} q (\vec{x}) \dx =  \boldsymbol{\Delta} \vec{x}^{{\aalpha}-\auno} 
 \sum_{\vert \abeta \vert =1}^{\gamma} 
 c_{\abeta} \int_{\Omega_0} \partial_{\aalpha} \varphi_{\abeta}(\vec{x}) \dx 
 \\ 
 & 
 =  \sum_{\substack{ \vert \abeta \vert \le \gamma, \abeta \ge \aalpha, \\ \abeta - \aalpha \; \text{even}}} 
 c_{\abeta } \; 
 \frac { \abeta !}  {(\abeta - \aalpha + \auno)! 
 \; 2^{\vert \abeta - \aalpha \vert }} \; 
 \boldsymbol{\Delta} \vec{x}^{{\abeta}+ \auno}  
 =
 \aalpha! \, c_{\aalpha}  \boldsymbol{\Delta} \vec{x}^{{\aalpha}} 
 \\ 
 & 
 = \partial_{\aalpha} u\atzero \boldsymbol{\Delta} \vec{x}^{{\aalpha}}
  + \Ogrande(\rho^2), 
 \end{align*}
 with  $\gamma=1 $ if $c_{\abeta} = c_{\abeta}^r, \, r=1, \cdots, 4, \; 
 \gamma=2 $ if $c_{\abeta} = c_{\abeta}^{(2)}. $ 

In order to simplify the notations, we consider square cells, i.e. $h=k$. From Proposition \ref{prop:ind}, see also the Appendix, 
using the expansions \eqref{eq:cellavg2d}, we have

\begin{align*}
I_{NE} & = B_1 + \big( 
u_x\atzerosmall 
u_{xx}\atzerosmall 
+ \tfrac 2 3  u_x\atzerosmall 
u_{xy}\atzerosmall(\azero)  
+ \tfrac 2 3  u_y\atzerosmall 
u_{xy}\atzerosmall 
+ u_y\atzerosmall 
u_{yy}\atzerosmall 
\big)
 h^3 
+ \Ogrande(\rho^4) 
\\
I_{NW}  & = B_1 + 
\big( 
u_x\atzerosmall 
u_{xx}\atzerosmall 
+ \tfrac 2 3  u_x\atzerosmall 
u_{xy}\atzerosmall 
- \tfrac 2 3  u_y\atzerosmall 
u_{xy}\atzerosmall 
+ u_y\atzerosmall 
u_{yy}\atzerosmall 
\big)
h^3 
+ \Ogrande(\rho^4) 
\\
I_{SE}& 
= B_1+ 
\big( 
u_x\atzerosmall 
u_{xx}\atzerosmall 
-\tfrac 23  
u_x\atzerosmall 
u_{xy}\atzerosmall 
+ \tfrac 2 3   
u_y\atzerosmall 
u_{xy}\atzerosmall 
-  u_y\atzerosmall 
u_{yy}\atzerosmall 
\big)
h^3 
+ \Ogrande(\rho^4) 
 \\
I_{SW}& 
= B_1 
+ \big(
-  u_x\atzerosmall 
u_{xx}\atzerosmall 
- \tfrac 2 3  u_x\atzerosmall 
u_{xy}\atzerosmall 
- \tfrac 2 3  
u_y\atzerosmall 
u_{xy}\atzerosmall 
-  u_y\atzerosmall 
u_{yy}\atzerosmall 
\big)
h^3 
+ \Ogrande(\rho^4) 
\\
I_{0} & = B_1 + \Ogrande(\rho^4),
\end{align*}
where $ B_1= (  u^2_x\atzerosmall +  u^2_y\atzerosmall) h^2.$
With easy computation,
we obtain $\theta(\hat{\tau}_3)=4$ with
\begin{equation} \label{eq:tau32d}
\hat{\tau}_{3}= \vert  t I_{NE}^{(1)} + (u-t) I_{NW}^{(1)}+ (u-t) I_{SE}^{(1)}+ tI_{SW}^{(1)} -2u I_{0}^{(2)}  \vert, \qquad \forall t, u \in \mathbb{R}.
\end{equation}
No choices for the coefficients can delete the 11 coefficients of order $\rho^4$.

Finally, we point out that also in 2D we obtain, for the possible choice of $\epsilon$, the same results as in $\mathsf{CWENOZ3}$ in 1D and reported in Table \ref{tab:epsilon}.

\section{Numerical experiments} \label{sec:numerics}

Here we present several numerical tests in order to assess the performance of the schemes proposed in this work. In particular we concentrate on the one-dimensional reconstructions \CWENOZ3 and \CWENOZ5 and on the two-dimensional \CWENOZ3, which are those for which a novel definition of $\hat{\tau}$ has been proposed.

First, in Section~\ref{ssec:recaccuracy} we test the accuracy and non-oscillatory properties of the novel \CWENOZ\ reconstructions.
Next, in Section~\ref{ssec:test1d} we consider one-dimensional test problems and we compare the $\CWENO$ reconstruction of \cite{CPSV:cweno} with the $\CWENOZ$ schemes.
Finally, in Section~\ref{ssec:test2d} we consider two-dimensional test problems based on the system of Euler equations for gas dynamics.

For the numerical solution we apply schemes of order up to $5$: we employ the Local Lax-Friedrichs flux with spatial reconstructions of order $3$ and $5$ and the classical third order strong stability preserving Runge-Kutta scheme with three stages~\cite{Gottlieb:2001} and the fifth order Runge-Kutta scheme with six stages~\cite[\S3.2.5]{Butcher:2008} for the time integration. All the simulations are run with a CFL of $0.45$ and the Local Lax-Friedrichs flux was chosen for simplicity.
All the \CWENO\ and \CWENOZ\ reconstructions employ a central optimal weight $d_0=\tfrac34$ and set the $d_k$ for $k\geq1$ as in  \cite{CPSV:cweno}.

\subsection{Accuracy of the reconstructions}\label{ssec:recaccuracy}
For the accuracy tests we consider the following functions and critical points:
\begin{center}
\begin{tabular}{cll}
$n_{cp}$ & \multicolumn{1}{c}{function} & $x_{\text{crit}}$ \\
0 & $u_0(x)=e^{-x^2}$ & $0.2$\\
1 & $u_1(x)=\sin(\pi x -\sin(\pi x)/\pi)$ & $0.596683186911209$\\
2 & $u_2(x)=1.0+\sin^3(\pi x) $ & $0.0$\\
3 & $u_3(x)=\cos^4(\pi x)$ & $0.2$\\
\end{tabular}
\end{center}
We compute the reconstruction polynomial for the cell containing the critical point, with cell averages initialized with a gaussian quadrature rule of higher order than the expected order of accuracy. For these tests, quadruple precision has been used. 

\begin{figure}
\centering
\includegraphics[width=0.49\linewidth]{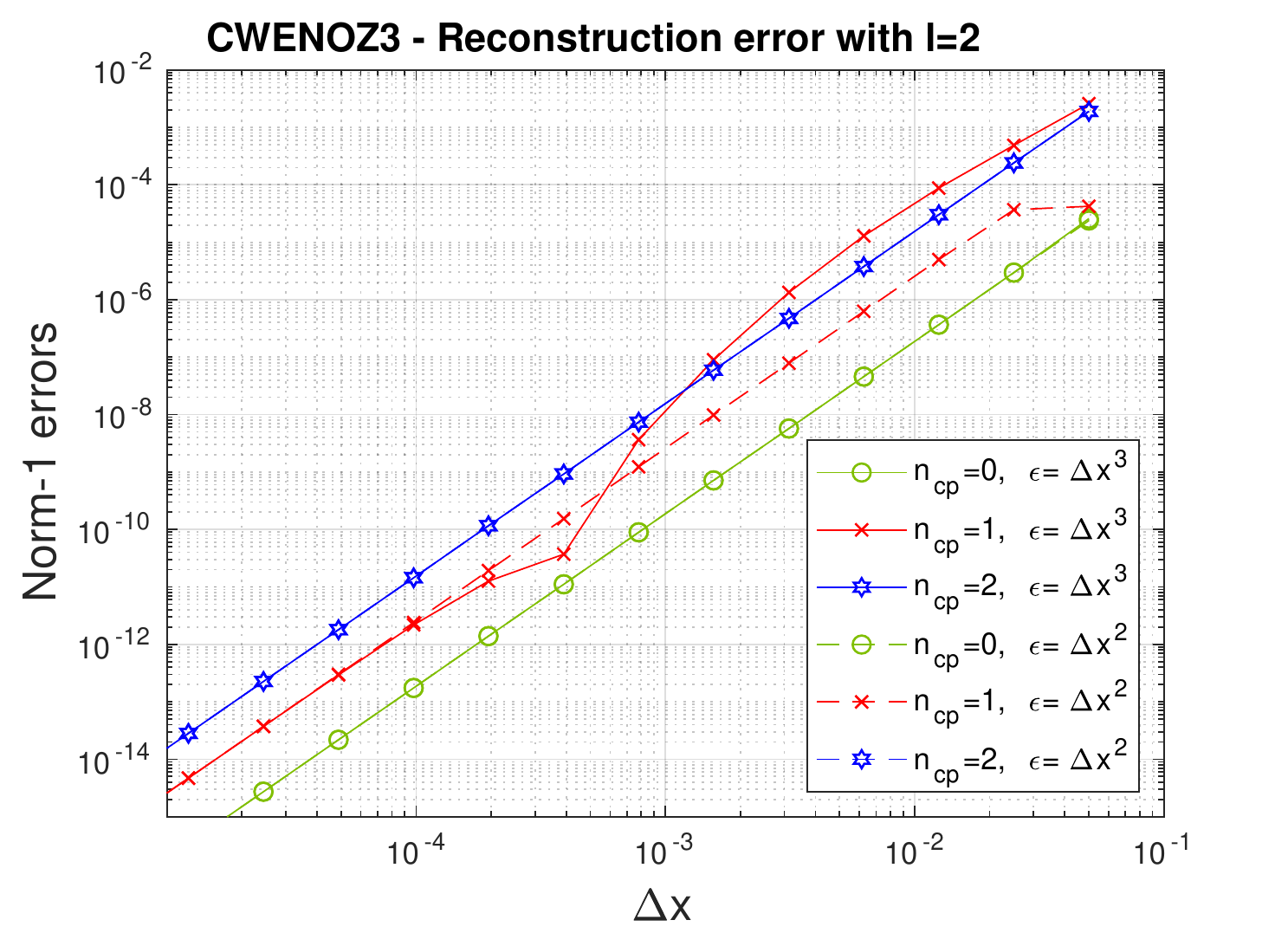}
\includegraphics[width=0.49\linewidth]{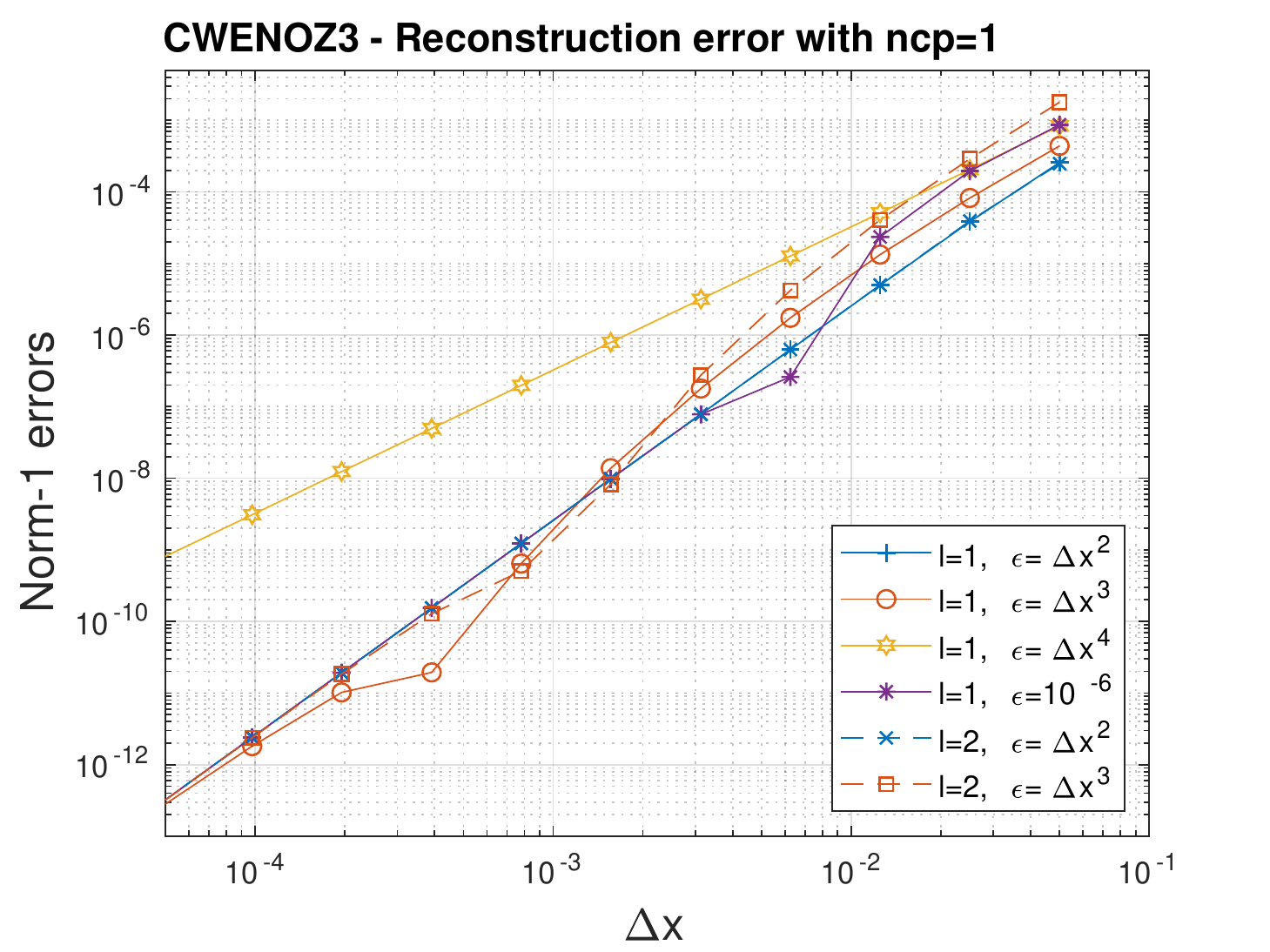}
\\
\includegraphics[width=\linewidth]{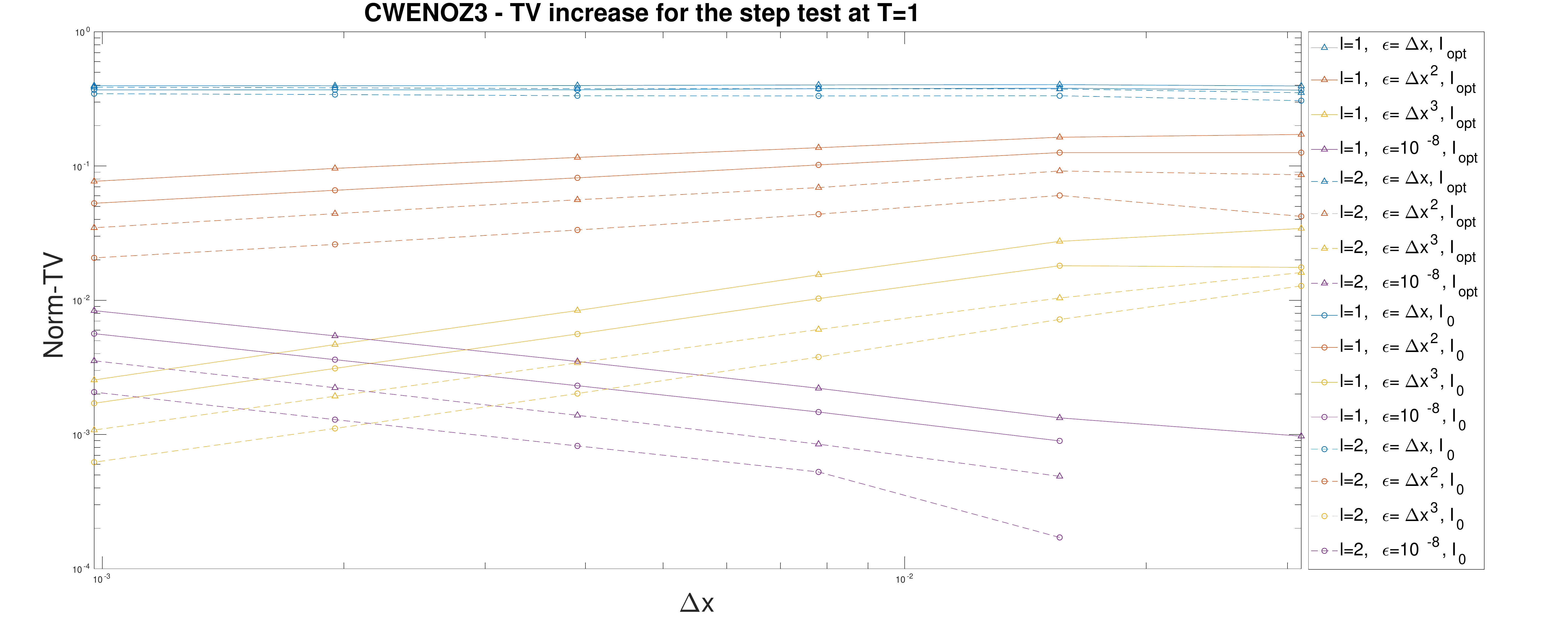}
\caption{Results with the \CWENOZ3\ reconstruction. 
Top left: reconstruction error in a cell with $n_{cp}=0,1,2$, using $\elle=1$. 
Top right: reconstruction error in a cell with $n_{cp}=1$.
Bottom: total variation increase for the linear transport of a step (missing data if a diminution occurred).}
\label{fig:CWZ3:recerr}
\end{figure}

\paragraph{\CWENOZ3}
In Figure~\ref{fig:CWZ3:recerr} we report on numerical experiments with the \CWENOZ3\ reconstruction. 
The top-left panel compares the reconstruction errors in a cell containing a critical point of order from $0$ to $2$, using $\elle=2$ and $I_0=I[\Popt]$. The result indicate that the most difficult situation is the presence of a critical point of order 1, which is further investigated in the top-right panel. For $\epsilon=\Delta x^4$ the reconstruction error does not decay with the correct order, showing that the bound of Theorem \ref{teo:accuracy} is sharp. Also, for $\epsilon=\Delta x^3$, which is just within the bounds  of Theorem \ref{teo:accuracy}, order 3 is indeed reached, but only for very small grids. These results are in agreement with those of Table~\ref{tab:CWZ3:weights}.
A similar irregular convergence history is observed when a fixed value of $\epsilon=10^{-6}$ is employed. Finally, using $\elle=1$ instead of $\elle=2$ yields lower reconstruction errors (top-right) and, on the contrary, defining $I_0=I[P_0]$ yields slightly larger errors.

The bottom panel of Figure~\ref{fig:CWZ3:recerr} analyses the discontinuous case. A double step, namely $u(x,0)=\chi_{[\nicefrac14,\nicefrac34]}(x)$, was evolved with $u_t+u_x=0$ in the domain $[0,1]$ with periodic boundary conditions until $t=1$. The increase in total variation at final time was studied, as it is a measure of the spurious oscillations produced by the numerical scheme. Using a fixed value for $\epsilon$ yields a diminution of the total variation on very coarse grids (missing data in the plot), but an increase on smaller ones, leading to a non-TVB scheme asymptotically; the transition between the regimes happens at a grid size depending on the chosen value for $\epsilon$. All choices $\hat{m}=1,2,3$ lead to TVB schemes, with $\hat{m}=2,3$ also guaranteeing a diminution of the increase of the total variation when the grid size is reduced. Here the parameter $\elle$ and the choice of $I_0$ act oppositely to the case of regular data: a smaller total variation error is obtained for larger values of $\elle$ and using $I_0=I[P_0]$ instead of $I_0=I[\Popt]$. 

Summarizing, it is unfortunately hard to indicate a parameter set that will perform optimally in all circumstances: $\hat{m}=2$ to have good convergence rates on coarse grids for smooth data, together with $\elle=2$ and $I_0=I[P_0]$ to better control the total variation on discontinuous ones seems the best overall choice, but better results can surely be obtained by fine-tuning the reconstruction parameters in specific situations.

\begin{figure}
\centering
\includegraphics[width=0.49\linewidth]{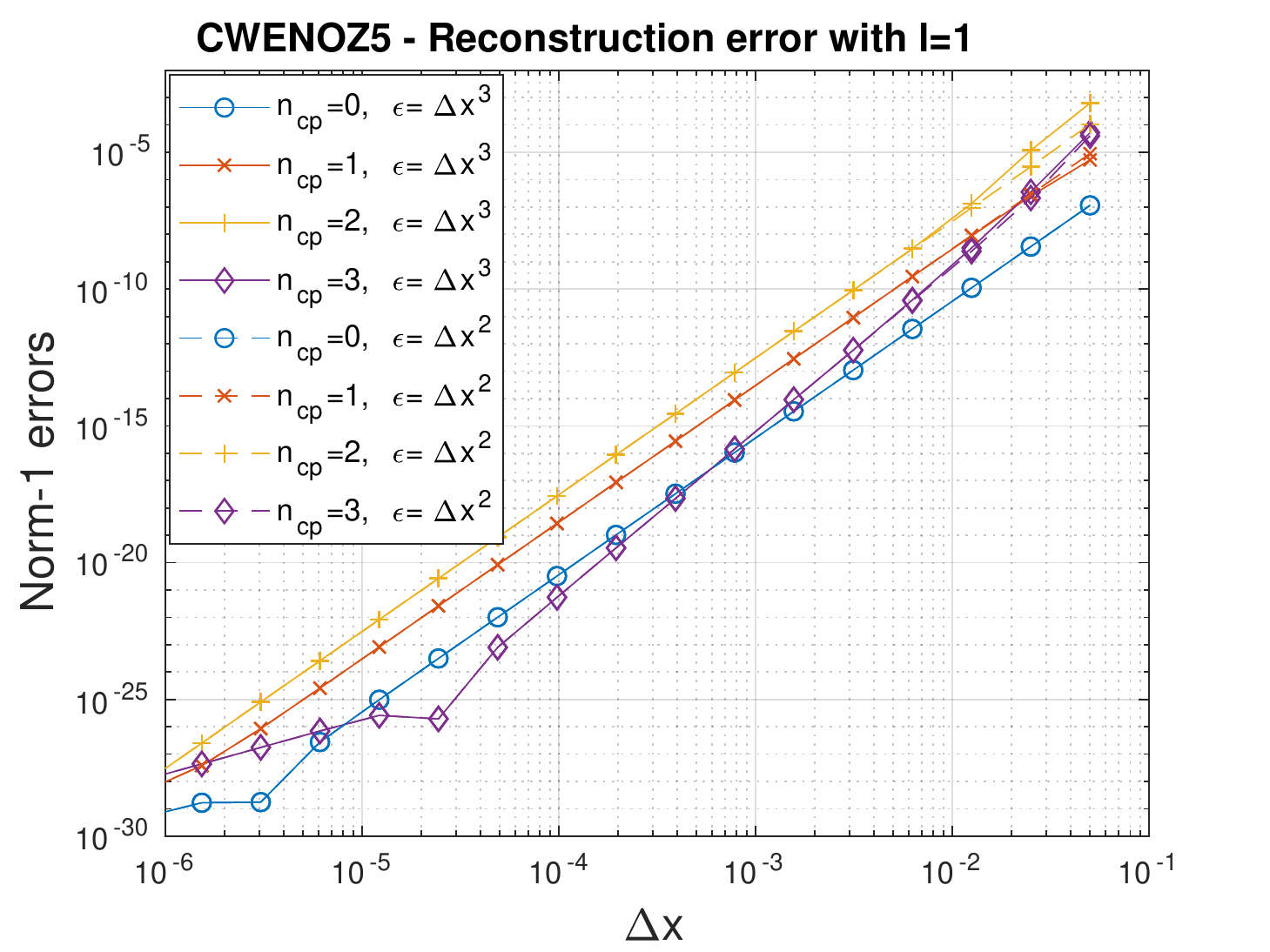}
\includegraphics[width=0.49\linewidth]{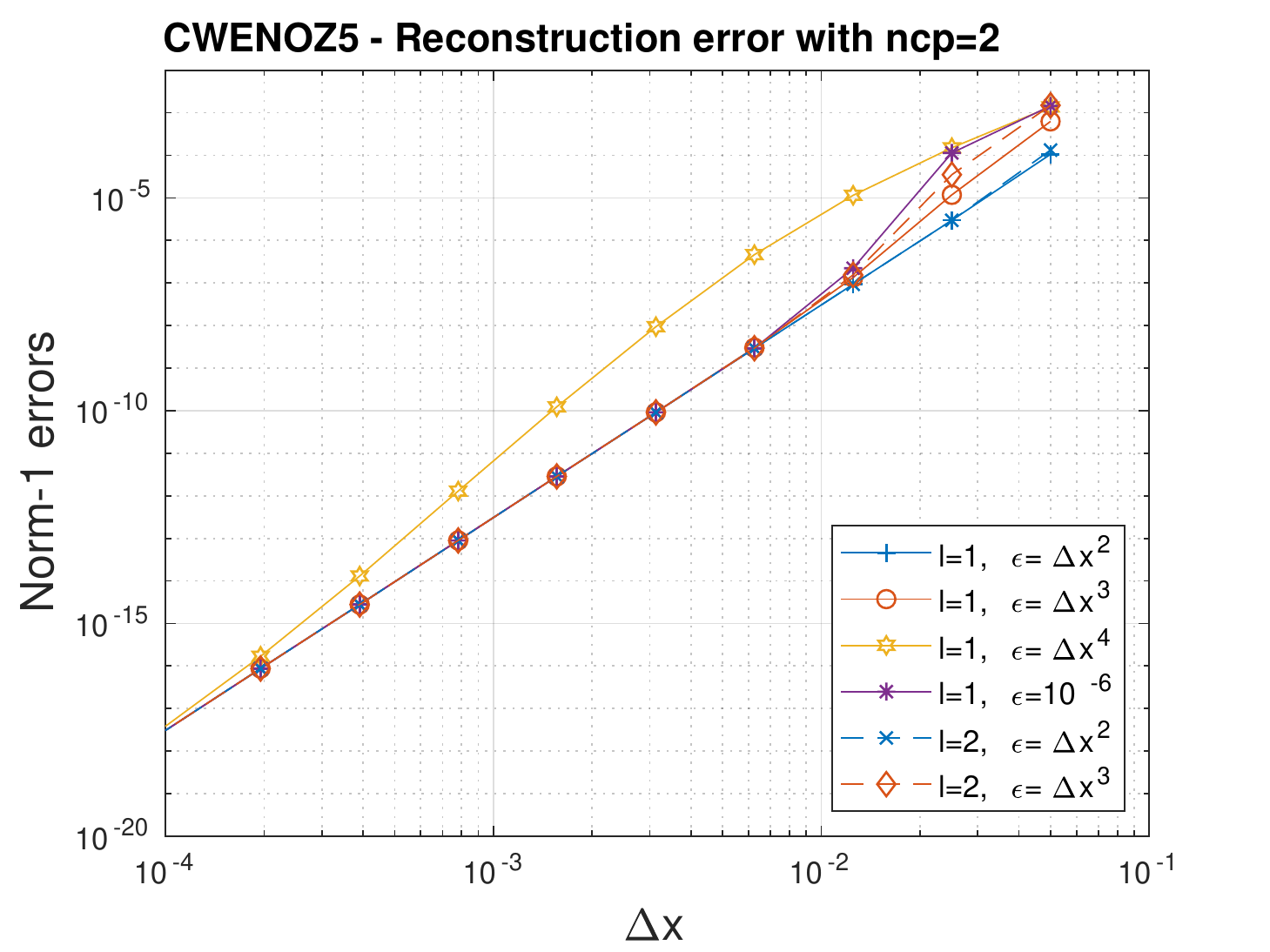}
\\
\includegraphics[width=\linewidth]{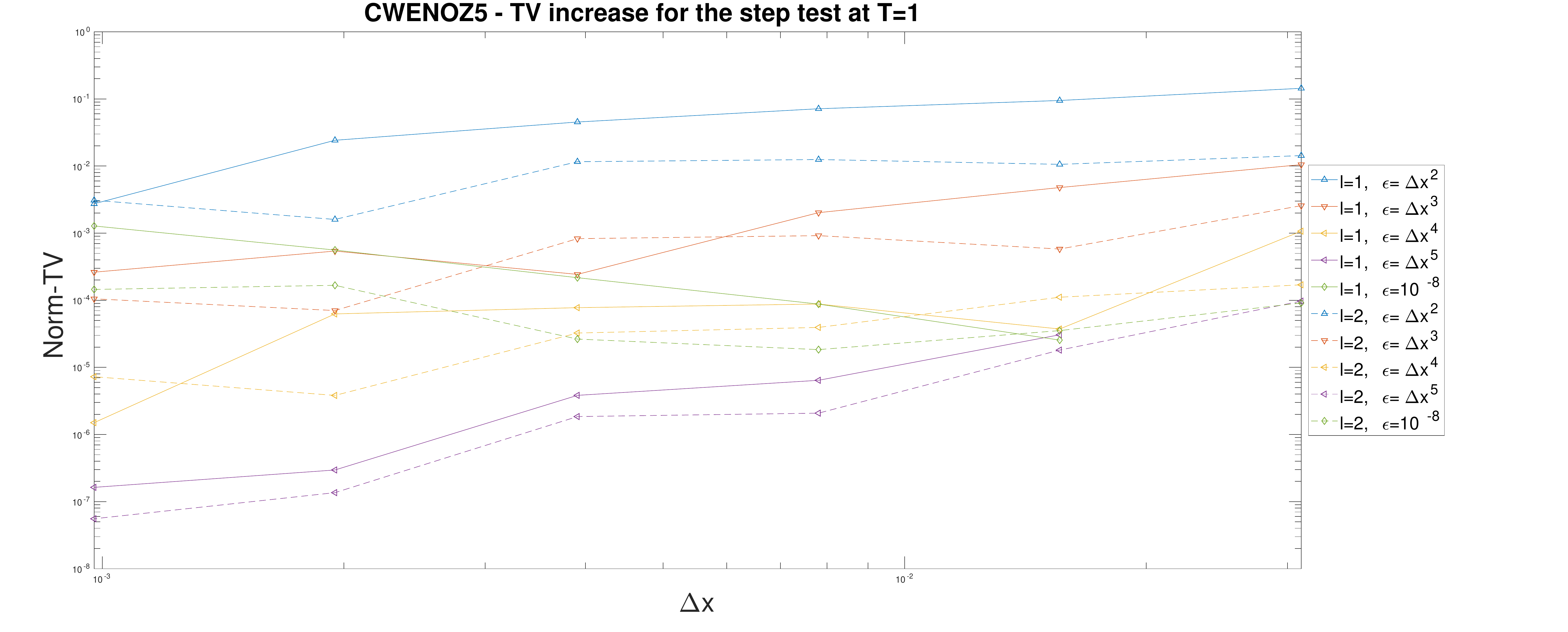}
\caption{Results with the \CWENOZ5\ reconstruction.
Top left: reconstruction error in a cell with $n_{cp}=0,1,2,3$. 
Top right: reconstruction error in a cell with $n_{cp}=2$.
Bottom: total variation increase for the linear transport of a step (missing data if a diminution occurred).}
\label{fig:CWZ5:recerr}
\end{figure}

\paragraph{\CWENOZ5}
In Figure~\ref{fig:CWZ5:recerr} we report on numerical experiments with the \CWENOZ5\ reconstruction. 
The top-left panel compares the errors in a cell containing a  critical point of order from $0$ to $3$ and 
the  top-right one investigates more carefully the more difficult situation, which is $n_{cp}=2$. Apart from the choice $\epsilon=\Delta x^5$, which is just within the bounds  of Theorem \ref{teo:accuracy}, order 5 is reached already on very coarse grids. Also the fixed choice $\epsilon=10^{-6}$ leads to uneven convergence rates on coarse grids. In general, the parameter $\elle$ does not influence significantly the errors.

The bottom panel analyses the discontinuous case.
Similarly to the case of \CWENOZ3, a fixed value of $\epsilon$ leads to an asymptotically non-TVB scheme. Also, the larger is $\hat{m}$, the lower is the total variation error and the higher is the rate at which the total variation increase is reduced when refining the grid. As for \CWENOZ3, using $\elle=1$ reduce the spurious oscillations (bottom panel) and the same happens defining $I_0=I[P_0]$ (not shown).
However, due to the much smaller absolute values of the total variation increase (compare the vertical scale with Figure \ref{fig:CWZ3:recerr}), we expect that spurious oscillations will be very small with any parameter set for \CWENOZ5.

Summarizing, we suggest to employ $\hat{m}=3$ or $\hat{m}=4$ for $\CWENOZ5$. Since $\elle$ does not play a major role, we suggest to take $\elle=1$, which is computationally cheaper.

\begin{figure} 
	\centering
\includegraphics[width=0.49\linewidth]{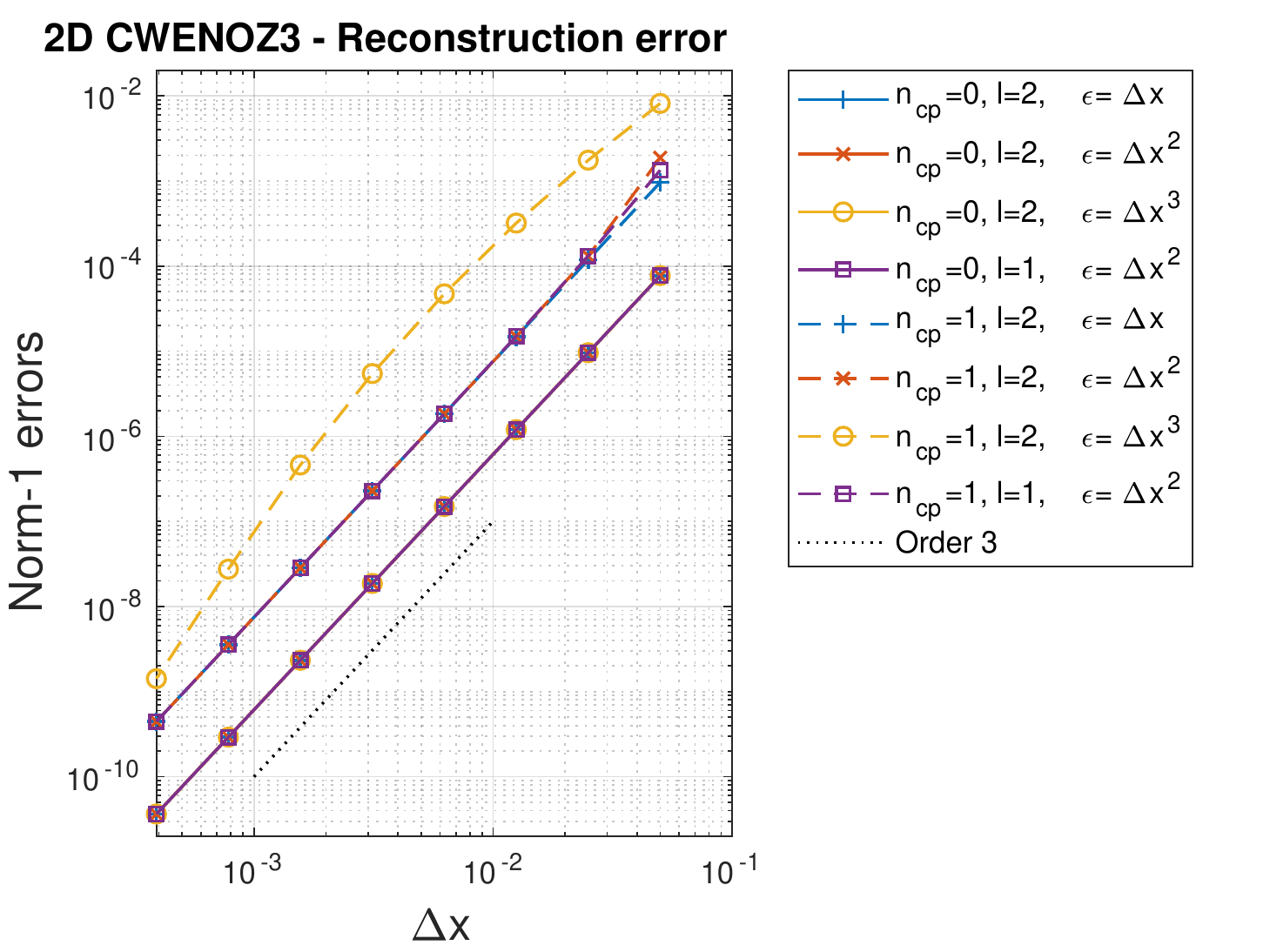}
\caption{Reconstruction errors for the \CWENOZ3\ reconstruction in 2D.}
\label{fig:CWZ32D:recerr}
\end{figure}

\paragraph{\CWENOZ3\ in 2D}
Figure~\ref{fig:CWZ32D:recerr} shows the behaviour of the \CWENOZ3\ reconstruction in two space dimensions. In particular, in the left panel one can see that the parameters $\elle$ and $\hat{m}$ do not influence significantly the performance of \CWENOZ3\ in a region with $n_{cp}=0$, but that the performance is reduced on coarse grids for $\hat{m}=3$ and $n_{cp}=1$. The situation is quite similar to the one-dimensional case. The right panel shows that the reconstruction errors are smaller for \CWENOZ3\ than \CWENO3.
In the sequel, $\elle=2$ and $\hat{m}=2$ have been used in all tests.

\subsection{Conservation laws in one space dimension} \label{ssec:test1d}
In order to distinguish them, here we name $\CWENOZDB$ the schemes using, as in~\cite{CPSV:coolweno}, the weights of Borges et al.~\cite{BCCD:2008:wenoz5} and simply $\CWENOZ$ the schemes using the new improved weights introduced in Section~\ref{ssec:tau1d} and namely $\hat\tau_3$ of equation \eqref{eq:tau3} and $\hat\tau_5$ of equation \eqref{eq:tau5}, with $t=1$.

Here, the value of $\epsilon$ is chosen as $\epsilon\approx \Delta x^{\hat{m}}$ for all the schemes.
In view of the convergence analysis of the nonlinear weights in Section~\ref{ssec:tau1d} and of the results of Section~\ref{ssec:recaccuracy}, we consider the following choices.

For $\CWENOZ3$ we show results for  $\hat{m}=2,3$, $\elle=1,2$ and $I_0=I[\Popt]$ or $I_0=I[P_0]$, showing that the best parameter set depends on the problem setting. Unless otherwise stated, the results are shown for $I_0=I[\Popt]$.
For the $\CWENOZ5$ scheme%
, we have tested $\hat{m}=3,4$, $\elle=1,2$ and both choices of $I_0$, but the schemes are almost insensitive to the parameters. Unless otherwise stated, the results are shown for $\hat{m}=3$, $\elle=1$ and $I_0=I[\Popt]$.
Finally, for $\CWENO$ and $\CWENOZDB$ schemes, we report the results with the settings of the papers  \cite{CPSV:coolweno,CPSV:cweno,SCR:CWENOquadtree}, namely  $\hat{m}=2$, $\elle=2$ and $I_0=I[P_0]$.

\paragraph{Linear transport of smooth data.}
We solve the linear scalar conservation law
\begin{equation} \label{eq:linadv}
	u_t+u_x=0
\end{equation}
on the periodic domain $x\in[-0.5,0.5]$ and up to final time $T=1$. As initial condition we consider the low frequency sinusoidal profile
\begin{equation} \label{eq:lfDatum}
	u_0(x) = \sin\left( 2\pi x \right)
\end{equation}
and the high frequency profile
\begin{equation} \label{eq:hfDatum}
	u_0(x) = \sin(2 \pi x) + \sin(30 \pi x)\exp(-80 x^2).
\end{equation}
The goal of this test is to numerically verify the convergence properties of the schemes.


\begin{table}
	\begin{center}
		\caption{The accuracy for the linear transport test of the low frequency smooth data~\eqref{eq:lfDatum} with schemes of order $3$ and $5$.}
		\label{tab:lfErrors}
		\tiny
		\setlength{\tabcolsep}{3pt}
		\begin{tabular}{|r|rr|rr|rr|rr|rr|rr|} \hline
			& \multicolumn{2}{c|}{$\CWENO 3$} & \multicolumn{2}{c|}{$\CWENOZDB 3$} & \multicolumn{2}{c|}{$\CWENOZ 3$}
			& \multicolumn{2}{c|}{$\CWENO 5$} & \multicolumn{2}{c|}{$\CWENOZDB 5$} & \multicolumn{2}{c|}{$\CWENOZ 5$}
			\\
			& & & & & \multicolumn{2}{c|}{($l=2$, $\hat{m}=2$)} & & & & & \multicolumn{2}{c|}{($l=1$, $\hat{m}=3$)}
			\\
			cells & error & rate & error & rate & error & rate
			& error & rate & error & rate & error & rate
			\\
			\hline
			50 & 6.49e-03 & - & 1.42e-03 & - & 7.44e-04 & - 
			
			& 6.41e-06 & - & 2.08e-06 & - & 2.08e-06 & -\\
			
			100 & 7.87e-04 & 3.04 & 1.00e-04 & 3.82 & 8.64e-05 & 3.10
			
			&2.02e-07 & 4.99 & 6.52e-08 & 4.99 & 6.52e-08 & 5.00\\
			
			200 & 9.46e-05 & 3.06 & 1.08e-05 & 3.21 & 1.08e-05 & 3.00 
			
			& 6.32e-09 & 4.99 & 2.04e-09 & 5.00 & 2.04e-09 & 5.00\\
			
			400 & 1.13e-05 & 3.07 & 1.35e-06 & 3.00 & 1.35e-06 & 3.00
			
			& 1.97e-10 & 5.00 & 6.37e-11 & 5.00 & 6.37e-11 & 5.00\\
			
			800 & 1.38e-06 & 3.03 & 1.69e-07 & 3.00 & 1.69e-07 & 3.00
			
			& 6.17e-12 & 5.00 & 1.99e-12 & 5.00 & 1.99e-12 & 5.00\\
			
			1600 & 1.71e-07 & 3.01 & 2.11e-08 & 3.00 & 2.11e-08 & 3.00
			
			& 2.29e-13 & 4.75 & 1.40e-13 & 3.83 & 1.40e-13 & 3.83\\ \hline
		\end{tabular}
	\end{center}
\end{table}

For the low frequency datum~\eqref{eq:lfDatum} the $\CWENOZ$ scheme and the $\CWENOZDB$ scheme have comparable errors and they are always more accurate than the $\CWENO$ scheme. For the $\CWENOZ$ scheme of order $3$ we do not see particular advantages in using different combinations of $\hat{m}=2,3$ and $\elle=1,2$.
The best choice is given by $\elle=2$ and $\hat{m}=2$, for which we observe smaller errors than $\CWENOZDB$ on coarser grids. In Table~\ref{tab:lfErrors} we report the $1$-norm errors and the convergence rates as functions of increasing number of cells for the schemes of order $3$ and $5$.

Table~\ref{tab:hfErrors} reports the data for the linear transport of the high frequency datum~\eqref{eq:hfDatum}. 
Also in this case the errors of the $\CWENOZ5$ are not strongly influenced by the choice of the parameters $\elle$ and $\hat{m}$. Moreover, the $\CWENOZ5$ scheme and the $\CWENOZDB5$ scheme provide the same results showing improvements with respect to $\CWENO5$. 
On the other hand,$\CWENOZ3$ is more sensible to the choice of the parameters $\elle$ and $\hat{m}$ as seen in Figure~\ref{fig:hfConvergence}, where we compare $\elle=1,2$ and $\hat{m}=2,3$. 
We observe that, on this smooth problem, taking $\elle=2$ and $\hat{m}=2$ helps the scheme to reach the theoretical order of convergence also on coarser grids, according to analysis and tests in Section~\ref{ssec:tau1d} and Section~\ref{ssec:recaccuracy}. All the methods reach the expected theoretical order of convergence, but we point out that the accuracy of the $\CWENOZ3$ scheme with this optimal choice of the parameters (in particular for the $\elle=1$ and $\hat{m}=2$) is more than half an order of magnitude better than the accuracy of $\CWENOZDB3$ and $\CWENO3$. This also demonstrates that the new improved weights computed in Section~\ref{ssec:tau1d} for one dimensional reconstructions improve the quality of the solution.

\begin{figure} 
	\centering
	\includegraphics[width=0.7\textwidth]{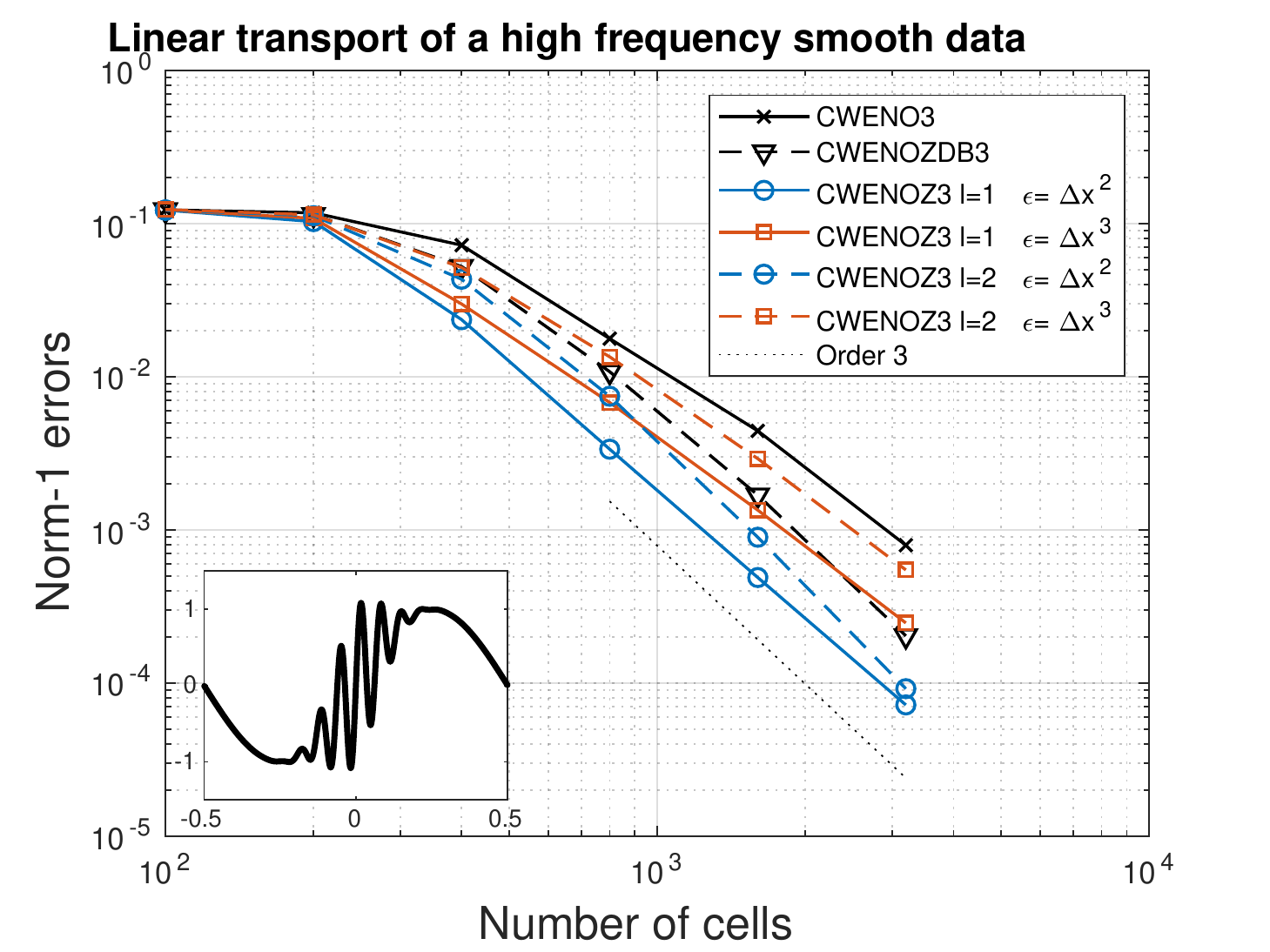}
	\caption{Convergence plot of the $\CWENO$, $\CWENOZDB$ and $\CWENOZ$ schemes of order $3$ for the high frequency datum~\eqref{eq:hfDatum}.} 
	\label{fig:hfConvergence}
\end{figure}

\begin{table}
	\begin{center}
		\caption{The accuracy for the linear transport test of the high frequency smooth data~\eqref{eq:hfDatum} with schemes of order $3$ and $5$.}
		\label{tab:hfErrors}
		\tiny
		\setlength{\tabcolsep}{3pt}
		\begin{tabular}{|r|rr|rr|rr|rr|rr|rr|} \hline
			& \multicolumn{2}{c|}{$\CWENO 3$} & \multicolumn{2}{c|}{$\CWENOZDB 3$} & \multicolumn{2}{c|}{$\CWENOZ 3$}
			& \multicolumn{2}{c|}{$\CWENO 5$} & \multicolumn{2}{c|}{$\CWENOZDB 5$} & \multicolumn{2}{c|}{$\CWENOZ 5$}
			\\
			& & & & & \multicolumn{2}{c|}{($\elle=2$, $\hat{m}=2$)} & & & & & \multicolumn{2}{c|}{($\elle=1$, $\hat{m}=3$)}
			\\
			cells & error & rate & error & rate & error & rate
			& error & rate & error & rate & error & rate
			\\
			\hline
			100 & 1.23e-01 & - & 1.22e-01 & - & 1.22e-01 & - 
			
			& 1.16e-01 & - & 1.05e-01 & - & 9.53e-02 & -\\
			
			200 & 1.17e-01 & 0.06 & 1.14e-01 & 0.10 & 1.12e-01 & 0.12 
			
			& 1.43e-02 & 3.03 & 6.02e-03 & 4.11 & 6.00e-03 & 4.00\\
			
			400 & 7.21e-02 & 0.70 & 5.23e-02 & 1.12 & 4.32e-02 & 1.38 
			
			& 6.08e-04 & 4.55 & 1.92e-04 & 4.97 & 1.93e-04 & 4.96\\
			
			800 & 1.77e-02 & 2.03 & 1.08e-02 & 2.28 & 7.45e-03 & 2.54 
			
			& 2.16e-05 & 4.82 & 6.02e-06 & 5.00 &  6.03e-06 & 5.00\\
			
			1600 & 4.43e-03 & 2.00 & 1.69e-03 & 2.68 & 8.95e-04 & 3.06 
			
			& 7.51e-07 & 4.85 & 1.88e-07 & 5.00 & 1.88e-07 & 5.00\\
			
			3200 & 7.93e-04 & 2.45 & 2.03e-04 & 3.05 & 9.18e-05 & 3.28
			
			& 2.73e-08 & 4.78 & 5.89e-09 & 5.00 & 5.89e-09 & 5.00\\ \hline
		\end{tabular}
	\end{center}
\end{table}

\paragraph{Linear transport of a non-smooth datum: the Jiang-Shu test.}
We again consider the linear scalar conservation law~\eqref{eq:linadv} but on the periodic domain $x\in[-1,1]$ and up to final time $T=8$.  Now we consider the non-smooth initial datum
\begin{subequations}\label{eq:jiangshu}
\begin{equation} 
	u_0(x) =
	\begin{cases}
		\frac16 \left( G(x,\beta,z-\delta) + G(x,\beta,z+\delta) + 4G(x,\beta,z) \right), & -0.8 \leq x \leq -0.6,\\
		1, & -0.4 \leq x \leq -0.2,\\
		1-\left|10(x-0.1)\right|, & 0 \leq x \leq 0.2,\\
		\frac16 \left( F(x,\alpha,a-\delta) + F(x,\alpha,a+\delta) + 4F(x,\alpha,a) \right), & 0.4 \leq x \leq 0.6,\\
		0, & \text{otherwise}		
	\end{cases}
\end{equation}
where
\begin{equation} 
	G(x,\beta,z) = \exp(-\beta(x-z)^2), \quad F(x,\alpha,a) = \sqrt{\max\{1 - \alpha^2(x - a)^2,0\}}
\end{equation}
\end{subequations}
and the constants are taken as $a = 0.5$, $z = -0.7$, $\delta = 0.005$, $\alpha = 10$ and $\beta = \log 2/36 \delta^2$. This problem, designed by Jiang and Shu in~\cite{JiangShu:96}, is used in order to investigate the properties of a scheme to transport different shapes with minimal dissipation and dispersion effects. The initial condition~\eqref{eq:jiangshu} is a combination of smooth and non-smooth shapes: precisely, from the left to the right side of the domain, we have a Gaussian, a square wave, a sharp triangle wave and a half ellipse.

\begin{figure}
	\centering
	\includegraphics[width=\linewidth]{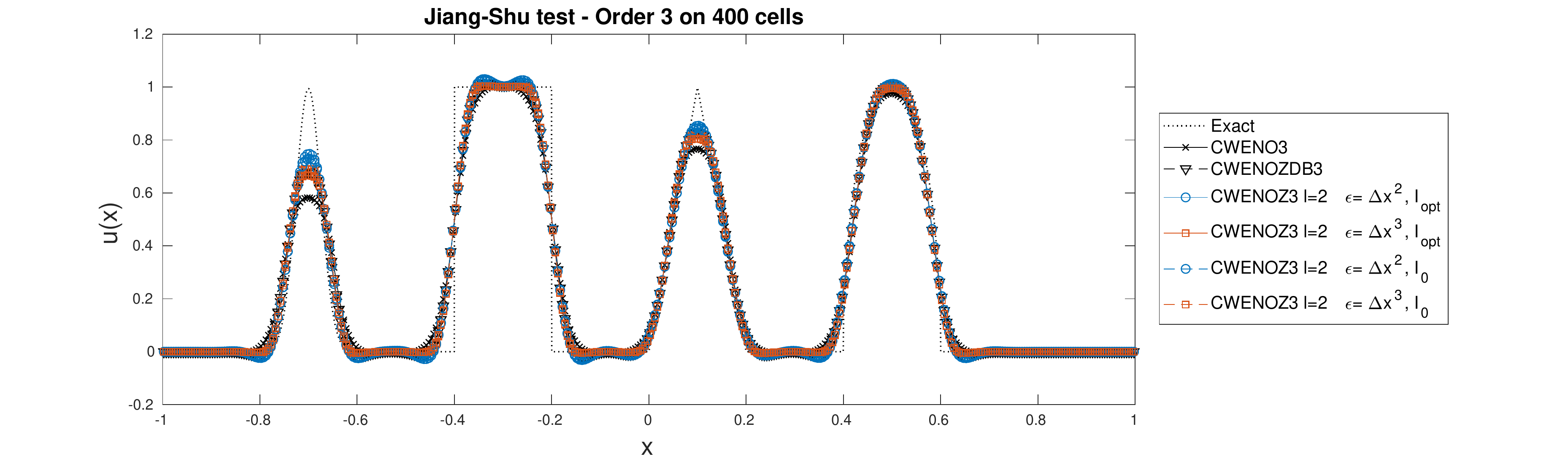}
	\\
	\includegraphics[width=0.49\linewidth]{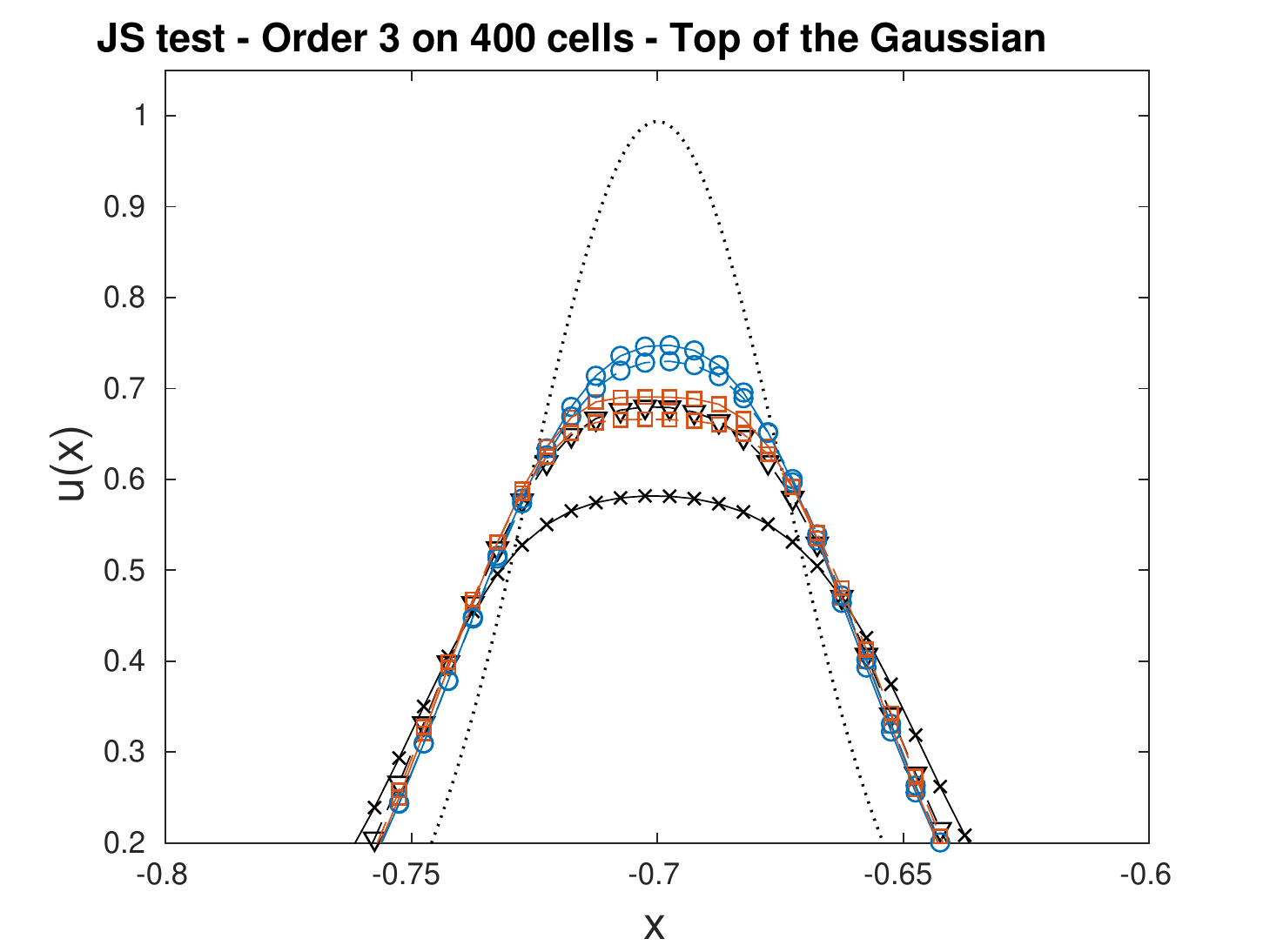}
	\includegraphics[width=0.49\linewidth]{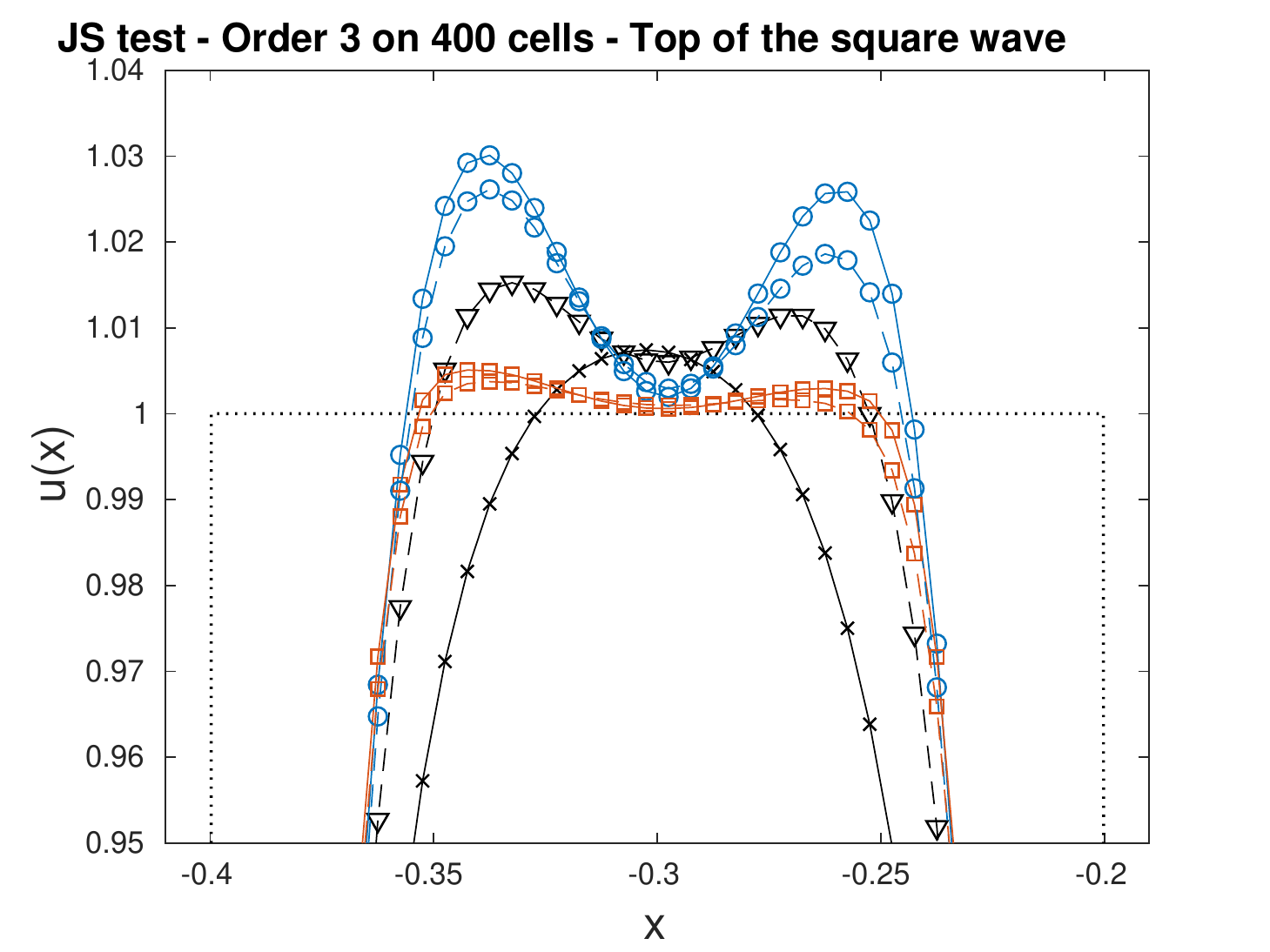}
	\caption{Top panel: numerical solution of the Jiang-Shu test problem~\eqref{eq:jiangshu} with schemes of order $3$ and $400$ cells. Bottom panels: top part of the Gaussian wave (left) and top part of the square wave (right).}
	\label{fig:solutionJiangShu3}
\end{figure}

\begin{figure}
	\centering
	\includegraphics[width=\linewidth]{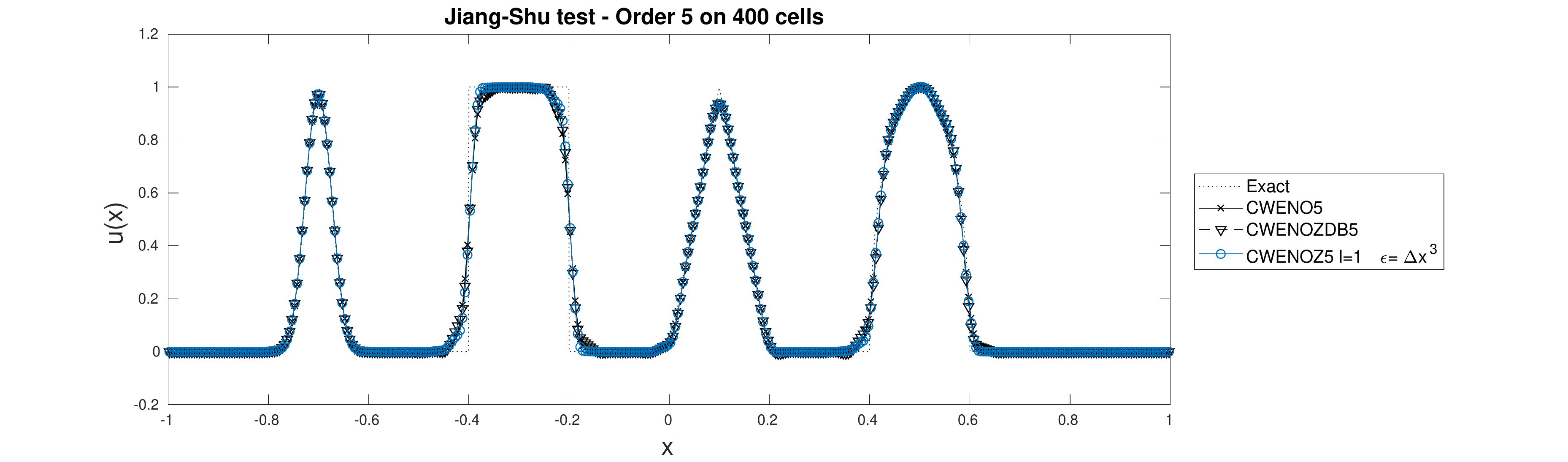}
	\\
	\includegraphics[width=0.49\linewidth]{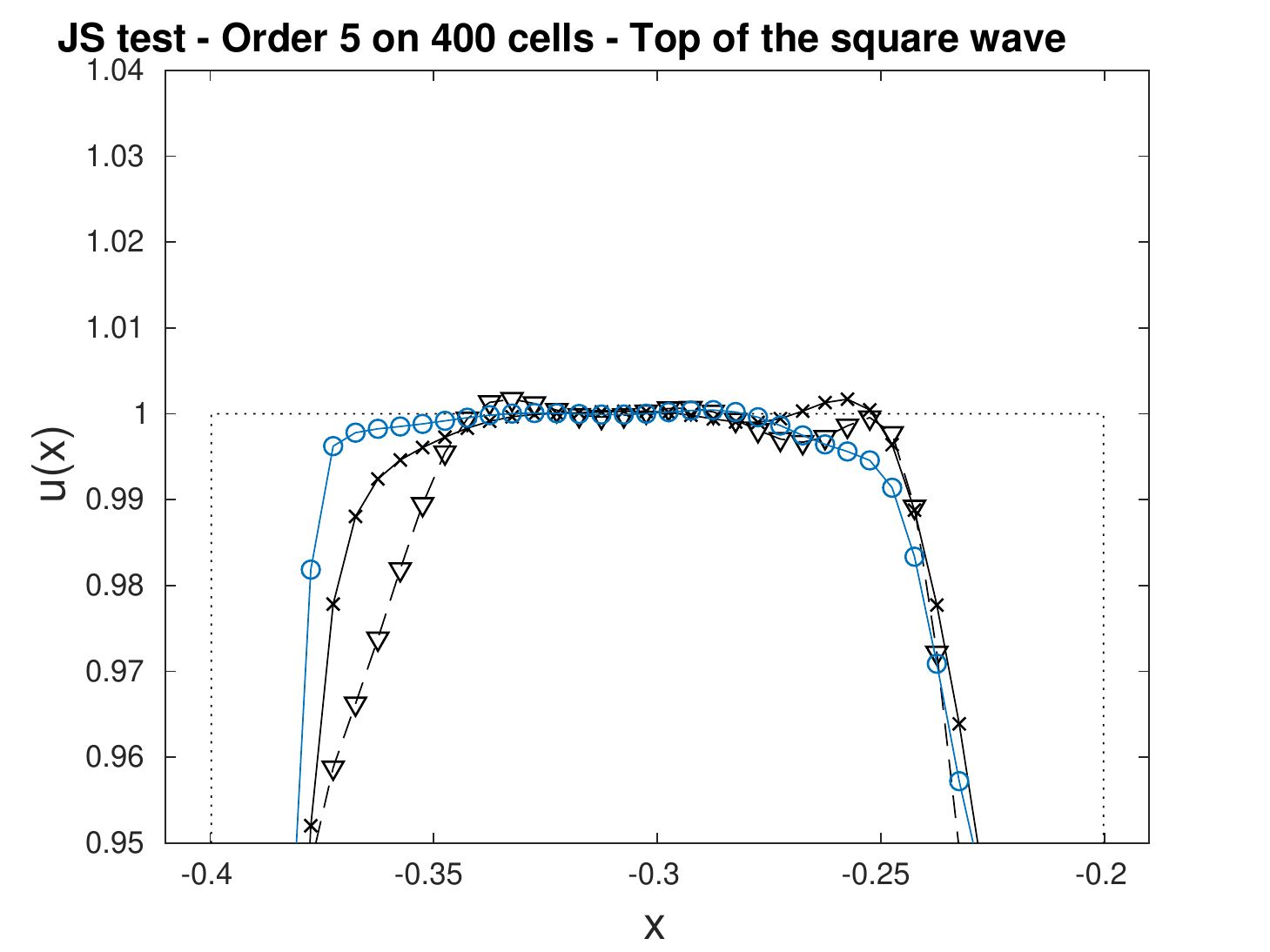}
	\includegraphics[width=0.49\linewidth]{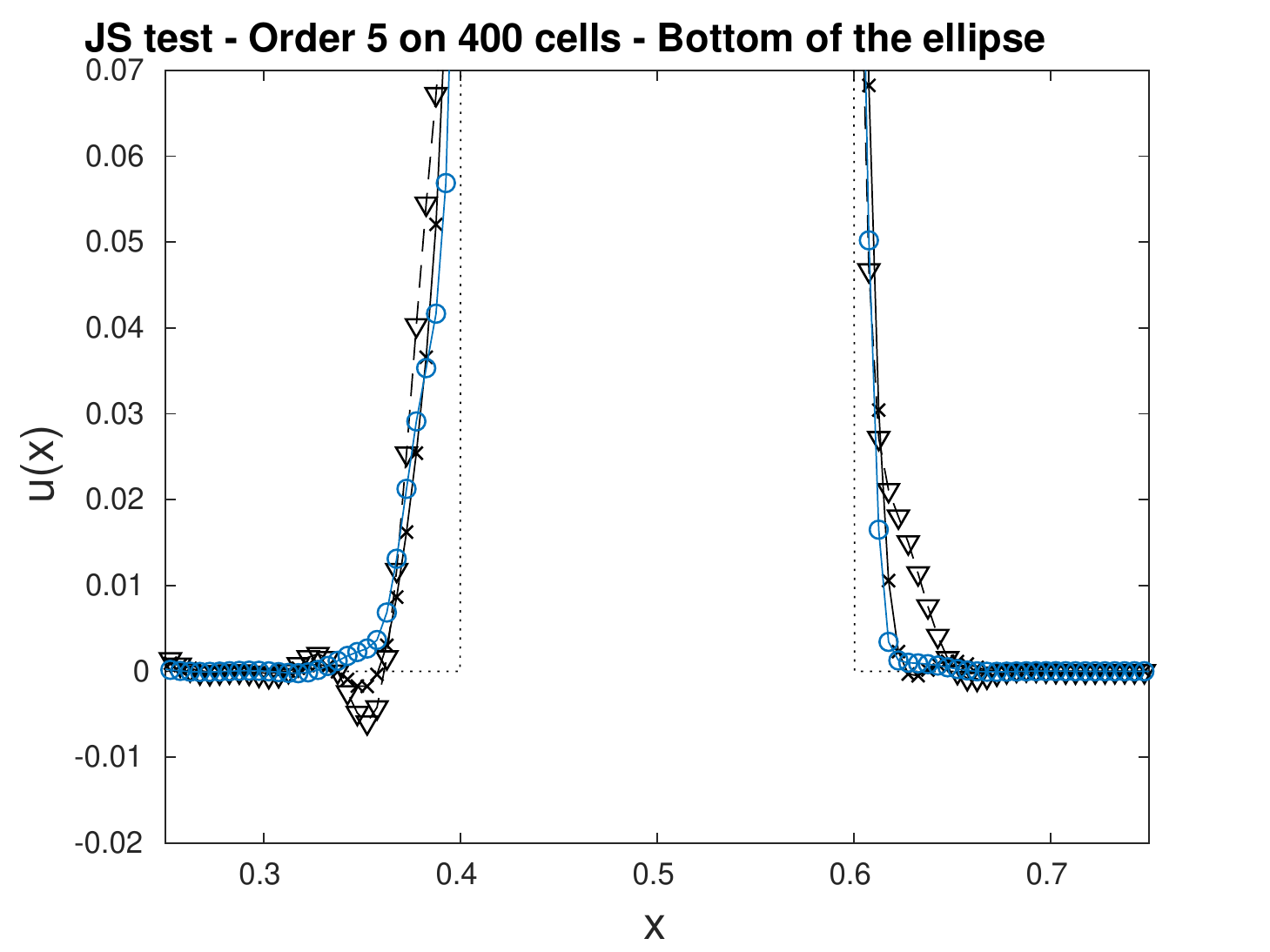}
	\caption{Top panel: numerical solution of the Jiang-Shu test problem~\eqref{eq:jiangshu} with schemes of order $5$ and $400$ cells. Bottom panels: top part of the square wave (left) and bottom part of the half ellipse (right).}
	\label{fig:solutionJiangShu5}
\end{figure}

The top panel of Figure~\ref{fig:solutionJiangShu3} shows the numerical solution of the Jiang and Shu test problem computed with the $\CWENO3$, $\CWENOZDB3$ and $\CWENOZ3$ schemes on $400$ cells. For the $\CWENOZ3$ schemes we compare the solutions for the parameters $\hat{m}=2,3$ with fixed $\elle=2$ since the choice $\elle=1$ leads to more oscillating schemes around the discontinuities of the square wave (see \S\ref{ssec:recaccuracy}). 
Moreover, for the same sets of parameters, we also consider the $\CWENOZ3$ schemes employing the indicator $I_0=I[P_0]$. 

In the bottom panels of Figure~\ref{fig:solutionJiangShu3} we focus on the top parts of the Gaussian wave and of the square waves since they give information on the behaviour of the schemes on smooth and non-smooth zones of the solution. 
In the case of a smooth profile, we observe that $\CWENOZ3$ and $\CWENOZDB3$ are less diffusive than $\CWENO3$. In particular, the $\CWENOZ3$ schemes with $\hat{m}=2$ provide a better approximation of the top of the Gaussian wave. We do not observe a significant difference in using $I_0=I[P_0]$. 
The situation for the top of the square wave (bottom right panel of Figure~\ref{fig:solutionJiangShu3}) the situation is more complex. The $\CWENOZ3$ scheme with $\hat{m}=2$ and the $\CWENOZDB3$ scheme seem to be more oscillating than $\CWENO3$. However, the choice of the parameter $\hat{m}=3$ for the $\CWENOZ3$ scheme allows to damp the spurious oscillations across the discontinuities and it also provides a less diffusive approximation than $\CWENO3$. In general, we also observe that using $I_0=I[P_0]$ for the $\CWENOZ3$ schemes mitigates the amplitude of the oscillations.

In Figure~\ref{fig:solutionJiangShu5} we show the numerical solution computed with the $\CWENO5$, the $\CWENOZDB5$ and the $\CWENOZ5$ schemes on $400$ cells. In this case, we consider only the parameters $\elle=1$ and $\hat{m}=3$ for the $\CWENOZ5$ scheme. All the schemes seem to accurately reproduce the solution at final time with significant improvements with respect to the schemes of order $3$. However, some zones deserve more attention since they show that the $\CWENOZ5$ scheme provides a better approximation: they are considered in the bottom panels of Figure~\ref{fig:solutionJiangShu5}. We show the top part of the square wave (left panel) and the bottom part of the half ellipse (right panel). We notice that the $\CWENO5$ scheme exhibits undershoots. On the contrary, the $\CWENOZDB5$ and $\CWENOZ5$ schemes avoid the oscillations. In particular, the $\CWENOZ5$ scheme has a slightly better resolution with less diffusivity close to the discontinuities showing that the new weights designed in Section~\ref{ssec:tau1d} improve the accuracy.

\paragraph{One-dimensional Euler equations: the shock-acoustic interaction test}
We consider the one-dimensional system of Euler equations for gas dynamics
\[ \partial_t \left( \begin{array}{c}
\rho \\ \rho u \\ E
\end{array}\right) +
\partial_x \left( \begin{array}{c}
\rho u \\ \rho u^2 + p \\ u(E+p)
\end{array}\right)  = 0,
\]
where $\rho$, $u$, $p$ and $E$ are the density, velocity, pressure and energy per unit volume of an ideal gas, whose equation of state is
$ E = \frac{p}{\gamma-1} + \frac12 \rho u^2, $
where $\gamma = 1.4$.

In this test we simulate the interaction of a strong shock with an acoustic wave on the domain $x\in[-5,5]$ with free-flow boundary conditions. The problem was introduced by Shu and Osher in~\cite{ShuOsher:89} and is characterized by a Mach 3 shock wave interacting with a standing sinusoidal density wave. The solution, behind the main strong shock, develops a combination of smooth waves and small discontinuities. The initial condition is
$$
	(\rho,u,p) = \begin{cases}
		(3.857143, 2.629369, 10.333333), & x < -4\\
		(1+0.2\sin(5x),0,1), & x \geq -4
	\end{cases}
$$
and we run the problem up to the final time $T=1.8$.

\begin{figure}
	\centering
	\includegraphics[width=\linewidth]{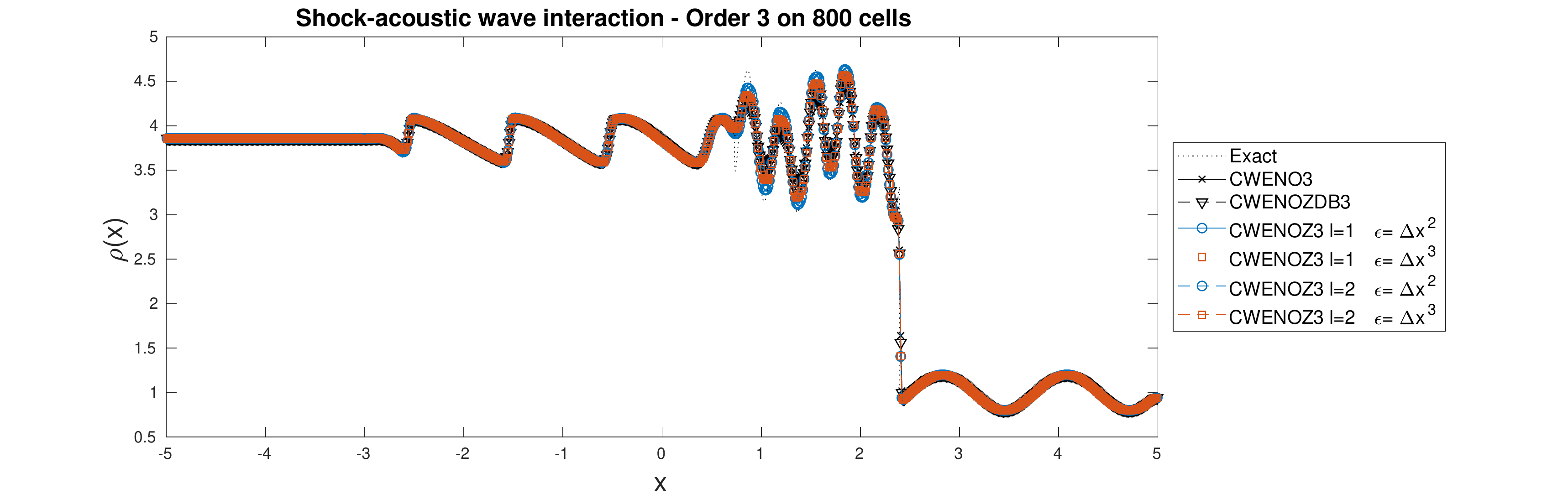}
	\\
	\includegraphics[width=0.49\linewidth]{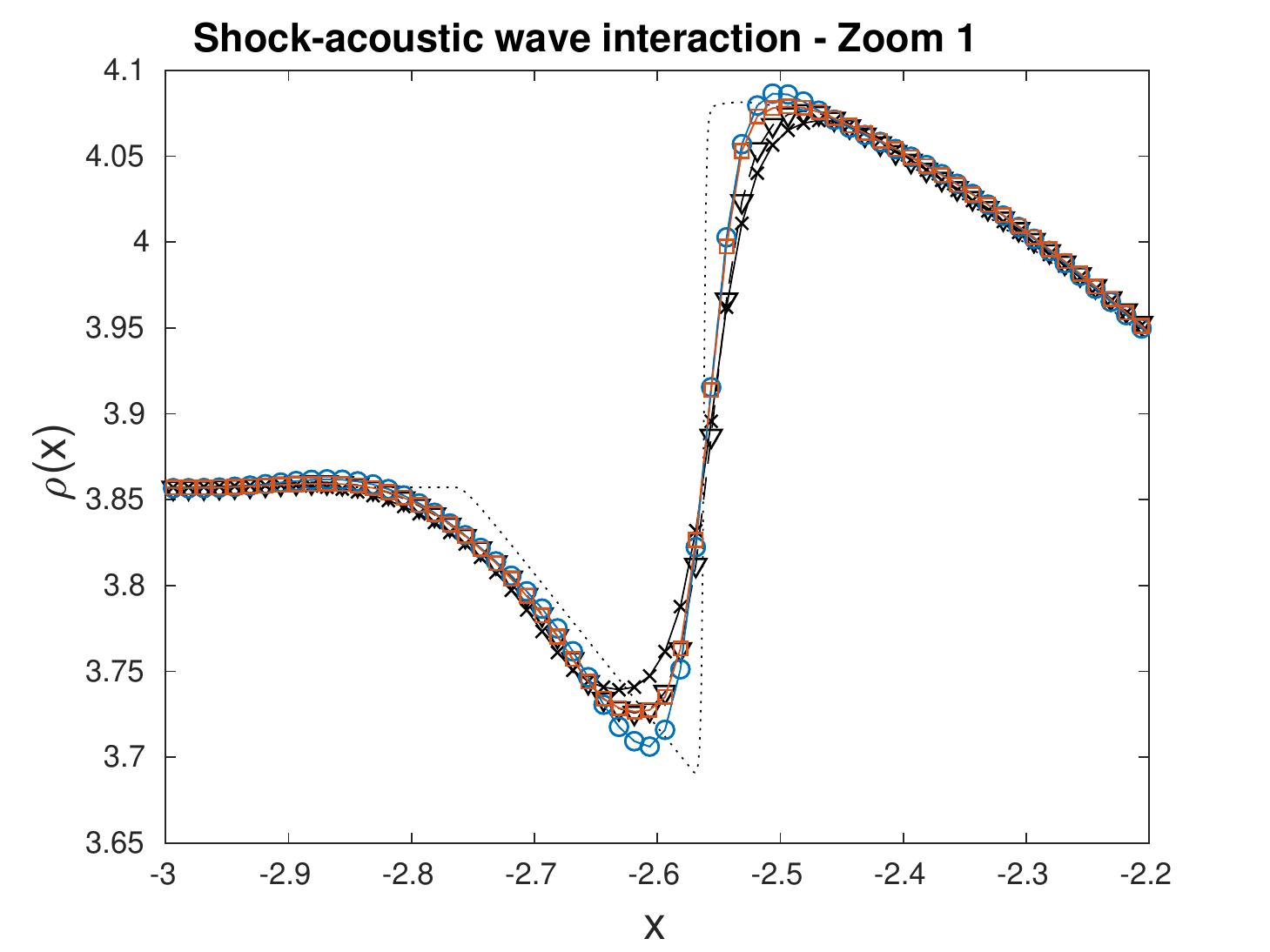}
	\includegraphics[width=0.49\linewidth]{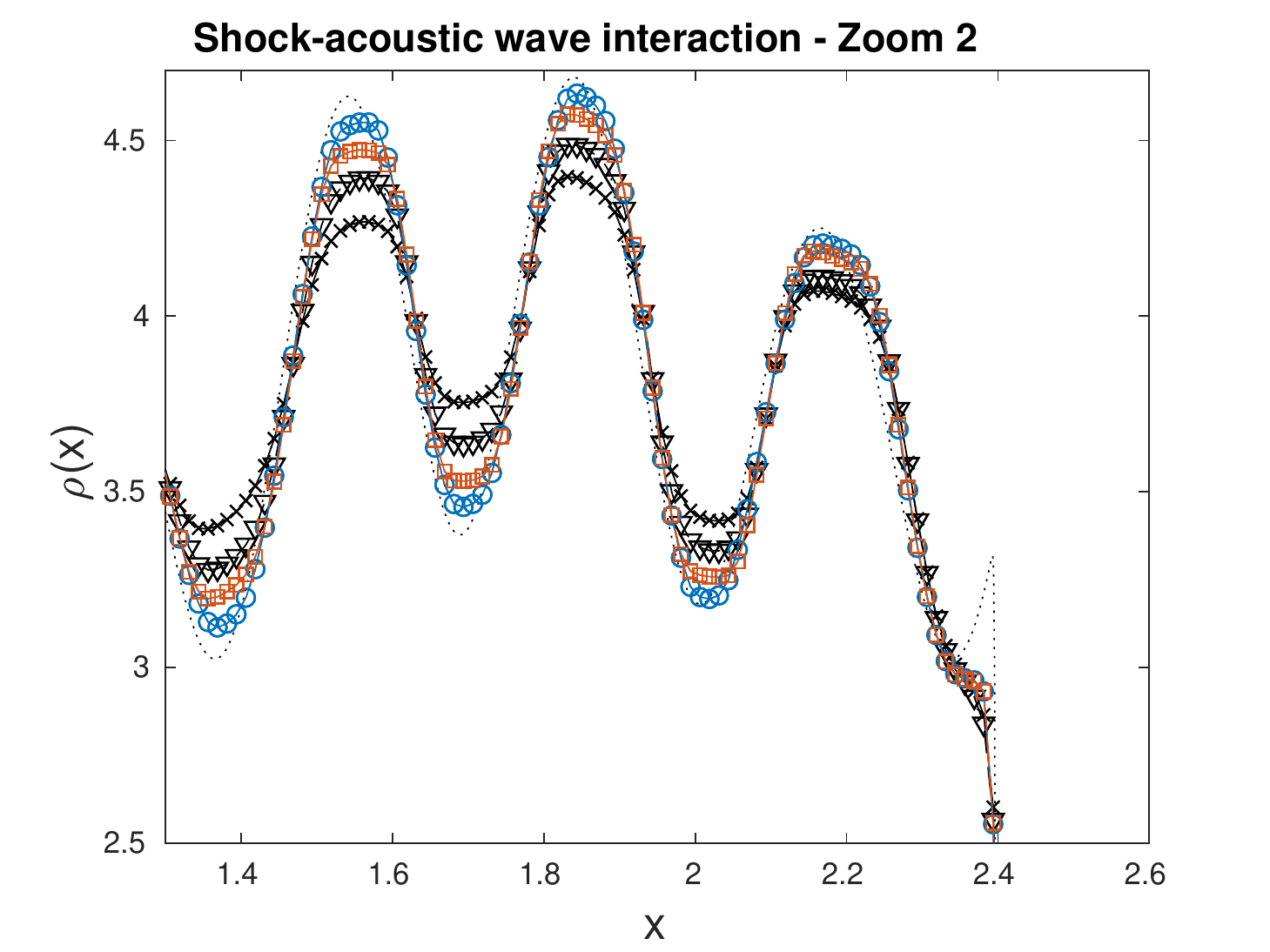}
	\caption{Numerical solution and zoom-in on two regions of the shock-acoustic wave interaction problem with schemes of order $3$ on $800$ cells.}
	\label{fig:shockacoustic3}
\end{figure}

Figure~\ref{fig:shockacoustic3} shows the numerical results computed with the $\CWENO3$, $\CWENOZDB3$ and $\CWENOZ3$ schemes on $800$ cells. 
For the $\CWENOZ3$ scheme, in view of the previous experiments, we consider $\elle=2$ with $\hat{m}=2,3$.
The bottom panels show the zoom-in of the solution on two regions of the computational domain. The reference solution (black dotted line) was generated using $8000$ cells and the fifth order $\CWENO$ scheme. 
We observe that $\CWENOZ3$, which employs the new improved weights introduced in Section~\ref{ssec:tau1d}, provides a better resolution, in particular with $\hat{m}=2$, of the turbulence region, which is characterized by the smooth high-frequency solution behind the main shock (bottom right panel of Figure~\ref{fig:shockacoustic3}). No extra oscillations at discontinuity is observed: the main shock and the shocklets are approximated better by the new $\CWENOZ3$ schemes, in particular with respect to $\CWENO3$, which is more diffusive. Moreover, note that, despite  the kink at the shock is under-resolved by the grid, no oscillations arise with any of the schemes.

\paragraph{Summary of the one-dimensional tests}
For \CWENOZ5, $\elle=1$ and $\hat{m}=3$ with the standard choice $I_0=I[\Popt]$ has performed well in all circumstances. Using $\hat{m}=4$ would yield even smaller spurious oscillations close to discontinuities, but in any case, the dependence of \CWENOZ5\ on the choice of $\elle$, $\hat{m}$ and $I_0$ proved to be weak.

On the other hand, \CWENOZ3\ exhibits a stronger dependence on the parameters. A sensible general choice is $\hat{m}=2$, which yields lower errors on smooth flows, coupled with $\elle=2$ and $I_0=I[P_0]$, which reduce the oscillations around discontinuities. However, we point out that $\hat{m}=3$ and $\elle=1$ produces significantly less spurious oscillations than the previous case and is comparable with \CWENOZDB3 in smooth parts: it could be a valid choice for problems where controlling spurious oscillations is more important than resolving the smooth parts on coarse grids.

\subsection{Two space dimensions}
\label{ssec:test2d}
In the following paragraphs, we consider test problems based on the two-dimensional system of Euler  equations for gas dynamics
\[ \partial_t \left( \begin{array}{c}
\rho \\ \rho u \\ \rho v \\ E
\end{array}\right) +
\partial_x \left( \begin{array}{c}
\rho u \\ \rho u^2 + p \\ \rho u v \\ u(E+p)
\end{array}\right) +
\partial_y \left( \begin{array}{c}
\rho v \\ \rho u v \\ \rho v^2 + p \\ v(E+p)
\end{array}\right)= 0,
\]
where $\rho$, $u$, $v$, $p$ and $E$ are the density, velocity in $x$ and $y$ direction, pressure and energy per unit mass. The thermodynamic closure is given by the equation of state
$
	E = \frac{p}{\gamma-1} + \frac12 \rho (u^2+v^2)
$,
where we take $\gamma=1.4$.

The scheme does not rely on dimensional splitting, but we make use of  the two-dimensional reconstruction procedure \CWENO3 of \cite{SCR:CWENOquadtree}, the novel \CWENOZ3  that employs the definition of $\hat{\tau}$ given in equation \eqref{eq:tau32d} with $t=1, u=2$, that is
\begin{equation}\label{eq:tau32d:ok}
\hat{\tau}_{3}= \vert  I_{NE}^{(1)} + I_{NW}^{(1)}+ I_{SE}^{(1)}+ I_{SW}^{(1)} -4 I_{0}^{(2)}  \vert
\end{equation}
and we also consider the alternative definition with $t=1,u=0$ that leads to
\begin{equation}\label{eq:tau32d:B}
\hat{\tau}_{3B}= \vert  I_{NE}^{(1)} - I_{NW}^{(1)} - I_{SE}^{(1)}+ I_{SW}^{(1)} \vert
\end{equation}
giving it the name  $\CWENOZ 3$(b). This latter does not make use the indicator of the optimal polynomial in the definition of $\hat{\tau}$. After the one-dimensional results, we consider here $\elle=2$, $\hat{m}=2$ and $I_0=I[\Popt]$.

For the computation of the numerical fluxes across the cell faces  2-point Gauss formulas are employed and their nodes dictate the position where the point-value reconstructions are computed, totalling 8 reconstructed values per cell. The rest of the numerical scheme is the straightforward generalization of the one-dimensional scheme. The implementation was carried on with the help of the PETSc libraries to handle parallelism and the simulations were run on a 24-cores node of the OCCAM cluster of the C3S centre of the Universit\`a di Torino ({\tt https://c3s.unito.it}).

\paragraph{Two-dimensional convergence test: the isentropic vortex}
The isentropic vortex problem is a test for the Euler equations in two space dimensions introduced by Shu in~\cite{Shu:97} and commonly used for testing the order of accuracy of a numerical scheme since an exact, smooth and analytic solution exists at all times.

The initial condition for the vortex problem is characterized by a uniform ambient flow with $T_\infty = 1.0$, more specifically
$
	(\rho_\infty,u_\infty,v_\infty,p_\infty) = (1.0,1.0,1.0,1.0)
$,
onto which the following isentropic perturbations are added in velocity and temperature:  
$$
	(\delta u, \delta v) = {\frac {\beta} {2 \pi}} \exp \left( {\frac {1-r^2} {2}} \right) (-y,x), \quad
	\delta T = - { \frac {(\gamma - 1 ) \beta^2} {8 \gamma \pi^2}} \exp \left( {1-r^2} \right),
$$
where $r=\sqrt{x^2+y^2}$. The initial state is thus given by
$$
	(\rho,u,v,p) = \left( \rho_\infty \left( \frac{T_\infty+ \delta T}{T_\infty} \right) ^{\frac{1}{\gamma-1}} , u_\infty+\delta u, v_\infty+\delta v, p_\infty \right).
$$
Here $\beta$ is the so-called strength of the vortex and is set to $\beta=5.0$.

The numerical tests are performed on the computational domain $[-5,5]\times [-5,5]$ with periodic boundary conditions. As a result of isentropy, the exact solution of this problem is simply characterized by a pure advection of the initial condition  with velocity $(u_\infty,v_\infty)$. At the final time $T=10$ the vortex is back to its original position and the final solution should match the initial one.
Since the solution is 
smooth, it should be simulated with optimal high accuracy and the limiting/stabilization procedure employed in the scheme should not have any effect.

\begin{table}
	\begin{center}
		\caption{The accuracy for the isentropic vortex test with schemes of order $3$.}
		\label{tab:vortex}
		\tiny
		\setlength{\tabcolsep}{3pt}
		\begin{tabular}{|r|rr|rr|rr|rr|rr|rr|} \hline
			& \multicolumn{6}{c|}{Density variable} & \multicolumn{6}{c|}{Energy variable}
			\\
			\hline
			& \multicolumn{2}{c|}{$\CWENO 3$} & \multicolumn{2}{c|}{$\CWENOZ 3$} & \multicolumn{2}{c|}{$\CWENOZ 3b$}
			& \multicolumn{2}{c|}{$\CWENO 3$} & \multicolumn{2}{c|}{$\CWENOZ 3$} & \multicolumn{2}{c|}{$\CWENOZ 3b$}
			\\
			cells & error & rate & error & rate & error & rate
			& error & rate & error & rate & error & rate
			\\
			\hline
			50   & 2.97e-01 &    - & 3.28e-01 &    - & 3.43e-01 &    - & 1.93e-00 &    - & 1.83e-00 &    - & 1.83e-00 &    -\\
			100  & 6.01e-02 & 2.31 & 6.41e-02 & 2.36 & 6.43e-02 & 2.41 & 3.23e-01 & 2.58 & 3.08e-01 & 2.57 & 3.08e-01 & 2.57\\
			200  & 9.15e-03 & 2.72 & 9.03e-03 & 2.83 & 9.04e-03 & 2.83 & 4.46e-02 & 2.86 & 4.24e-02 & 2.86 & 4.24e-02 & 2.86\\
			400  & 1.25e-03 & 2.87 & 1.15e-03 & 2.97 & 1.15e-03 & 2.97 & 5.73e-03 & 2.96 & 5.39e-03 & 2.97 & 5.39e-03 & 2.97\\
			800  & 1.61e-04 & 2.96 & 1.44e-04 & 3.00 & 1.44e-04 & 3.00 & 7.28e-04 & 2.98 & 6.82e-04 & 2.98 & 6.82e-04 & 2.98\\ 
			1600 & 2.02e-05 & 2.99 & 1.80e-05 & 3.00 & 1.80e-05 & 3.00 & 9.70e-05 & 2.91 & 9.12e-05 & 2.90 & 9.12e-05 & 2.90\\
			\hline
		\end{tabular}
	\end{center}
\end{table}

In Table~\ref{tab:vortex} we show the errors and the convergence rates in density and in total energy for the isentropic vortex test with the two-dimensional schemes of order $3$. We observe that all the schemes reach the theoretical order of convergence. Both $\CWENOZ$ reconstructions yield lower errors and better convergence rates than $\CWENO$. In particular, the $\CWENOZ 3$ scheme seems slightly better than the $\CWENOZ 3$(b).

However, in the tests involving strong shocks, we have observed that the indicator employed in \CWENOZ3(b), that does not take into account the central interpolating polynomial in the expression for $\hat{\tau}$, may lead to a breakdown of the simulations, in particular in the forward-facing step and in the double Mach reflection problems. For this reason, only the results for \CWENO3 and \CWENOZ3 are presented in the final part of the paper.

\paragraph{The forward facing step problem}
This problem was proposed by Emery~\cite{Emery:68} and Woodward and Colella~\cite{WoodwardColella:84}. It is characterized by a Mach 3 flow entering in a wind tunnel from the left. The tunnel has a reduction of size due to a step, which opposes the direction of the flow, emanating shock waves that later are reflected back by the top wall. At initial time the tunnel is filled with a uniform gas in the state
$$
	(\rho,u,v,p) = (\gamma, 3, 0, 1). 	
$$
 The rectangular computational domain is $[0,3] \times [0,1] \backslash [0.6,3] \times [0,0.2]$ and we use reflective wall boundary conditions on the upper and lower boundaries of the domain. Moreover, at the inflow boundary (left) we maintain the initial Mach 3 flow, while outflow boundary conditions are employed on the right end of the domain. The challenges of this problem are the stability in the initial boundary layer on the step and of  the flow around the corner, and the emergence of shock waves bouncing on the walls and interacting among themselves. In particular, the wave emerging from the triple point in the upper region is Rayleigh-Taylor unstable, but diffusive numerical schemes often smooth out the instability.

For the numerical simulations, the final time is set to $T=4$. First, we compare the solutions computed on a grid of $1920\times640$ cells (1M degrees of freedom) by the $\CWENO$ and $\CWENOZ$ schemes of order $3$. We point out that for both schemes no special treatment was needed at the corner of the step which is the center of a rarefaction fan. 

\begin{figure}
	\centering
	\includegraphics[width=\linewidth]{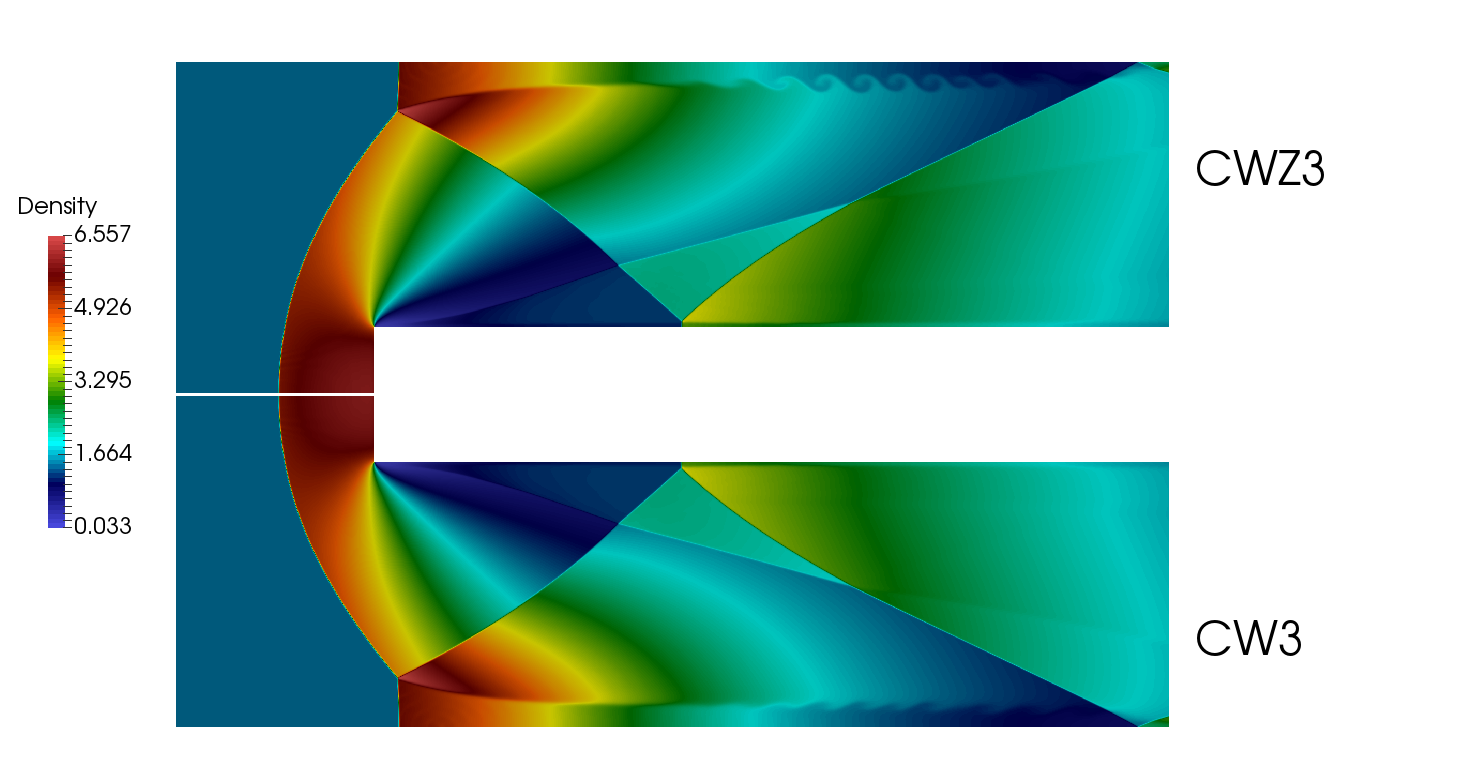}
	\caption{Numerical of the forward facing step problem at time $t=2.4$ with the $\CWENOZ$ (top) and $\CWENO$ (bottom) schemes of order $3$.}
	\label{fig:ffscompare:early}
\end{figure}

In Figure~\ref{fig:ffscompare:early} we plot the solutions at time $t=2.6$, when the contact discontinuity that emerges from the triple point has just formed. Here it is evident that the $\CWENOZ 3$ scheme can compute correctly the Rayleigh-Taylor instability of the contact, that appears almost stable with the $\CWENO 3$ scheme at this resolution. At later times, the instability diffuses out also with the \CWENOZ 3\ scheme.

\begin{figure}
	\centering
	\includegraphics[width=\linewidth]{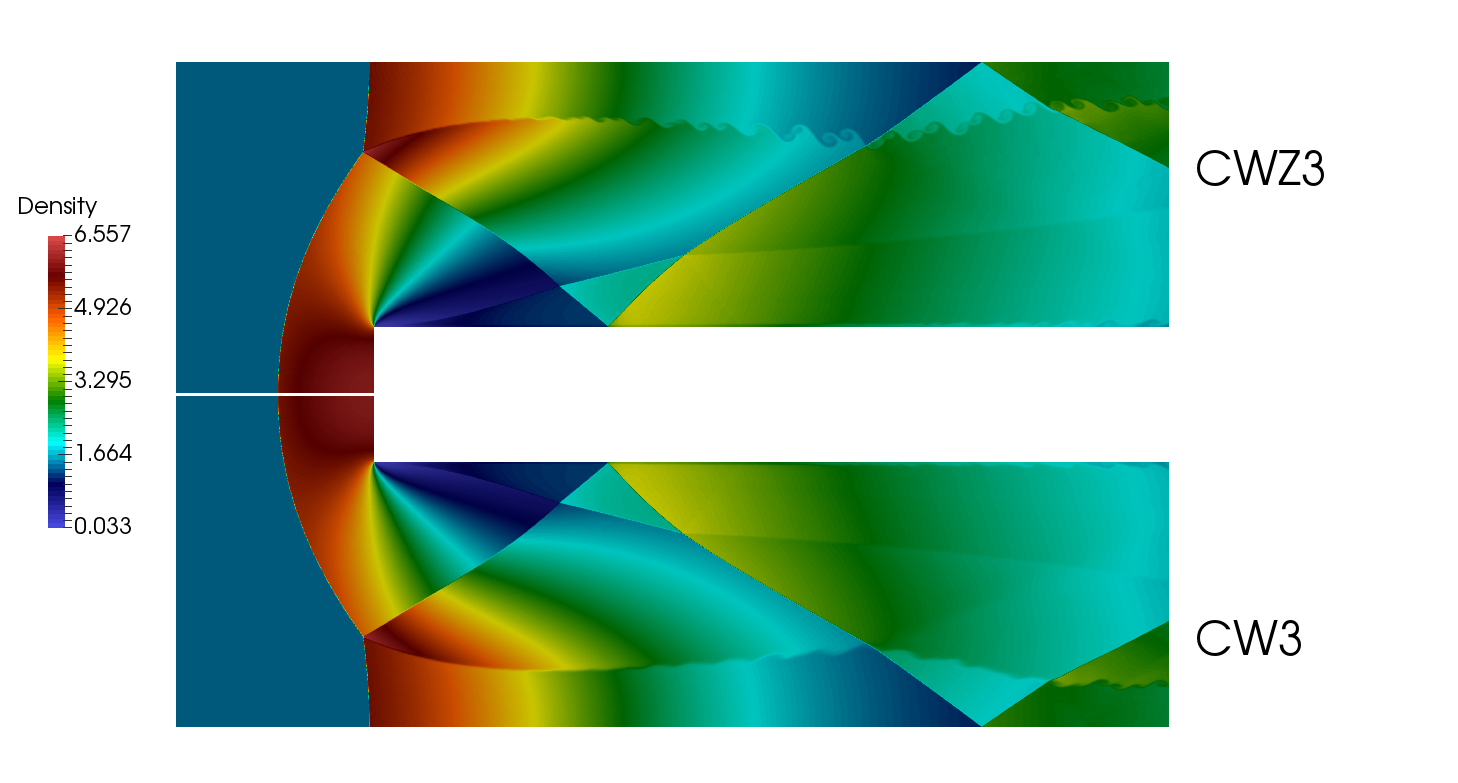}
	\caption{Numerical of the forward facing step problem at final time $T=4$ with the $\CWENOZ$ (top) and $\CWENO$ (bottom) schemes of order $3$ on a grid of $3480\times1280$ cells.}
	\label{fig:ffscompare:late}
\end{figure}

Next, in Figure~\ref{fig:ffscompare:late} we compare the solutions of the third order schemes at final time $T=4$ using an even finer grid ($3480\times1280$ cells).
It is evident that the curly instability patterns around the contact are maintained until final time by the $\CWENOZ 3$ scheme, while the $\CWENO 3$ scheme has completely diffused them even on this finer grid with 4M degrees of freedom.

\paragraph{Double Mach reflection problem}
The double Mach reflection problem of a strong shock was originally proposed by Woodward and Colella~\cite{WoodwardColella:84}. The problem is characterized by a Mach $10$  shock, which is incident on a ramp having an angle of $30^\circ$ with the $x$ axis. As in~\cite{WoodwardColella:84}, a suitable rotation is considered so that the ramp is aligned with the $x$ axis and therefore the shock forms an angle of $60^\circ$ with it. The initial conditions in front of and after the shock wave are given by
$$
	(\rho, u, v, p) = \begin{cases}
	(8.0, 8.25, 0.0, 116.5), & x'<1/6, \\
    (1.4, 0.0, 0.0, 1.0), & x'\geq 1/6, 
	\end{cases}
$$
where $x'$ is the coordinate along the shock direction in the rotated coordinate system. The rectangular computational domain is $[0,3.5] \times [0,1]$. Reflecting wall boundary conditions are prescribed on the bottom and the exact solution of an isolated oblique Mach $10$ shock wave is imposed on the upper boundary. Inflow and outflow boundary conditions are set on the left and the right sides. As the shock moves, it hits on the ramp and a complex shock reflection structure forms. This test is challenging due to the contemporaneous presence of strong waves, very weak ones and of complex smooth features in the so-called recirculation zone.

\begin{figure}
	\centering
	\includegraphics[height=4cm]{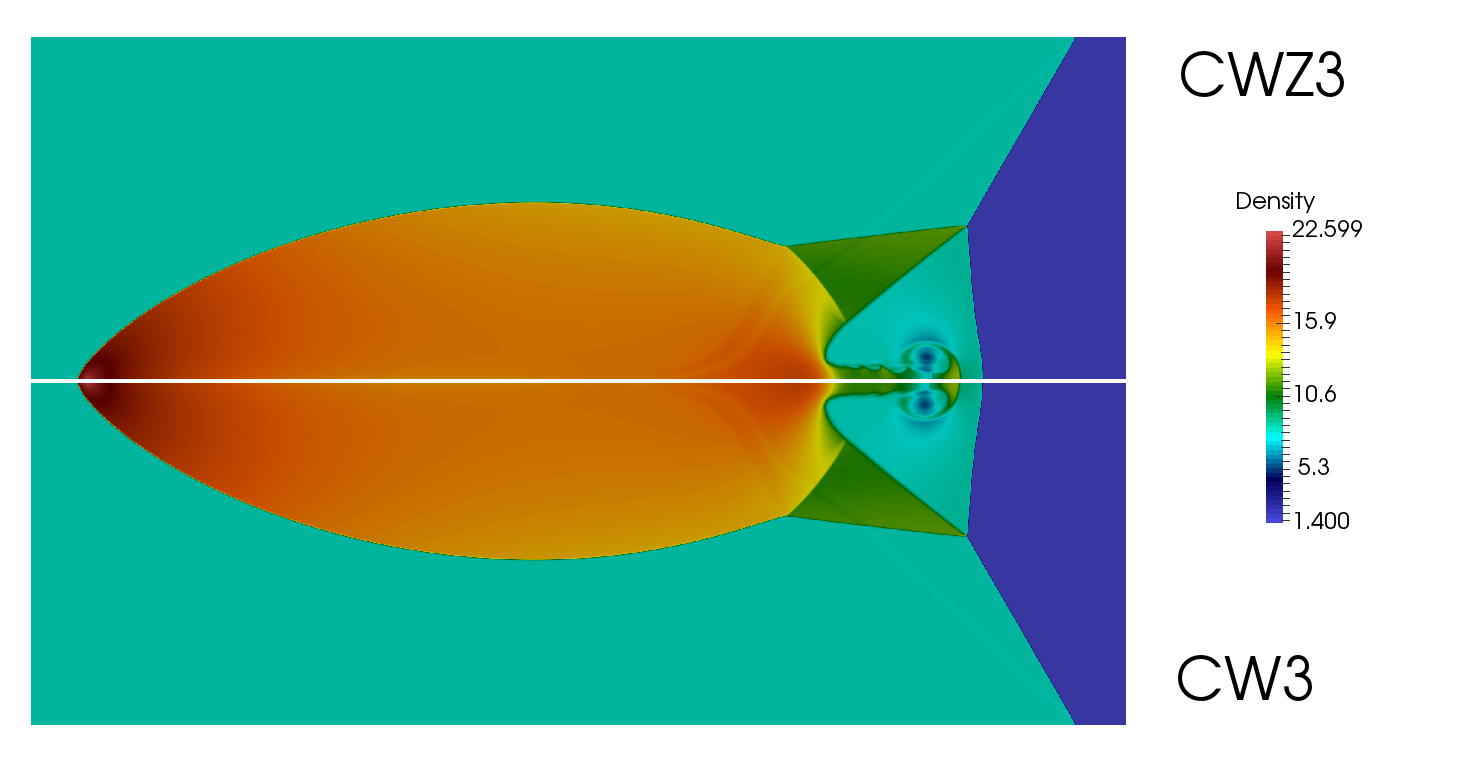}
    \includegraphics[height=4cm]{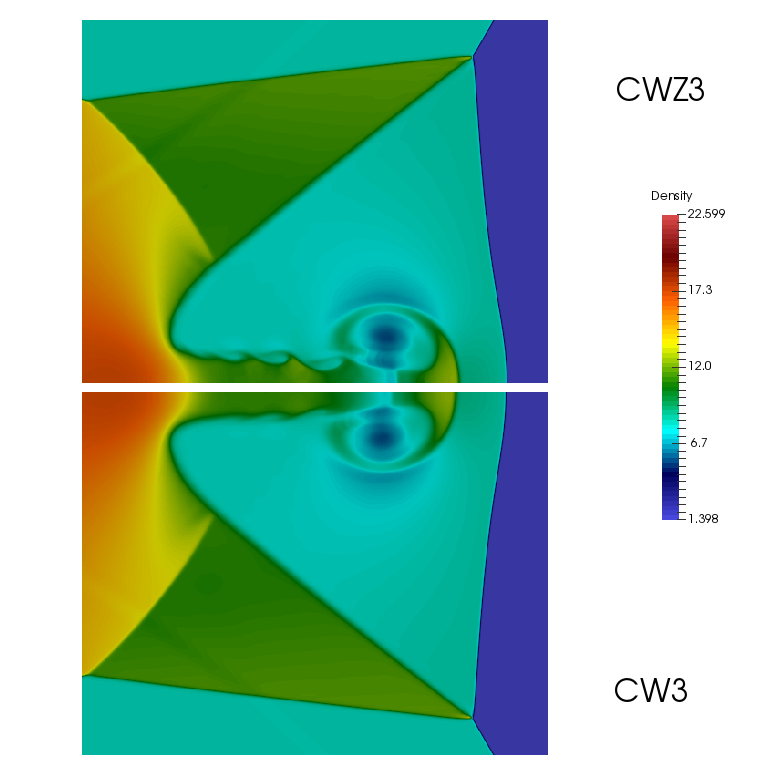}
    \caption{Numerical solution of the double Mach reflection problem with the $\CWENOZ$ (top) and $\CWENO$ (bottom) schemes of order $3$ on a grid of $2560\times 800$ cells.}
	\label{fig:dmr2560}
\end{figure}

The problem was run on a mesh of $2560 \times 800$ cells (2M degrees of freedom, Fig. \ref{fig:dmr2560}) and on a finer mesh of $5120 \time 1600$ cells (8M degrees of freedom, Fig.~\ref{fig:dmr5120}) up to final time $T=0.2$ when the front shock should be close to the right edge of the domain. 
The exact solution of this problem is not available but it is known that the discontinuity produces Kelvin-Helmholtz instabilities when high-resolution and sufficiently non-dissipative schemes are employed for the simulation. For this, both figures include also a zoom-in of the solution in the recirculation zone. 

We observe that the $\CWENO 3$ scheme is not able to reproduce the Kelvin-Helmholtz phenomena on the coarser grid, while the $\CWENOZ 3$ shows the instabilities across the moving stem, see Figure~\ref{fig:dmr2560}. On the finer grid, both schemes show the instability of the solution, see Figure~\ref{fig:dmr5120}. This proves that the $\CWENOZ 3$ scheme, with the new improved weights proposed in Section~\ref{ssec:tau2d} has a better accuracy than the classical version $\CWENO 3$.

\begin{figure}
	\centering
	\includegraphics[height=4cm]{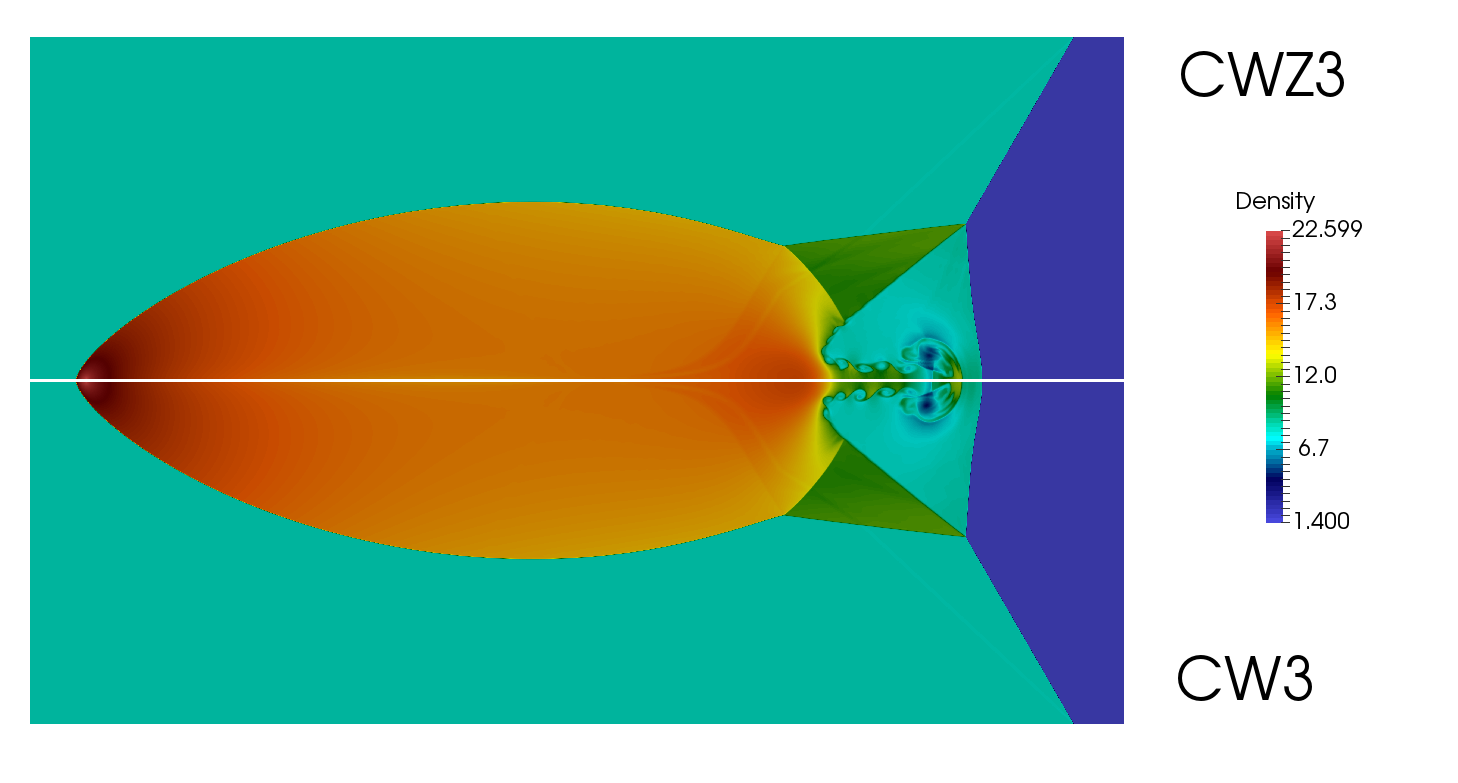}
	\includegraphics[height=4cm]{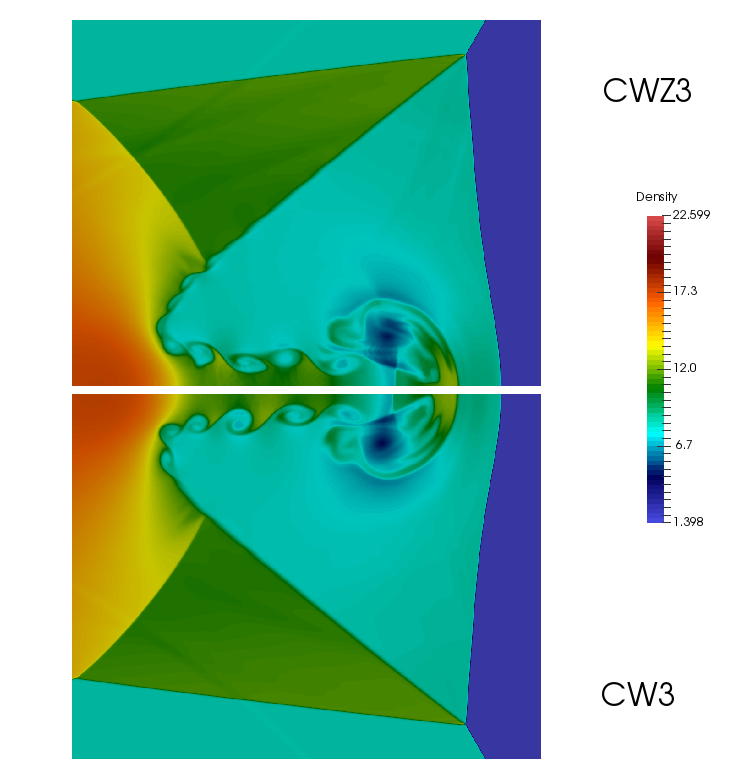}
    \caption{Numerical solution of the double Mach reflection problem with the $\CWENOZ$ (top) and $\CWENO$ (bottom) schemes of order $3$ on a grid of $5120\times 1600$ cells.}
	\label{fig:dmr5120}
\end{figure}

\paragraph{Shock-bubble interaction}
Here we consider the challenging problem in which a right-moving shock impinges on a standing bubble of gas at lower density, see \cite{CadaTorrilhon:09,SCR:CWENOquadtree}.
The computational domain is set to $\Omega=[-0.1, 1.6] \times [−0.5, 0.5]$ and the initial data are described considering three distinct areas: 
(A) the pre-shock region for $x < 0$, 
(B) the bubble  of center $(0.3, 0.0)$ and 
(C) radius $0.2$ and the post-shock region of all points with $x > 0$ that lie outside the bubble. 
The initial data are
\begin{center}
\begin{tabular}{l|llll}
&
\multicolumn{1}{c}{$\rho$} &
\multicolumn{1}{c}{$u$} &
\multicolumn{1}{c}{$v$} &
\multicolumn{1}{c}{$p$}
\\
\hline
A &
$11/3$ & $2.7136021011998722$ & $0.0$ & $10.0$
\\
B &
$0.1$ & $0.0$ & $0.0$ & $1.0$
\\
C &
$1.0$ & $0.0$ & $0.0$ & $1.0$
\\
\end{tabular}
\end{center}
Boundary conditions are of Dirichlet type on the left with the pre-shock data A, free-flow on the right and  reflecting on
$y=\pm0.5$. 
The symmetry in the $y$ variable permits to compute the numerical solution considering only in the upper half of the domain and symmetry boundary conditions at $y = 0$.
The final time is $T = 0.4$.

\begin{figure}
  \centering
  \includegraphics[width=\linewidth]{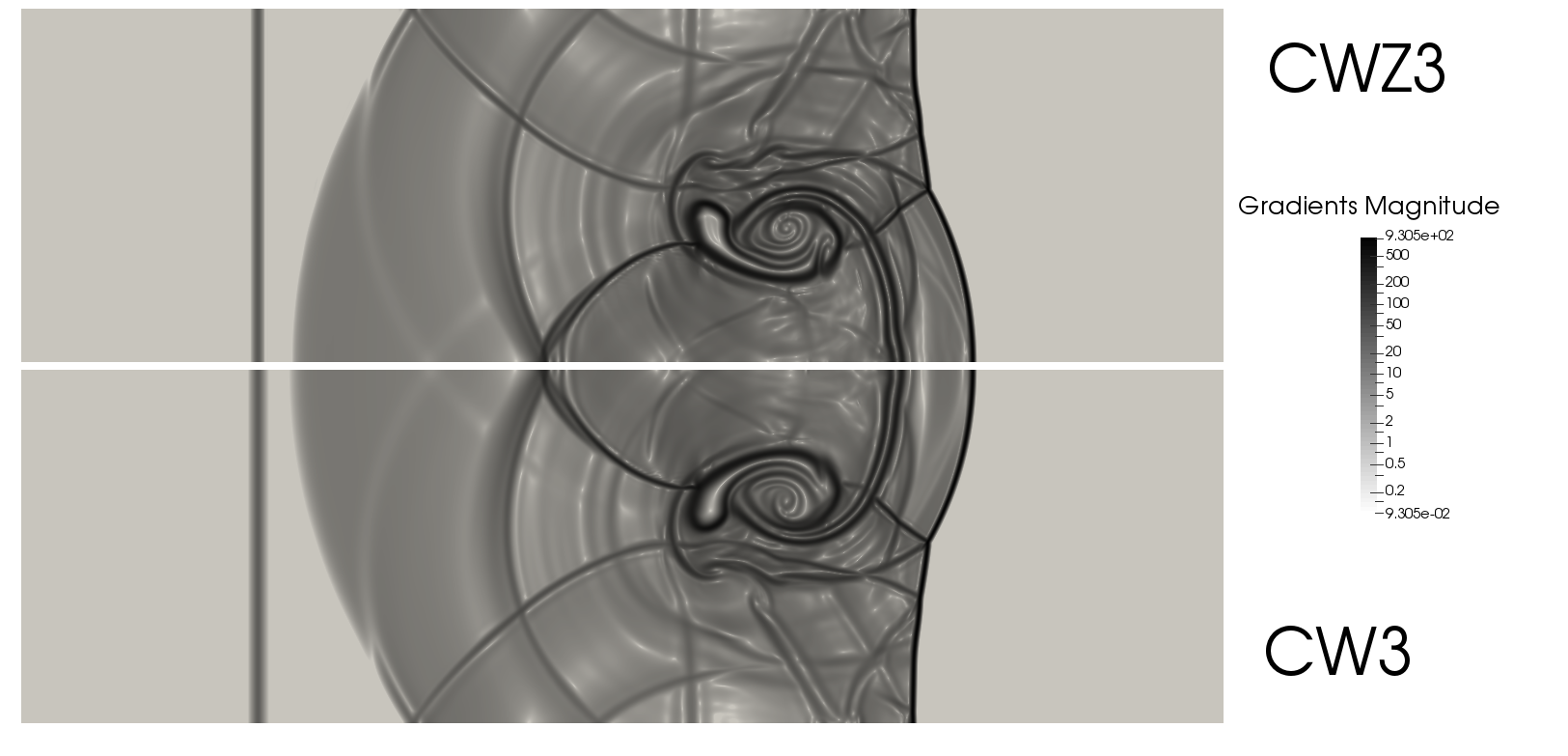}
  \\
  \includegraphics[width=\linewidth]{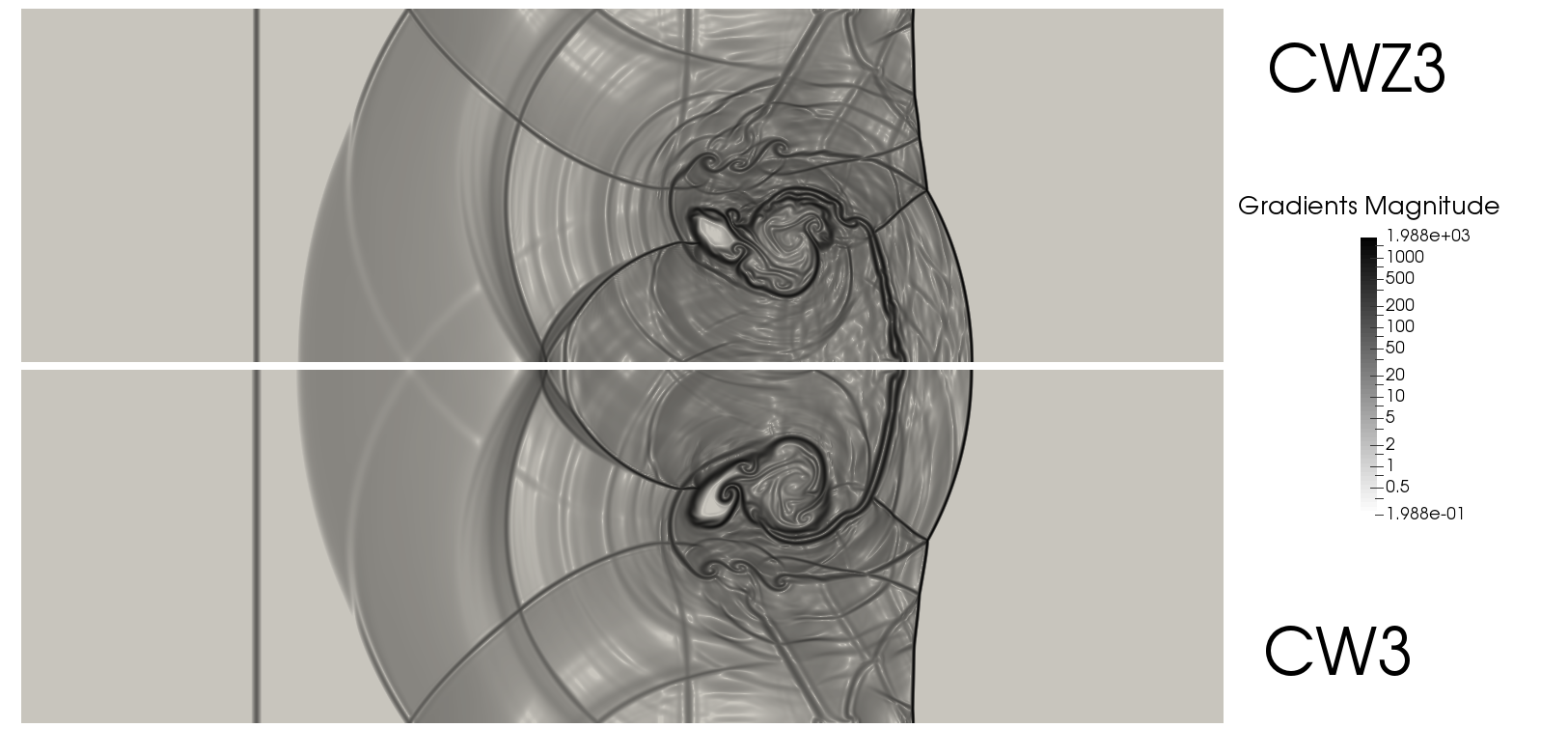}
  \caption{Numerical solution of the shock-bubble interaction problem at $t=0.32$ with the $\CWENOZ$ and $\CWENO$  schemes of order $3$ on a grid of $1360\times 400$ cells (top) and $2720\times800$ (bottom)}
	\label{fig:sb32}
\end{figure}

\begin{figure}
  \centering
  \includegraphics[width=\linewidth]{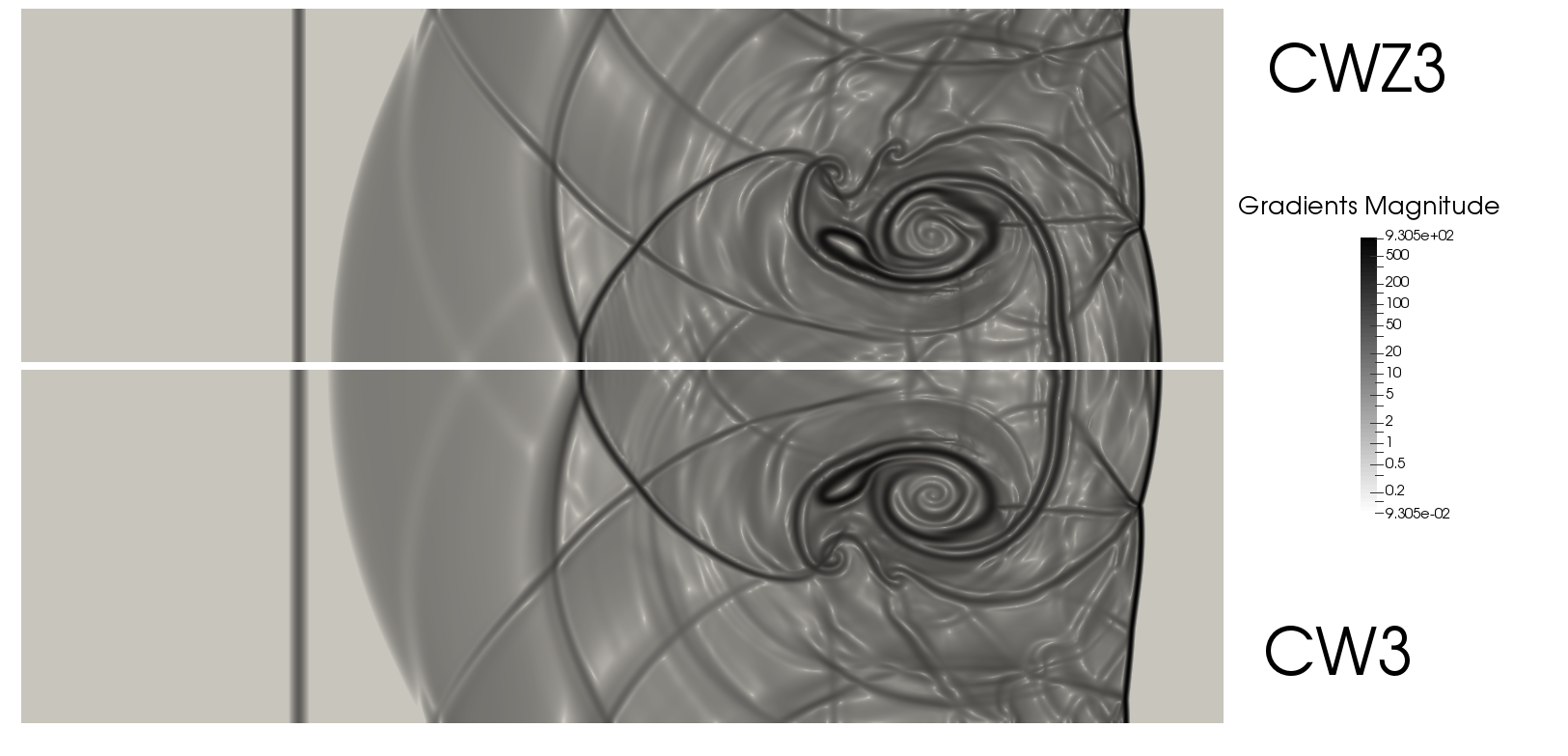}
  \\
  \includegraphics[width=\linewidth]{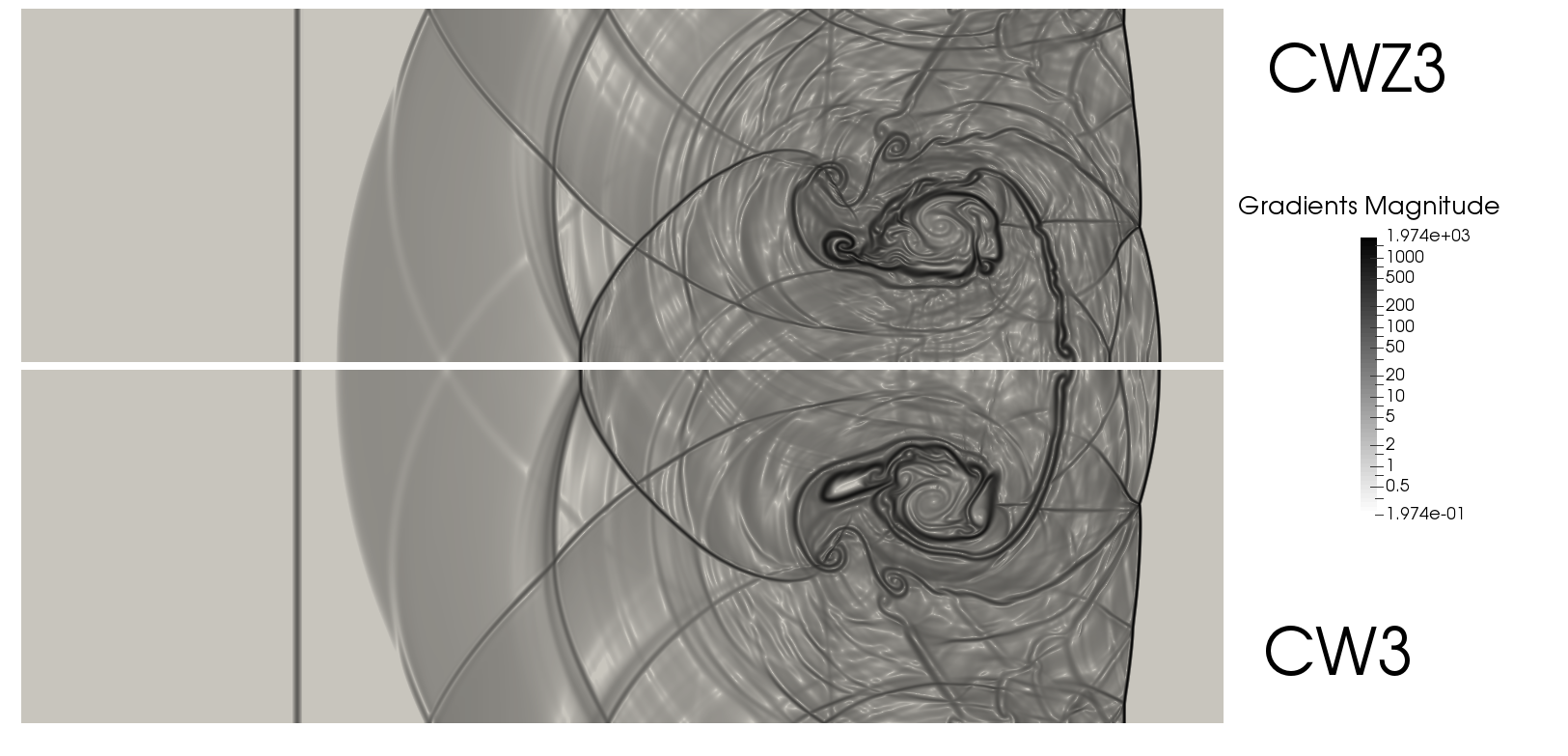}
  \caption{Numerical solution of the shock-bubble interaction problem at final time $T=0.4$ with the $\CWENOZ$ and $\CWENO$  schemes of order $3$ on a grid of $1360\times 400$ cells (top) and $2720\times800$ (bottom)}
	\label{fig:sb40}
\end{figure}

For this test we present the comparison of the solutions computed with \CWENO3 and \CWENOZ3 at different times and different resolutions. The more remarkable differences can be appreciated in the Schlieren plots, see Fig. \ref{fig:sb32} and \ref{fig:sb40}. When the shock impinges on the low-density bubble, it pushes forward and deforms it, while being refracted through it. The refracted shock then bounces back towards the bubble, creating a very complex interaction pattern. Moreover the bubble is known to be unstable and the \CWENOZ3 is able to show its instability already on the coarser grid, where small differences start to appear, while on the finer grid it computes a much more complex breaking and interaction pattern, due to its enhanced resolution.

\section{Conclusions}
\label{sec:concl}

In this paper we have analysed the optimal definition of the global smoothness indicator  that is employed in the computation of the \WENOZ-style non-linear weights in the setting of Central \WENO\ reconstructions.
The analysis is performed in multidimensions, avoiding the use of dimensional splitting, so that is be generalizable to less regular grid setups than the ones considered here.

To this end, in particular, we have proven asymptotic expansions of the Jiang-Shu smoothness indicators in $\R^n$ and derived a general result on the consistency order of \CWENO\ and \CWENOZ\ reconstructions. 
Next, we considered again the \CWENOZ\ reconstructions introduced in \cite{CPSV:coolweno} and, with the help of the abovementioned results, we defined the optimal $\hat{\tau}$ indicator for this setup. Finally, we have introduced a third order \CWENOZ\ reconstruction in two space dimensions that is based on the same stencils of the \CWENO\ one of \cite{SCR:CWENOquadtree,CWENOandaluz}.
Numerous one- and two-dimensional numerical tests confirmed the improved resolution of the new schemes.

The optimal global smoothness indicators defined for the reconstructions of this paper are constructed on the basis of results on the asymptotic expansions of the Jiang-Shu indicators that are quite general. We think that these results may prove very useful in the design of \CWENOZ\ reconstructions in the future, based on different choice of meshes, stencils and polynomial degrees.

\paragraph{Acknowledgements}
This work was supported by the
``National Group for Scientific Computation (GNCS-INDAM)''.
Some of the material used for this work has been discussed during the SHARK-FV 2018 conference.

\bibliographystyle{spmpsci}      
\bibliography{CSVCWENOZbiblio} 

\appendix
\section{Smoothness measuring matrices} 
\label{sec:appendix}

\subsection{Indicators for 1D \CWENO\ and \CWENOZ\ of order 3 and 5} 
Let us consider $q(x)= \sum_{i=0}^M a_i x^i $, with $M=4$ and $x\in\RR.$ By definition \eqref{eq:ind} and taking $\Omega_0 = \left[ -\frac{\Delta x}{2},\frac{\Delta x}{2} \right]$, the Jiang-Shu indicator is
\begin{equation} \label{eq:ind4} I[q]= a_1^2 \Delta x^2 + \left( \frac {13} 3 a_2^2 + \frac 1 2 a_1 a_3 \right) \Delta x^4 + \left( \frac {3129} {80} a_3^2 + \frac {21} 5 a_2 a_4 \right) \Delta x^6 + \frac {87617} {140} a_4^2  \Delta x^8.
\end{equation}
The matrices involved in Proposition \ref{prop:exA} are 

$$  \vec{U}= \left[
\begin{array}{@{\hspace{1pt}}c@{\hspace{2pt}}c@{\hspace{2pt}}c@{\hspace{2pt}}c@{\hspace{2pt}}c@{\hspace{1pt}}}
1 & 0 & \frac 1 {12} & 0 & \frac 1 {80}\\
0 & 1 & 0 & \frac 1 4 & 0 \\
0 & 0 & 2 & 0 & 1 \\
0 & 0 &0 & 6 & 0 \\
0 & 0 & 0 & 0 & 24
\end{array}%
\right], 
\; 
\vec{B} = \left[
\begin{array}{@{\hspace{1pt}}c@{\hspace{2pt}}c@{\hspace{2pt}}c@{\hspace{2pt}}c@{\hspace{2pt}}c@{\hspace{1pt}}}
1 & 0 & \frac 1 {12} & 0 & \frac 1 {80}\\
 0 & \frac 1 {12}& 0 & \frac 1 {80}  & 0\\
\frac 1 {12} & 0 & \frac 1 {80} & 0 & \frac 1 {448}\\
0 & \frac 1 {80} &0 & \frac 1 {448} & 0 \\
\frac 1 {80} & 0 & \frac 1 {448} & 0 & \frac 1 {2304}
\end{array}%
 \right], \;
\vec{A}= 
\left[
\begin{array}{@{\hspace{1pt}}c@{\hspace{2pt}}c@{\hspace{2pt}}c@{\hspace{2pt}}c@{\hspace{2pt}}c@{\hspace{1pt}}}
1 & 0 & 0 & 0 & 0\\
0 & \frac 1 {12} & 0 & -\frac 1 {720} & 0\\
0 & 0 & \frac 1 {720} & 0 & - \frac 1 {30240}\\
0 & -\frac 1 {720} &0 & \frac 1 {30240} & 0 \\
0 & 0 &  - \frac 1 {30240} & 0 & \frac 1 {1209600}
\end{array}%
\right] $$
and
$$ \vec{w} = \vec{U} \, \vec{D}\, \vec{a} = \left( \begin{array}{c}
a_0 \Delta x+ \frac 1 {12} a_2 \Delta x^3  + \frac 1 {80} a_4 \Delta x^5\\
a_1 \Delta x^2 + \frac 1 4 a_3 \Delta x^4 \\
2 a_2 \Delta x^3 + a_4 \Delta x^5 \\
6 a_3 \Delta x^4 \\
24 a_4 \Delta x^5
\end{array}
\right).
$$
%
The matrices involved in Proposition
 \ref{prop:ind} are
$$  \vec{Q}^1= \left[
\begin{array}{@{\hspace{1pt}}c@{\hspace{2.5pt}}c@{\hspace{2.5pt}}c@{\hspace{2.5pt}}c@{\hspace{2.5pt}}c@{\hspace{1pt}}}
0 & 1 &  0 &  0 &  0 \\
0 & 0 &  1 &  0 &  0 \\
0 & 0 &  0 &  1 &  0 \\
0 & 0 &  0 &  0 &  1 \\
0 & 0 &  0 &  0 &  0 \\
\end{array}
\right], \;
\vec{Q}^2= \left[
\begin{array}{@{\hspace{1pt}}c@{\hspace{2.5pt}}c@{\hspace{2.5pt}}c@{\hspace{2.5pt}}c@{\hspace{2.5pt}}c@{\hspace{1pt}}}
0 & 0 &  1 &  0 &  0 \\
0 & 0 &  0 &  1 &  0 \\
0 & 0 &  0 &  0 &  1 \\
0 & 0 &  0 &  0 &  0 \\
0 & 0 &  0 &  0 &  0 \\
\end{array}
\right], 
\; 
\vec{Q}^3= \left[
\begin{array}{@{\hspace{1pt}}c@{\hspace{2.5pt}}c@{\hspace{2.5pt}}c@{\hspace{2.5pt}}c@{\hspace{2.5pt}}c@{\hspace{1pt}}}
0 & 0 &  0 & 1 &0 \\
0 & 0 & 0 & 0 & 1 \\
0 & 0 & 0 & 0 &0 \\
0 & 0 &0 & 0 &0 \\
0 & 0 &0 & 0 &0
\end{array}
\right]
,\;
\vec{Q}^4= \left[
\begin{array}{@{\hspace{1pt}}c@{\hspace{2.5pt}}c@{\hspace{2.5pt}}c@{\hspace{2.5pt}}c@{\hspace{2.5pt}}c@{\hspace{1pt}}}
0 & 0 &  0 & 0 &1  \\
0 & 0 & 0 & 0 & 0\\
0 & 0 & 0 & 0 &0 \\
0 & 0 &0 & 0 &0 \\
0 & 0 &0 & 0 &0
\end{array}
\right]
$$
and then
$$
\vec{C}  = \sum_{ \alpha  =1}^4  (\vec{Q}^{\alpha})^T \vec{A} \vec{Q}^{\alpha} = 
\left[
\begin{array}{@{\hspace{1pt}}c@{\hspace{7pt}}c@{\hspace{2pt}}c@{\hspace{2pt}}c@{\hspace{2pt}}c@{\hspace{1pt}}}
0 & 0 &  0  & 0 & 0\\
0 & 1 & 0 & 0 & 0\\
0 & 0 & \frac {13} {12} & 0 & - \frac 1 {720} \\
0 & 0 &0 & \frac {781} {720} & 0 \\
0 & 0 & - \frac 1 {720} & 0 & \frac {32803}  {30240}
\end{array}%
\right]
$$
and, introducing $\vec{v}= \frac 1 {\Delta x}\vec{w}$, 
we obtain \eqref{eq:ind4} as 
$ I[q]= \langle \vec{v},  \vec{C} \, \vec{v} \rangle $.

From these matrices, we can obtain the indicators of all the polynomials of degrees $M \le 4$. 
Infact, for example, a polynomial of degree $M=2$, corresponds to assuming $a_4=a_3=0$ in the coefficients vector $\vec{a}$. Then, the corresponding indicator is obtained by using $3 \times 3$ left upper submatrices of the above ones. In this way  we are able to compute all
the indicators of polynomials involved in 1D \CWENO and \CWENOZ of order 3 and 5. 

\subsection{Indicators for 2D \CWENO\ and \CWENOZ\ of order 3}
Let us consider $q(\vec{x})=\sum_{\vert \aalpha \vert \le 2} a_{\aalpha}
{\vec{x}}^{{\aalpha}}$, with $M=2$ and $\vec{x}\in\RR^2$. For $ \boldsymbol{\Delta} \vec{x}= (h,k)$, 
taking $\Omega_0 = \left[ -\frac{h}{2},\frac{h}{2} \right] \times \left[ -\frac{k}{2},\frac{k}{2} \right] $, the Jiang-Shu indicator is
\begin{equation} \label{eq:ind42D} 
I[q] 
= 
\sum_{\elle=1}^2 \! \sum_{\vert \aalpha \vert = \elle} 
\!\!
\boldsymbol{\Delta} \vec{x}^{2 \aalpha -\auno} 
\!\!\!
\int_{\Omega}
\!
 (\partial^{\elle}_{\aalpha} q (\vec{x}))^2 \dx 
= a_{10}^2 h^2 + a_{01}^2 k^2 + \frac {13} 3 a_{20}^2 h^4 + \frac 7 6 a_{11}^2 h^2 k^2 + \frac {13} 3 a_{02}^2 k^4.
\end{equation}
The matrices involved in Proposition
\ref{prop:exA} are 
$$  
\vec{U}= 
\left[
\begin{array}{@{\hspace{1pt}}c@{\hspace{2pt}}|@{\hspace{2pt}}c@{\hspace{4pt}}c@{\hspace{2pt}}|@{\hspace{1pt}}c@{\hspace{2pt}}c@{\hspace{2pt}}c@{\hspace{1pt}}}
1 & 0 & 0 & \frac 1 {12} & 0 &  \frac 1 {12} \\
\hline
0 & 1 & 0 & 0 & 0 & 0  \\
0 & 0 & 1 & 0 & 0 & 0  \\
\hline
0 & 0 & 0 & 2 & 0 & 0 \\
0 & 0 &0 &0 & 1 & 0 \\
0 & 0 &0 &0 & 0 & 2
\end{array}
\right], \;  
\vec{B}= \left[
\begin{array}{@{\hspace{1pt}}c@{\hspace{1pt}}|@{\hspace{1pt}}c@{\hspace{2pt}}c@{\hspace{1pt}}|@{\hspace{1pt}}c@{\hspace{2pt}}c@{\hspace{2pt}}c@{\hspace{1pt}}}
1 & 0 & 0 & \frac 1 {12} & 0 &  \frac 1 {12} \\
\hline
0 & \frac 1 {12} & 0 & 0 & 0 & 0  \\
0 & 0 & \frac 1 {12} & 0 & 0 & 0  \\
\hline
\frac 1 {12} & 0 & 0 & \frac 1 {80} & 0 & \frac 1 {144} \\
0 & 0 &0 &0 & \frac 1 {144} & 0 \\
\frac 1 {12} & 0 &0 &\frac 1 {144} & 0 & \frac 1 {80}
\end{array}
\right] 
,\;
\vec{A}= 
\left[
\begin{array}{@{\hspace{1pt}}c@{\hspace{1pt}}|@{\hspace{1pt}}c@{\hspace{2pt}}c@{\hspace{1pt}}|@{\hspace{1pt}}c@{\hspace{2pt}}c@{\hspace{2pt}}c@{\hspace{1pt}}}
1 & 0 & 0 & 0 & 0 &  0 \\
\hline
0 & \frac 1 {12} & 0 & 0 & 0 & 0  \\
0 & 0 & \frac 1 {12} & 0 & 0 & 0  \\
\hline
0 & 0 & 0 & \frac 1 {720} & 0 & 0 \\
0 & 0 &0 &0 & \frac 1 {144} & 0 \\
0 & 0 &0 &0 & 0 & \frac 1 {720}
\end{array}
\right].$$
The blocks in the matrices correspond to the terms of homogeneous degree.
Also
\begin{equation} \vec{w} = \vec{U}\,\vec{D}\, \vec{a} = \left( \begin{array}{c}
a_{00} h k+ \frac 1 {12} a_{20} h^3 k + \frac 1 {12} a_{02} h k^3 \\
\hline
a_{10} h^2 k \\
a_{01} h k^2 \\
\hline
2 a_{20} h^3 k \\
a_{11} h^2 k^2 \\
2 a_{02}h k^3
\end{array}
\right). \label{eq:vectorw}
\end{equation}
From Proposition \ref{prop:ind} we obtain the matrices $
\vec{Q}^{\abeta} $ of \eqref{eq:matrixQ}with $ \vert \abeta \vert =1, 2$ as
$$  
\vec{Q}^{(1,0)} = 
\left[
\begin{array}{@{\hspace{2pt}}c@{\hspace{2pt}}|@{\hspace{2pt}}c@{\hspace{4pt}}c@{\hspace{2pt}}|@{\hspace{2pt}}c@{\hspace{4pt}}c@{\hspace{4pt}}c@{\hspace{1pt}}}
0& 1 & 0 & 0 & 0 & 0 \\
\hline
0& 0 & 0 & 1 & 0 & 0 \\
0& 0 & 0 & 0 & 1 & 0 \\
\hline
0& 0 & 0 & 0 & 0 & 0 \\
0& 0 & 0 & 0 & 0 & 0 \\
0& 0 & 0 & 0 & 0 & 0 
\end{array} \right], 
\quad 
\vec{Q}^{(0,1)} = 
\left[
\begin{array}{@{\hspace{2pt}}c@{\hspace{2pt}}|@{\hspace{2pt}}c@{\hspace{4pt}}c@{\hspace{2pt}}|@{\hspace{2pt}}c@{\hspace{4pt}}c@{\hspace{4pt}}c@{\hspace{1pt}}}
0& 0 & 1 & 0 & 0 & 0 \\
\hline
0& 0 & 0 & 0 & 1 & 0 \\
0 &0 & 0 & 0 & 0 & 1 \\
\hline
0& 0 & 0 & 0 & 0 & 0 \\
0& 0 & 0 & 0 & 0 & 0 \\
0& 0 & 0 & 0 & 0 & 0 
\end{array} \right],
$$
$$
\vec{Q}^{(2,0)} = 
\left[
\begin{array}{@{\hspace{2pt}}c@{\hspace{2pt}}|@{\hspace{2pt}}c@{\hspace{4pt}}c@{\hspace{2pt}}|@{\hspace{2pt}}c@{\hspace{4pt}}c@{\hspace{4pt}}c@{\hspace{1pt}}}
0 & 0 & 0 & 1 & 0 & 0  \\
\hline 
0 & 0 & 0 & 0 & 0 & 0\\ 
0 & 0 & 0 & 0 & 0 &  0\\ 
\hline 
0 & 0 & 0 & 0 & 0 & 0\\ 
0 & 0 & 0 & 0 & 0 & 0\\ 
0 & 0 & 0 & 0 & 0 & 0
\end{array} \right], 
\quad
\vec{Q}^{(1,1)} = 
\left[
\begin{array}{@{\hspace{2pt}}c@{\hspace{2pt}}|@{\hspace{2pt}}c@{\hspace{4pt}}c@{\hspace{2pt}}|@{\hspace{2pt}}c@{\hspace{4pt}}c@{\hspace{4pt}}c@{\hspace{1pt}}}
0 & 0 & 0 & 0 & 1 & 0  \\
\hline 
0 & 0 & 0 & 0 & 0 & 0\\ 
0 & 0 & 0 & 0 & 0 &  0\\ 
\hline 
0 & 0 & 0 & 0 & 0 & 0\\ 
0 & 0 & 0 & 0 & 0 & 0\\ 
0 & 0 & 0 & 0 & 0 & 0
 \end{array}
 \right], \quad
 \vec{Q}^{(0,2)} = 
\left[
\begin{array}{@{\hspace{2pt}}c@{\hspace{2pt}}|@{\hspace{2pt}}c@{\hspace{4pt}}c@{\hspace{2pt}}|@{\hspace{2pt}}c@{\hspace{4pt}}c@{\hspace{4pt}}c@{\hspace{1pt}}}
0 & 0 & 0 & 0 & 0 & 1  \\
\hline 
0 & 0 & 0 & 0 & 0 & 0\\ 
0 & 0 & 0 & 0 & 0 &  0\\ 
\hline 
0 & 0 & 0 & 0 & 0 & 0\\ 
0 & 0 & 0 & 0 & 0 & 0\\ 
0 & 0 & 0 & 0 & 0 & 0
 \end{array}
\right].
$$
Finally, introducing $\vec{v}= \frac 1 {\boldsymbol{\Delta} \vec{x}^{\auno}} \vec{w} = \frac 1 {hk} \vec{w}$ we have 
\begin{equation} \vec{C}  = \sum_{\vert \aalpha \vert =1}^2  (\vec{Q}^{\aalpha})^T \vec{A} \vec{Q}^{\aalpha} = 
\left[
\begin{array}{@{\hspace{2pt}}c@{\hspace{2pt}}|@{\hspace{2pt}}c@{\hspace{4pt}}c@{\hspace{2pt}}|@{\hspace{2pt}}c@{\hspace{4pt}}c@{\hspace{4pt}}c@{\hspace{1pt}}}
0 & 0 &  0  & 0 & 0 & 0 \\
\hline
0 & 1 &  0  & 0 & 0 & 0 \\
0 & 0 &  1  & 0 & 0 & 0 \\
\hline
0 & 0 &  0  & \frac {13} {12} & 0 & 0 \\
0 & 0 &  0  & 0 & \frac 7 6 & 0 \\
0 & 0 &  0  & 0 & 0 & \frac {13} {12} \\
\end{array}%
\right] \label{eq:matrixC} \end{equation}
and consequently we obtain \eqref{eq:ind42D} as 
$ I[q]= \langle \vec{v},  \vec{C} \, \vec{v} \rangle $.

All the previous matrices are computed by using the order $(1,0), \, (0,1)$ for $\vert \aalpha \vert =1$  and  $(2,0), \,(1,1), \, (0,2)$ for $ \vert \aalpha \vert =2.$ 

We note that $I[q] = \langle \vec{v},  \vec{C}\,\vec{v}  \rangle $ with $q \in \Poly{1}_{2} $ is obtained by using the $2 \times 2$ left upper  block submatrices,
i.e. the  blocks related to $\vert \aalpha \vert =0,1. $

\end{document}